\documentclass[12pt,a4paper,reqno]{amsart}
\usepackage{amsmath,amssymb,graphics,epsfig,color,enumerate,mathrsfs,caption}
\usepackage{dsfont}  \usepackage{verbatim}
\newtheorem{Theorem}{Theorem}[section]

\newtheorem{TheoremA}{Theorem}
\newtheorem{Lemma}[Theorem]{Lemma}
\newtheorem{Proposition}[Theorem]{Proposition}
\newtheorem{Corollary}[Theorem]{Corollary}
\newtheorem{Remark}[Theorem]{Remark}

\newtheorem{Claim}[Theorem]{Claim}
\newtheorem{Definition}[Theorem]{Definition}
\newtheorem{iDefinition}[Theorem]{(i)-Definition}
\newtheorem{Warning}{Warning}[section]

\definecolor{yellow(ncs)}{rgb}{1.0, 0.7, 0.0}

\definecolor{light}{gray}{0.9}

\newcommand{\cB}{\ensuremath{\mathcal B}}

\newcommand{\cE}{\ensuremath{\mathcal E}}
\newcommand{\cF}{\ensuremath{\mathcal F}}
\newcommand{\cG}{\ensuremath{\mathcal G}}

\newcommand{\cI}{\ensuremath{\mathcal I}}

\newcommand{\cL}{\ensuremath{\mathcal L}}

\newcommand{\cP}{\ensuremath{\mathcal P}}

\newcommand{\cR}{\ensuremath{\mathcal R}}

\newcommand{\cV}{\ensuremath{\mathcal V}}

\newcommand{\bbB}{{\ensuremath{\mathbb B}} }

\newcommand{\bbE}{{\ensuremath{\mathbb E}} }

\newcommand{\bbG}{{\ensuremath{\mathbb G}} }

\newcommand{\bbN}{{\ensuremath{\mathbb N}} }

\newcommand{\bbP}{{\ensuremath{\mathbb P}} }
\newcommand{\bbQ}{{\ensuremath{\mathbb Q}} }
\newcommand{\bbR}{{\ensuremath{\mathbb R}} }

\newcommand{\bbV}{{\ensuremath{\mathbb V}} }

\newcommand{\bbZ}{{\ensuremath{\mathbb Z}} }

\newcommand{\lrgh}{\leftrightarrow}

\newcommand{\rgh}{\rightarrow}

\newcommand{\ezd}{\e \bbZ^d}

\let\a=\alpha \let\b=\beta   \let\d=\delta  \let\e=\varepsilon
 \let\g=\gamma       \let\l=\lambda
      \let\o=\omega      
  \let\s=\sigma    
  \let\z=\zeta
\let\D=\Delta   \let\G=\Gamma  \let\L=\Lambda 
\let\O=\Omega      
\newcommand{\rrr}{\textcolor{black}}
\newcommand{\rui}{\textcolor{black}}
\newcommand{\kui}{\textcolor{black}}

\newcommand{\be}{\begin{equation}}
\newcommand{\en}{\end{equation}}

\author[A.\ Faggionato]{Alessandra Faggionato}
\address{Alessandra Faggionato.
  Department of Mathematics, University  La Sapienza.
  P.le Aldo Moro 2, 00185 Rome, Italy}
\email{faggiona@mat.uniroma1.it}

\author[H.A. Mimun]{Hlafo Alfie Mimun}
\address{Hlafo Alfie  Mimun.  Department of Economy and Finance, LUISS Guido Carli. Viale Romania 32, 00197  Rome, Italy}
\email{hmimun@luiss.it}

\thanks{This work has been partially supported by the ERC Starting Grant 680275 MALIG and by the Grant  PRIN 20155PAWZB "Large Scale Random Structures"}

\title[]{Left-right crossings in the Miller-Abrahams random resistor network and in  generalized Boolean models}

\begin{document}
{\maketitle}

%

\begin{abstract} We consider  random graphs  $\cG$ built on a homogeneous  Poisson point process on $\bbR^d$, $d\geq 2$, with points   $x$ marked by i.i.d. random variables $E_x$.  Fixed a symmetric function $h(\cdot, \cdot)$, the vertexes of $\cG$ are given by points of the Poisson point process, while the edges   are given by pairs $\{x,y\}$ with $x\not =y$ and $|x-y|\leq h(E_x,E_y)$. We call $\cG$ \emph{Poisson $h$--generalized  Boolean model}, as one recovers the standard Poisson Boolean model by taking $h(a,b):=a+b$ and $E_x\geq 0$. Under general conditions, we show that in the supercritical phase the maximal number of vertex-disjoint left-right crossings in a box of size $n$ is lower bounded by $Cn^{d-1}$ apart from  an event of exponentially small probability. As special applications, when the marks are non-negative,    we consider the   Poisson Boolean model and its generalization to $h(a,b)=(a+b)^\g$ with $\g>0$,   the weight-dependent random connection models with max-kernel and with  min-kernel  and  the   graph obtained from the  Miller-Abrahams random resistor network in which  only  filaments  with  conductivity lower bounded by a fixed positive constant are kept.
 
\smallskip


\smallskip

\end{abstract}

\section{Introduction} 
We introduce   a random graph $\cG$, called \emph{Poisson $h$--generalized Boolean model},  with vertexes in $\bbR^d$ where
$d\geq 2$, 
  whose construction depends on a structural symmetric function $h(\cdot, \cdot)$ with real entries, a parameter $\l>0$ and  a probability measure $\nu$ on $\bbR$. 
Given a homogeneous Poisson point process  (PPP) $\xi$ on $\bbR^d$ with \rui{intensity}   $\l$, we mark points $x$ of $\xi$ by i.i.d. random variables $E_x$ with common distribution $\nu$. Then the vertexes of $\cG$ are the points in  $\xi$, while edges of $\cG$ are given by unordered  pairs of vertexes $\{x,y\}$ with $x\not =y$ and $|x-y|\leq h(E_x,E_y)$. \rui{When $E_x\in [1,+\infty)$ and $h(\cdot, \cdot)$ is non--decreasing in each entry, then $\cG$ belongs to the class of  \emph{weight-dependent random connection models} introduced in \cite{GHMM}}. 

\smallskip

When $\nu $ has support inside $[0,+\infty)$ and $h(a,b)=a+b$, one recovers the so called Poisson Boolean model \cite{MR}. Another relevant example is related to the  Miller-Abrahams (MA) random resistor network. This resistor network  has been introduced in \cite{MA} to study 
the anomalous conductivity at low temperature in amorphous materials as doped semiconductors,  in the regime of  Anderson localization and at  low density of impurities. It has been further investigated  in the physical literature (cf.  \cite{AHL}, 
\cite{POF}  and references therein), where percolation properties have been heuristically analyzed.
A fundamental target has been 
 to get a more robust derivation of  the so called  Mott's law, which  is a physical law predicting the  anomalous decay of electron conductivity  at low temperature (cf. \cite{Fhom2,F_final,FM,FSS,POF} and references therein).
 
 When built on a marked  PPP $\{(x,E_x)\,:x\in \xi\}$ as above,  the MA random resistor network has nodes given by points in $\xi$ and  electrical filaments connecting each pair of nodes. The  electrical conductivity of the filament between $x$ and $y$ is given by (cf. \cite[Eq. (3.7)]{AHL})
\be\label{condu}
c(x,y):=\exp\Big\{ - \frac{2}{\gamma} |x-y| -\frac{\b}{2} ( |E_x|+ |E_y|+ |E_x-E_y|) \Big\}\,,
\en
where $\b$ is the inverse temperature and $\g$ is the so--called  localization  length. The physically relevant distributions $\nu$ (for inorganic materials) are of the form $\nu_{\rm phys}(dE)\propto
\mathds{1} ( |E| \leq a_0) |E|^\a dE $ with  $\a\geq 0 $ and $a_0>0$. 

Note that the conductivity of the filaments is smaller  than $1$. Fixed $c_0\in (0,1)$, the resistor subnetwork $\cG$ given by the filaments with conductivity lower bounded by $c_0$ is  a Poisson $h$--generalized  Boolean model  
with $h(a,b):=-(\g/2) \ln c_0  -(\g \b/4) (|a|+|b|+|a-b|)$. We point out that, in the low temperature regime (i.e. when
$\b \uparrow \infty$), the conductivity of the MA random resistor network is mainly supported  by  the subnetwork $\cG$ associated to   a suitable constant $c_0=c_0(\b)$ (cf. \cite{AHL,F_final}).
Hence,  to shorten the terminology and  with a slight abuse, in the rest we call $\cG$ itself  ``MA random resistor network".  Moreover, without loss of generality, we take $\b=\g=2$.

\smallskip
For the supercritical Bernoulli site and bond percolations on $\bbZ^d$ with $d\geq 2$, it is known that the maximal number of
vertex-disjoint left-right  crossings of a box $[-L,L]^d$
  is lower bounded by $C L^{d-1}$ \rui{apart from}  an event with  exponentially small probability (cf. \cite[Thm.~(7.68), Lemma~(11.22)]{G} and \cite[Remark~(d)]{GM}). A similar result is proved for the supercritical  Poisson Boolean model with deterministic radius  in \cite{T}. Our main result is that, under  general suitable conditions,  
    the same behavior holds  for the $h$--generalized Boolean model (cf. Theorem \ref{teo1}).  This result  for  the MA random resistor network  is relevant when studying the  low--temperature conductivity in amorphous solids (cf. \cite{AHL,F_final}). 
\smallskip

 We comment now some technical aspects in the derivation of our contribution.
 To prove Theorem \ref{teo1} we  first show that  it is enough to derive  a  similar result (given by Theorem 
 \ref{teo2} in Section \ref{sec_discreto})
for  a suitable random graph $\bbG_+$ with vertexes in  $\ezd$, defined in terms of i.i.d. random variables parametrized by points in $\ezd$. 
The
 proof of Theorem \ref{teo2}  is then  inspired by the renormalization procedure developed  by Grimmett and Marstrand in \cite{GM} for site percolation on $\bbZ^d$ and  by a construction 
 presented by Tanemura in  \cite[Section 4]{T}. We recall that   in \cite{GM} it is proved  
 that the critical  probability of  a slab in $\bbZ^d$ converges to the critical probability  of  $\bbZ^d$  when the  thickness of the slab goes to $+\infty$.  
 
  We point out that the renormalization method developed in \cite{GM} 
does not apply verbatim to  our case. In particular the adaptation of   Lemma 6 in \cite{GM} to our setting presents several obstacles due to  spatial correlations in our model. 
A main novelty here is to    build, by a Grimmett-Marstrand-like renormalization procedure, an increasing family of  quasi-clusters in our graph $\bbG_+$. We use here the term  ``quasi-cluster''  since  usually these sets  are    not  connected in $\bbG_+$ and can present some cuts at suitable localized regions.  By 
  expressing  the PPP of \rui{intensity} $\l$ as  superposition of two independent PPP's with \rui{intensity} $\l -\d$ and $\d\ll 1 $,  respectively,  a quasi-cluster is built only by means of  points in the PPP with \rui{intensity} $\l-\d$.  On the other hand, we will show that, with high probability,  when superposing the PPP with \rui{intensity} $\d$ we will insert a family of points linked with the quasi-cluster, making the resulting set connected in 
  $\bbG_+$.   \rui{This construction relies  on the idea of ``sprinkling" going back to \cite{ACCFR} (see also \cite{G} and references therein)}.
  
We  remark that in \cite{MT}  Martineau and Tassion have extended, with some modifications and simplifications, the Grimmett-Marstrand renormalization scheme.  Nevertheless, we have followed here the  construction  in \cite{GM} since it  is more suited to be combined with Tanemura's algorithm, which prescribes step by step to build  clusters in $\bbG_+$ along the axes of $\bbR^d$ (while in \cite{MT} there is a good building direction, which is not explicit). On the other hand, the fundamental steps in our Grimmett-Marstrand-like renormalization scheme are slightly simplified w.r.t. the original ones in \cite{GM}.


As special \rui{applications} of Theorem \ref{teo1} we consider, \rui{for non-negative  marks, the
cases $h(a,b):=(a+b)^\g$ with  $\g>0$, $h(a,b):=\max(a,b)$, $h(a,b):=\min(a,b)$ and $h(a,b):= \z - (|a|+|b|+|a-b|)$ with $\z>0$  (see Corollaries \ref{cor1} and  \ref{cor2})}. 
 The first case, with $\g=1$, generalizes Tanemura's result to  the  Poisson Boolean model with random radius. \rui{The second and third  cases correspond respectively to  the min-kernel and the max-kernel random connection model (see \cite{GHMM} and references therein)}.
The \rui{fourth} case corresponds to the MA random resistor network with non-negative energy marks. Although this result does not cover the mark distributions $\nu_{\rm phys}$ mentioned above, it  is suited to distributions $\nu(dE)\propto
\mathds{1} ( 0\leq E \leq a_0) |E|^\a dE $ with  $\a\geq 0 $ and $a_0>0$, which have similar scaling properties to the physical ones $\nu_{\rm phys}$.  We stress  that these scaling properties  are relevant in the heuristic derivation of Mott's law as well as in its rigorous analysis \cite{F_phys,F_final}. The restriction to non-negative marks 
in the MA random resistor network comes from the fact that  the Grimmett-Marstrand method (as well as  its extension in \cite{MT})  
 relies on the FKG inequality. As \rui{one can check}, when considering  the MA random resistor network with marks having different signs,  the FKG inequality can fail.  
  
For the reader's convenience, an outline of the paper is provided in Section \ref{pigolano} after the presentation of models and main results.
Finally, we point out that percolation results  for the subcritical MA random resistor network have been obtained in \cite{FMim1}.

\section{Model and main results}\label{moda}
We introduce a class of random graphs built by means of a   symmetric  \emph{structural function} 
\be 
h:\D\times \D \to \bbR\,,\qquad \D\subset \bbR\,.
\en
To this aim we call $\O$  the space  of locally finite sets of marked  points  in $\bbR^d$, $d\geq 2$, with marks in $\D$. More precisely, a generic element   $\o\in \O$ has the form $\o= \{ (x, E_x)\,:\, x\in \xi\}$, where $\xi$ is a locally finite subset of $\bbR^d$   and $E_x \in \D $ for any $x\in \xi$ ($E_x$ is thought of as the mark of point $x$).  It is standard (cf. \cite{DV})  to define on $\O$ a distance such that the $\s$--algebra of Borel sets $\cB$ of $\O$
is generated by the sets $\{ |\o \cap A|=n\}$, $A$  varying among the Borel subsets of $\bbR^d$ and $n $  varying in $\bbN$. We assume that  $(\O,\cB)$ is endowed with  a probability measure $P$, thus defining  a marked simple  point process.  

\begin{Definition}[\rui{Graph $\cG[\o]$}]
To each $\o= \{ (x, E_x)\,:\, x\in \xi\}$ in $ \O$  we associate the unoriented graph $\cG[\o]=\bigl( \cV[\o], \cE[\o]\bigr)$ with vertex set $\cV[\o]:=\xi$ and edge set $\cE[\o]$ given by the unordered  pairs $\{x,y\}$ with $x\not =y$ in $\xi$ and
\be\label{pico}
|x-y| \leq h(E_x,E_y) \,. 
\en
We call $h$--generalized Boolean model the resulting random graph $\cG=\cG[\o]$ defined on $(\O,\cB,P)$.
\end{Definition}

   When $\D:=\bbR_+$ and  $h(a,b):=a+b$, one has indeed the so--called Boolean model \cite{MR}. As discussed in the Introduction, another relevant example is given by the MA  random resistor network with lower bounded conductances: in this case $\D=\bbR$ and, fixed the parameter $\z>0$, the structural function $h: \bbR\times \bbR\to \bbR$ is given by 
 \be \label{hMA}
 h(a,b):=\z-\bigl( |a|+|b|+|a-b|\bigr)\,. 
 \en

We focus here on  the left-right  crossings of the graph $\cG=\cG[\o]$:
 
\begin{Definition}[\rui{LR crossing and $\cR_L(\cG)$}] \label{def_LR}
Given  \rui{$L\in (0,+\infty)$}, a left-right (LR) crossing  of  the box $[-L,L]^d$ in the graph  $\cG$    is any  sequence of distinct points $x_1,x_2 , \dots, x_n \in \xi$ such that 
\begin{itemize}
\item $\{x_i ,x_{i+1}\} \in  \cE$ for all $i=1,2,\dots, n-1$;
\item $x_1 \in (-\infty,-L)\times [-L,L]^{d-1} $;
\item $x_2, x_3,\dots, x_{n-1}\in [-L,L]^d$;
\item  $x_n \in (L, +\infty )\times [-L,L]^{d-1} $.
\end{itemize}
We also define \rui{$\cR_L(\cG)$}  as the  maximal number of  vertex-disjoint LR crossings of $[-L,L]^d$ in $\cG$.
\end{Definition}

In what follows, given $\l>0$ and a probability measure $\nu $ with support contained in   $\D$, we consider the marked Poisson point process (PPP)  obtained by sampling $\xi$ according to   a homogeneous PPP  with \rui{intensity} $\l$ on $\bbR^d$ ($d\geq 2$)  and marking each point $x\in \xi $ independently with a random variable $E_x$ having distribution $\nu$ (conditioning  to $\xi$, 
 the marks $(E_x)_{x\in\xi}$ are i.i.d. random variables with distribution $\nu$).  The above marked point process is   called \emph{$\nu$-randomization}  of  the PPP with \rui{intensity} $\l$ \rui{(cf.~\cite[Chp.~12]{Kal})}.  The resulting random graph $\cG$, whose construction depends also on the structural function $h$, will be denoted by $\cG(h,\l)=\cG(h,\l)[\o]$ when necessary (note that $\nu$ is understood).

\smallskip

To state our main assumptions, we recall that, given a generic graph with vertexes in $\bbR^d$, one  says that 
it  percolates if it has an unbounded connected component. We also  fix, once and for all, a constant  \rui{$\l>0$} and a probability measure $\nu $ on $\D$. We write ${\rm supp}(\nu)$ for the support of $\nu$.

 \medskip

 {\bf Assumptions}: 
  \emph{
 \begin{itemize}
 \item[(A1)]  $P$ is the $\nu$--randomization of a PPP with \rui{intensity} $\l$ on $\bbR^d$, $d\geq 2$.
   \item[(A2)] There exist $\l_*<\l$ and $\ell _* >0$ such that $\cG(h-\ell_*, \l_*)$ percolates  a.s..
   \item[(A3)] $\sup \left\{ h(a,b)\,:\, a,b \in \text{supp}(\nu) \right\} \in (0,+\infty)$.
     \item[(A4)]  For any $\d>0$ there exists 
      a Borel subset $U_*(\d)\subset \D$ with $\nu(U_*(\d))>0$ such that 
  \kui{ \be\label{indigestione}
  \inf_{ a\in U_*(\d)  }  h(a,b)
   \geq \sup_{a'\in {\rm supp}(\nu)} h (a',b) -\d \,, \qquad \forall  b\in {\rm supp}(\nu)\,.
   \en}
    \item[(A5)]  As $a,b $ vary in ${\rm supp}(\nu)$,  $h(a,b)$ is weakly decreasing both in $a$ and in $b$ (shortly, $h \searrow$), or $h(a,b)$ is weakly increasing both in $a$ and in $b$ (shortly, $h \nearrow$).
    \end{itemize}
}

Let us comment our assumptions. 

 The definition of $h$ is relevant only for entries in ${\rm supp(\nu)}$, hence one could as well restrict to the case $\D={\rm supp}(\nu)$.
Since the $h$--generalized Boolean model presents spatial correlations by its own definition, Assumptions (A1) avoids  further spatial correlations inherited from the marked point process. 

By a simple coupling argument, (A2) implies that $\cG(h-\ell', \l')$ percolates a.s. for any $\ell'\in [0,\ell_*]$ and $\l'\in [\l_*, \l]$. Hence, (A2) assures some  form of ``stable supercriticality'' of the graph $\cG(h, \l)$.

 Due to \eqref{pico},  (A3) both excludes the  trivial case $h\leq 0$ on ${\rm supp}(\nu)$  (which would imply that 
$\cG(h,\l)$  has no edges) and guarantees    that the length of the edges of $\cG(h,\l)$ is a.s.  bounded by some 
deterministic  constant.    

We move to (A4). By definition of supremum and due to (A3), for any $\d>0$ and for any $b\in {\rm supp}(\nu) $ there exists $a\in {\rm supp}(\nu)$ such that  $h(a,b)\geq \sup_{a' \in {\rm supp}(\nu) }h(a,b)-\d$. Assumption (A4) enforces this free inequality requiring that   it is satisfies uniformly in $a$ varying in some subset $U_*(\d)$ with positive $\nu$--measure. For example, if  $h$ is continuous, ${\rm supp}(\nu)$ is bounded  and (A5) is satisfied, then (A4) is automatically satisfied. Indeed, if e.g. $h \nearrow$, then $
\sup_{a' \in {\rm supp}(\nu) }h(a',b)=h(A,b)$ with $A:=\max \bigl(  \text{supp}(\nu)\bigr)$ and the claim follows by uniform continuity.

We move to (A5). This assumption implies that we enlarge the graph  when reducing the marks if $h \searrow $   or increasing the marks if $h \nearrow$. Moreover, (A5)    guarantees the validity of the FKG inequality (cf. Section \ref{scremato}), which in general can fail.

\medskip

Our main result is the following one:
\begin{TheoremA} \label{teo1} 
Suppose that the \rui{intensity} \rui{$\l>0$} and the mark probability distribution $\nu$ satisfy the above Assumptions (A1),...,(A5). Then 
 there exist 
positive constants $c, c'$ such that  
\begin{equation}\label{stimetta}
P \left( \rui{\cR_L(\cG)} \geq c L^{d-1} \right) \geq 1- e^{- c' L^{d-1}} \,,
\end{equation}
for $L$ large enough, where $\cG= \cG(h,\l)$.
\end{TheoremA}

The proof of Theorem \ref{teo1} is localized as follows: by Proposition \ref{fragolino} to get Theorem \ref{teo1} it is enough to prove Theorem \ref{teo2}. By Proposition \ref{carletto} to get Theorem \ref{teo2}  it is enough to prove \eqref{problema2d} \rui{in Proposition \ref{carletto}}. The proof of \eqref{problema2d} \rui{in Proposition \ref{carletto}} is given in Section \ref{sec_ginepro}.

\begin{Remark} We point out that \eqref{stimetta} cannot hold for all $L>0$, but fixed $L_0>0$ one can play with the constants $c,c'>0$ to extend \eqref{stimetta} to all $L\geq L_0$. Let us explain this issue. Calling   $\ell_0$  the supremum in (A3), we get that all edges of $\cG$ have length at most $\ell_0$. The event in the l.h.s. of \eqref{stimetta} implies that $\cR_L(\cG)\geq 1$ and therefore the PPP must contain  some point both  in $[-L-\ell_0,-L) \times [-L,L]^{d-1}$ and in $(L,L+\ell_0] \times [-L,L]^{d-1}$.  Hence the probability in \eqref{stimetta} is upper bounded  by $( 1- e^{-\l \ell_0 (2L)^{d-1}})^2$, which is of order $L^{2d-2}$ as $L\downarrow 0$.  As $d\geq 2$ we conclude that \eqref{stimetta} cannot hold for $L$ too small.
On the other hand, fix $L_0>0$ and suppose that \eqref{stimetta} holds for all   $L\geq L_1$ for some $L_1>L_0$. Take now $L\in [L_0,L_1]$.
Our assumptions imply that there exists a set $W\subset \bbR$ such that $\nu(W)>0$ and $r:=\inf _{a,b\in W} h(a,b)>0$ (cf. \eqref{mango}). Set $e_1:=(1,0,\dots, 0)$, $\ell:=\min(r,L_0)$  and let   $N$ be the minimal positive  integer such that $ N \ell /2 > L_1$.  Then   $\cR_L(\cG)\geq 1$ whenever each ball of radius $\ell/10$  and centered at   $k(\ell/2) e_1$    contains some   point $x $ of the Poisson process   with mark  $E_x\in W$, where $k$ varies from $-N-1$ to $N+1$. Note that this last event does not depend on $L$, and therefore the same holds for its positive probability. At this point, by playing with  $c,c'$ in \eqref{stimetta}, one can easily extend \eqref{stimetta} to all $L\geq L_0$.
\end{Remark}

We now discuss  some applications. 
\smallskip

Let $d\geq 2$, $\l >0$ and let $h:\bbR_+\times\bbR_+\rgh\bbR$ be \rui{one of the following  functions (for the first one $\g$ is a fixed positive constant): 
\be \label{hgamma}
 h(a,b):=(a+b)^\g\,,\qquad h(a,b):= \min (a,b)\,  \qquad 
h(a,b):=\max(a,b)\,.\en}
 Consider the graph $\cG=\cG(h,\l)$ built on the  $\nu$--randomization of a  PPP on $\bbR^d$ with \rui{intensity} $\l$. 
As proved in  Section \ref{bin_MA} (cf. Lemma \ref{silvestro}),  
if $\nu $ has bounded support and $\nu(\{0\})\not =1$, then  there exists a critical \rui{intensity} \rui{$\l_c$} \rui{in  $(0,+\infty)$} such that 
\begin{equation}\label{selva0}
P \bigl( \cG  \text{ percolates})= 
\begin{cases}
1 & \text{ if } \l>\l_c \,,\\
0 &\text{ if } \l < \l_c\,. 
\end{cases}
\end{equation}
\begin{Corollary}\label{cor1} Let $d\geq 2$  \rui{and let $h$ be  given by one of the functions  in \eqref{hgamma}}.
Consider the graph $\cG=\cG(h,\l)$ built on the  $\nu$--randomization of a  PPP on $\bbR^d$ with \rui{intensity} $\l>\l_c$, where the law $\nu$ has bounded support and $\nu(\{0\})\not =1$.  Then there exist $c,c'>0$ such that  \eqref{stimetta} is fulfilled for $L$ large  enough.
\end{Corollary}
We postpone the proof of the above corollary to Section \ref{bin_MA}.
\rui{Note  that  we recover the Poisson Boolean model \cite{MR}  when $h(a,b):=(a+b)^\g$ and   $\g=1$. We point  out that, according to the notation in \cite{GHMM},  the above graph  $\cG(h,\l)$  corresponds to the min-kernel (or max-kernel) weight--dependent random connection model by taking  $h(a,b)=\max(a,b)$ (respectively  $h(a,b)=\min(a,b)$) and by defining  the weight of the point $x$  as a suitable rescaling of our $E_x$.}

\smallskip

Let now $d\geq 2$ and $\z,\l>0$. 
Consider  the MA random resistor network $\cG=\cG(h,\l)$ with parameter $\z$ (i.e. with structural function $h$ given by \eqref{hMA}), built on the  $\nu$--randomization of a  PPP on $\bbR^d$  with \rui{intensity} $\l $.  \rui{When $\nu$ has bounded support it is trivial to check} that there exists a critical length  $\z_c$ \rui{in $(0,+\infty)$}  such that 
\begin{equation}\label{selva}
P \bigl( \cG  \text{ percolates})= 
\begin{cases}
1 & \text{ if } \z>\z_c \,,\\
0 &\text{ if } \z < \z_c\,. 
\end{cases}
\end{equation}
 \rui{Indeed, if $\nu$ has support in $[-A,A]$, then $\cG$ contains (is contained in) a random graph distributed as a Poisson Boolean model with intensity $\l$ and deterministic radius $\z/2-2A$  ($\z/2$ respectively)}.  
Equivalently, one could keep $\z$ fixed and play with the \rui{intensity} $\l$ getting a phase transition as in \eqref{selva0}  \rui{under suitable assumptions (see \cite[Prop.~2.2]{FMim1})}. For physical reasons it is more natural to vary $\z$ while keeping $\l$ fixed.
Then we have:
\begin{Corollary} \label{cor2}  Let $d\geq 2$ and $\l>0$. Consider the MA  random resistor network with parameter $\z>\z_c$  built on the $\nu$--randomization of a  PPP on $\bbR^d$  with \rui{intensity}  $\l$. 
Suppose that $\nu$ has  bounded support contained in $[0,+\infty)$ or in $(-\infty,0]$.  Then there exist 
positive constants $c, c'$ such that  \eqref{stimetta} is satisfied for $L$ large enough.
\end{Corollary}
We postpone the proof of the above corollary to Section \ref{bin_MA}. We point out that 
given $a,b \in \bbR$, it holds
\begin{equation}\label{semplice}
|a|+|b|+|a-b| =
\begin{cases}
2 \max \bigl( |a|, |b|\bigr) & \text{ if } a  b \geq 0 \,,\\
2|a-b| & \text{ if } a  b < 0 \,.
\end{cases}
\end{equation}
Hence, the structural function $h$ defined by \eqref{hMA} reads $h(a,b)= \z-2 \max(|a|,|b|)$ if  $a b\geq 0$, and $h(a,b)=\z- 2|a-b|$ if $ab <0$.  As a consequence, if the support of $\nu$ intersects both $(-\infty,0)$ and $(0,\infty)$, then Assumption (A5) fails.


\subsection{Outline of the paper}\label{pigolano}
The proof of Theorem \ref{teo1} consists of two main steps: 1) reduction to the analysis of  the LR crossings inside a 2d slice of  a suitable  graph  approximating $\cG$ and having  vertexes in a lattice, 2) combination of Tanemura's algorithm   in \cite[Section 4.1]{T} and   Grimmett-Marstrand renormalization  scheme  in \cite{GM} to perform the above reduced analysis.

For what concerns the first part, 
 in Section \ref{sec_discreto} we show that it is enough to prove the analogous of Theorem \ref{teo1} for a suitable graph $\bbG_+$ whose vertexes lie inside  $\e \bbZ^d$, with  $\e>0$ small enough. By a  standard argument (cf. \cite[Remark (d)]{GM}), we then show that it is indeed  enough
to  have a good control from below  of the number of vertex--disjoint LR crossings of $\bbG_+$ contained in a 2d slice (this control is analogous  to \eqref{stimetta} for $d=2$, cf. Eq. \eqref{problema2d}). 

We then move to the second part.
In Section \ref{sec_japan} we recall (with some extension) Tanemura's algorithm in \cite{T} to exhibit a maximal set of vertex-disjoint LR crossings for a generic  subgraph of $\bbZ^2$, where edges are given by pairs of nearest--neighbor points.  Under suitable conditions on the random subgraph of $\bbZ^2$, this algorithm allows to stochastically dominate  the maximal number of the  vertex-disjoint LR crossings by the analogous quantity for a site percolation. If the latter is supercritical, then one gets an estimate for the random subgraph of $\bbZ^2$ as \eqref{problema2d}. 
We then apply Tanemura's results by taking as random  subgraph of $\bbZ^2$ a suitable graph   built from $\bbG_+$ by a renormalization procedure similar to  the one developed by  Grimmett \& Mastrand in \cite{GM}. The combination of the two methods  to  conclude the proof of the bound \eqref{problema2d}, and hence of Theorem \ref{teo1}, is provided in Section \ref{sec_ginepro}, where the renormalization scheme and its main properties are only  roughly described. 
A detailed treatment of the renormalization scheme is given in Sections \ref{sec_RN_tools} and \ref{moto_GP}.

For the reader's convenience, in order to allow a better comprehension of the main structure of the proof of Theorem \ref{teo1}, we have postponed several technical proofs  in  the last sections \kui{(see Sections \ref{sio5}, \ref{trieste65}, \ref{patroclo}, \ref{puffo1}). Corollaries \ref{cor1} and \ref{cor2} are proven in Section \ref{bin_MA}.}

Finally, in   Appendix \ref{app_tanemura} we illustrate Tanemura's algorithm in a specific example and in Appendix \ref{app_locus} we collect some straightforward but cumbersome geometric bounds used in the renormalization scheme. In Appendix \ref{app_ultimatum}  we gather some technical arguments for the proof of \eqref{problema2d}.  
\section{Discretization and reduction to 2d slices}\label{sec_discreto}
\begin{Warning}\label{aaah} Without loss of generality  we  take $\D={\rm supp}(\nu)$ and assume that  $ \sup _{a,b\in \D} h(a,b)= 1$ (cf. (A3)). In particular, a.s. the length of the edges of $\cG$ will be  bounded by  $1$. 

\end{Warning}

In this section we show how to reduce the problem of estimating the probability $P \left( \rui{\cR_L(\cG)} \geq c L^{d-1} \right) $ in \eqref{stimetta} to a similar problem for a graph  $\bbG_+$
with vertexes contained in the rescaled lattice $\e \bbZ^d$. Afterwards, we show that, for  the second problem,  it is enough to  have
  a good control on the LR crossings of $\bbG_+$  contained in 2d slices.
  
  \smallskip
We need to introduce some notation since we will deal with several couplings:
\begin{itemize}
\item We write PPP($\rho$) for the Poisson point process with \rui{intensity} $\rho$.
\item We write PPP($\rho,\nu$) for the marked PPP obtained as $\nu$--randomization of a PPP($\rho$).
\item  Given a sequence of   i.i.d. random variables  $(X_n)_{n\geq 1} $  with law $\nu$ and, independently,   a Poisson random variable $N$ with parameter $\rho$, we 
 write $\cL(\rho,\nu)$ for the law of   $\inf \{X_1,X_2, \dots, X_N\}$ 
 if $h\searrow $ and  the law of $\sup \{X_1,X_2, \dots, X_N\}$ if $h\nearrow $.
  When $N=0$, the set $ \{X_1,X_2, \dots, X_N\}$ is given by $\emptyset$ and we use the convention that  $\inf\emptyset:=+\infty$ and $\sup \emptyset := -\infty$.
 \end{itemize}

Recall the constants $\l_*, \ell_*$  appearing in Assumption (A2)  and the set $U_*(\d)$ appearing in Assumption (A4).
\begin{Definition}[Parameters $\a,\e, K $, set $U_*$ and  boxes $R_z$'s] \label{vinello}
We fix  a constant    $\a>0$ small enough such that $10 \a\leq \ell_* $, $10 \a\leq 1$   and   $\sqrt{d}/\a\in \bbN_+$. We define $\e$ by 
$\e \sqrt{d}:= \a/100$ (note that $1/\e\in \bbN_+$). For each $z\in \ezd$ we set $R_z:= z+[0,\e)^d $. We fix a  positive integer $K$, very large. In Section \ref{francia}  we will explain  how to choose $K$. Finally, we set $U_*:= U_*(\a/2)$. 
\end{Definition}


\begin{Definition}[\rui{Fields $(A_z)$, $(T^{(j)}_z)$ and \kui{$(A_z^{\rm au})$}}] \label{cavallo} 
We introduce  the following  $K+1$  independent random  fields defined on a common probability space $(\Theta, \bbP)$: \begin{itemize}
\item Let $(A_z)_{z\in \ezd}$ be i.i.d. random variables with law  $\cL\bigl( \l_* \e^d , \nu\bigr)$.
\item Given $j\in \{1,2,\dots ,K\}$,
   let    $(T^{(j)}_z)_{z\in \ezd}$ be  i.i.d. random variables with law  $\cL\bigl( (\l-\l_*)\e^d/K  , \nu\bigr)$.
    \end{itemize}
   By means of  the above fields we  build an   \emph{augmented random field} given by the i.i.d. random variables $(A_z^{\rm au})_{z\in \ezd}$, where 
   \be\label{augmentin}
   A_z^{\rm au}:= 
   \begin{cases}
    A_z\wedge \min _{1\leq j \leq K} T_z^{(j)} & \text{ if } h \searrow\\
   A_z\lor  \max  _{1\leq j \leq K} T_z^{(j)} & \text{ if } h \nearrow
   \end{cases}
 \quad  \,.
   \en
      \end{Definition}
Note that $A_z$ and $A_z^{\rm au}$ have value in  $\D\cup \{+\infty\}$ if $h \searrow$, and  in $\D\cup \{-\infty\}$ if $h \nearrow$.
Let us  clarify  the relation of the  random fields introduced in Definition \ref{cavallo} with the PPP$(\l,\nu)$. We observe that 
 a PPP$(\l,\nu)$ can be obtained as follows.
 Let 
 \begin{align}
 & \{(x,E_x)\,:\, x\in \s\} \,, \label{mela71}\\
&  \{ (x,E_x)\,:\, x\in  \xi^{(j)} \} \qquad j=1,2,\dots, K\,, \label{mela72}
 \end{align} be independent marked PPP's, respectively with law  PPP$(\l_*,\nu)$ and PPP$( (\l-\l_*)/K,\nu)$.  \rui{In particular, 
 the random sets $\s$ and $\xi^{(j)}$, with $1\leq j \leq K$, are a.s. disjoint and correspond to PPP's with intensity $\l_*$ and $ (\l-\l_*)/K$ respectively.}
 The number of points in $\s\cap R_z$ (respectively  $\xi^{(j)}\cap R_z $) is a Poisson random variable with parameter $\l_* \e^d$ (respectively  $(\l-\l_*)\e^d /K$).
 Then,  setting $\xi := \s \cup \bigl( \cup _{j=1}^K  \xi^{(j)} \bigr)$,  the marked point process $\{(x,E_x)\,:\, x\in \xi\} $ is a PPP$(\l, \nu)$.
Let us first suppose that   $h\searrow$. We define
\begin{align}
 B_z &:= \inf\{ E_x\,:\, x \in \s \cap  R_z\}\,, \qquad \;\;\;z \in  \e\bbZ^d\,,\label{bibo}\\
  B_z^{(j)}&:= \inf \{ E_x\,:\, x \in \xi^{(j)} \cap  R_z\}\,, \qquad z \in  \e\bbZ^d\,,\;j=1,2,\dots,K\,.\label{biboj}
 \end{align}
 Trivially, we have 
\be\label{lince}
B_z^{\rm au}:=B_z\wedge  \min _{1\leq j \leq K} B_z^{(j)} =  \inf\{ E_x\,:\, x \in \xi \cap  R_z\}\,, \qquad z \in  \e\bbZ^d\,.
\en
The above fields in \eqref{bibo} and \eqref{biboj} are independent  (also varying $j$) and moreover we have the following identities between laws:
\begin{align}
& \left(B_z\right)_{z\in \e \bbZ^d} \stackrel{\cL}{=} \left(A_z\right)_{z\in \e \bbZ^d}\,, \\
& \left(B^{(j)} _z\right)_{z\in \e \bbZ^d} \stackrel{\cL}{=} \left (T^{(j)}_z\right)_{z\in \e \bbZ^d} \text{  for } j=1,2,\dots, K\,,\\
& \left(B_z^{\rm au}\right)_{z\in \e \bbZ^d} \stackrel{\cL}{=} \left(A_z^{\rm au}\right)_{z\in \e \bbZ^d}\,.
\end{align}
When $h \nearrow $ the above observations remain valid by replacing $\inf, \min$ with $\sup, \max$, respectively.

%
%

 \begin{Definition}[\rui{Graph $\bbG_+=(\bbV_+,\bbE_+) $}] \label{vichinghi_+}
 On the probability space $(\Theta,\bbP)$ we define the graph  
 $\bbG_+=(\bbV_+,\bbE_+) $  as
   \begin{align}
      \bbV_+&:=\{z\in \ezd\,:\, A^{\rm au}_z \in \bbR \} \,,\\
  \bbE_+& :=\left \{ \{z,z'\}: z\not = z' \text{ in } \bbV_+\,, \;\;\; |z-z'| \leq h(A^{\rm au}_z,A^{\rm au}_{z'}) -\a \right\}\,.
  \end{align}
   \end{Definition}
The plus suffix comes from the fact that in Sections \ref{scremato} and \ref{sec_RN_tools}  we will introduce  two other  graphs, $\bbG_-,$ and $\bbG$  respectively, such that 
$\bbG_-\subset \bbG\subset \bbG_+$. 
%

 
\begin{Definition}[\rui{LR crossing in $\bbG_+$ and $\bbR_L(\bbG_+)$}] \label{def_LR_bis}
Given  $L>0$, a left-right (LR) crossing  of  the box $\D_L:=[-L-2,L+2] \times [-L,L]^{d-1}$ in the graph $\bbG_+$ is any  sequence of distinct vertexes  $x_1,x_2 , \dots, x_n $ of $\bbG_+$ 
such that 
\begin{itemize}
\item $\{x_i ,x_{i+1}\} \in  \bbE_+$ for all $i=1,2,\dots, n-1$;
\item $x_1 \in (-\infty,-L-2)\times [-L,L]^{d-1} $;
\item $x_2, x_3,\dots, x_{n-1}\in \D_L$;
\item  $x_n \in (L+2, +\infty )\times [-L,L]^{d-1} $.
\end{itemize}
  We also define $\bbR_L(\bbG_+)$ as  the  maximal number of  vertex-disjoint LR crossings of $\D_L$ in $\bbG_+$.
\end{Definition}

 \begin{TheoremA}\label{teo2} Let $\bbG_+$ be the random graph given in Definition \ref{vichinghi_+}. Then
  there exist 
positive constants $c, c'$ such that 
\begin{equation}\label{fiorfiore}
\bbP \left( \bbR_L(\bbG_+)\geq c L^{d-1} \right) \geq 1- e^{- c' L^{d-1}} 
\end{equation}
for $L$ large enough.
\end{TheoremA}

To get Theorem \ref{teo1} it is enough to prove  Theorem \ref{teo2}:
\begin{Proposition}\label{fragolino}  Theorem \ref{teo2} implies Theorem \ref{teo1}.
\end{Proposition}
\begin{proof} We restrict to the case  $h\searrow $ as the case $h\nearrow $ can be treated by similar arguments.
By the above discussion concerning \eqref{bibo}, \eqref{biboj} and \eqref{lince},  
$\bbG_+$ has the same law of the following  graph $\bar \bbG_+$ built in terms of the
random field \eqref{lince}. The vertex set of $\bar \bbG_+$ is given by $\{z\in \ezd\,:\, B^{\rm au}_z <+\infty\}$. The edges of $\bar \bbG_+$ are given by the unordered pairs $\{z,z'\}$  with $z\not =z'$ in the vertex set  and 
  \be \label{castello}
  |z-z'| \leq h(B_z^{\rm au}, B_{z'}^{\rm au}) -\a\,.
  \en
  Due to \eqref{lince} for each 
 vertex $z$ of $\bar \bbG_+$ we can fix a point $x(z)\in \xi\cap R_z $ such that  $E_{x(z)}=  B^{\rm au}_z$. Hence, if $\{z,z'\}$ is an edge of $\bar \bbG_+$,  then $x(z)$ and $x(z')$ are defined 
 and it holds 
 $|z-z'|  \leq h( E_{x(z)}, E_{x(z')})-\a$. As $x(z) \in R_z$ it must be $|x(z) -z| \leq \sqrt{d}\e =\a/100$ and, similarly, $|x(z')-z'| \leq \a/100$. It then follows that $|x-y| \leq h (E_x, E_y)$ where $x=x(z)$ and $y=x(z')$. As $x\in R_z$, $y\in R_{z'}$ and $R_z\cap R_{z'}=\emptyset$, it must be $x\not =y$. Due to the above observations  $\{x,y\}$ is an edge of $\cG(h,\l)$.

 We extend Definition \ref{def_LR_bis} to $\bar \bbG_+$ (it is enough to replace $\bbG_+$ by $\bar \bbG_+$ there). Due to the above discussion, 
 if $z_1, z_2, \dots, z_n$  is a LR crossing of  the box $\D_L$ for $\bar \bbG_+$, then we can extract from $x(z_1), x(z_2), \dots, x(z_n)$  a  LR crossing  of the box $[-L-1,L+1]^d$ for $\cG(h,\l)$ (we use that  $x(z_1), x(z_2), \dots, x(z_n)$ itself is a path for $\cG(h,\l)$,   $|x(z_i)-z_i| \leq \a/100$, and 
edges of $\bar \bbG_+$ have length at most $1-\a$). Since disjointness is preserved, we deduce that 
 $R_{L+1}\bigl( \cG(h,\l)\bigr) \geq \bbR_L(\bar \bbG_+)$. Due to this  inequality  Theorem \ref{teo2} implies Theorem \ref{teo1} (by changing the constants $c,c'$ when moving from Theorem \ref{teo2} to Theorem \ref{teo1}).
\end{proof}

Finally,   we show that, to prove Theorem \ref{teo2}, it is enough to  have
  a good control on the LR crossings contained in 2d slices:

  \begin{Proposition}\label{carletto} Fixed   a positive integer $k$, we call $\rui{\bbR^*_L(\bbG_+)}$ 
   the maximal number of vertex-disjoint LR crossings of the box  $\D_L=[-L-2,L+2]\times [-L,L]^{d-1}$ for the graph  $\bbG_+$ whose vertexes, \rui{apart from} the first and last one,  are included in the slice 
   \be\label{rondine}
   [-L-2,L+2] \times [-L,L] \times [-k,k)^{d-2} \,,
   \en
   while the first and last one are included, respectively, in 
    $ (-\infty, -L-2) \times [-L,L] \times [-k,k)^{d-2}$ and $(L+2,+\infty) \times [-L,L] \times [-k,k)^{d-2} $.
   If  there exist positive constants $c_1,c_2$ such that
   \begin{equation}\label{problema2d}
  \bbP( \rui{\bbR^*_L(\bbG_+)} \geq c_1 L) \geq 1- e^{-c_2 L}
  \end{equation} 
    for  $L$ large enough,   
  then the claim of Theorem \ref{teo2} is fulfilled (i.e. $\exists c,c'>0$ such that \eqref{fiorfiore} is true for $L$ large enough).
  \end{Proposition}
  \begin{proof}
  For each $z\in 2k \bbZ^{d-2}$ we consider the slice 
 \[ S(z):=[-L-2,L+2] \times [-L,L] \times \left ( z+ [-k,k)^{d-2} \right)  \,.
 \] Note that, when varying $z$ in $2k \bbZ^{d-2}$,  the above slices are disjoint and that 
 $\D_L$ contains at least  $\lfloor 2 L/ 2k
\rfloor^{d-2}\asymp c_0 L^{d-2}$  slices of the above form.

Let us assume \eqref{problema2d}.
By translation invariance and independence of the random variables \rui{$A^{\rm au}_z$ with $z\in \ezd$ appearing in  \eqref{augmentin}}, the number  \rui{$X$} of disjoint slices $S(z) \subset \D_L$ \rui{with $z\in 2k \bbZ^{d-2}$}  including  at least  $c_1 L$ vertex-disjoint LR crossings of $\D_L$  for  $\bbG_+$  stochastically dominates a binomial random variable $Y$ with parameters $ n\asymp  c_0  L^{d-2}$ and $p:=1- e^{-c_2 L}$ (at cost  to enlarge the probability space $(\Theta,\bbP)$ we can think $Y$ as defined on $\Theta$). Setting $\d := e^{-c_2 L}$ we get
\begin{multline*}
 \bbP( \rui{X}< n/2)  \leq 
\bbP( Y < n/2)= \bbP( \d ^{Y} > \d^{\frac{n}{2}})\leq \d^{-\frac{n}{2}} \bbE\bigl[ \d^{Y}\bigr]  = 
\d^{-\frac{n}{2}} [ \d p + 1-p] ^n \\= \d^{-\frac{n}{2}} [ \d- \d^2 +\d] ^n\leq \d^{-\frac{n}{2}} [ 2\d ] ^n= 2^{c_0(1+o(1)) L^{d-2}} e^{- c_0 c_2 (1+o(1))L^{d-1}/2} \,.
\end{multline*}
\rui{As the event $X\geq n/2$ implies that $\bbR(\bbG_+)\geq c_1 L n/2 \asymp  (c_0 c_1 /2) L^{d-1}$, we get} \eqref{fiorfiore} in Theorem \ref{teo2}.
\end{proof}

\subsection{Properties of $\bbG_+$} \label{scremato}
In this subsection we want to isolate the properties of $\bbG_+$ that follow from the main assumptions and that will be crucial to prove Theorem \ref{teo2}.

 \begin{Definition}[\rui{Graph  $\bbG_-= (\bbV_-, \bbE_-)$}] \label{vichinghi_-}
 On the probability space $(\Theta,\bbP)$ we define the graph $\bbG_-= (\bbV_-, \bbE_-)$
 as 
  \begin{align}
   & \bbV_-:=\{z\in \ezd\,:\, A_z \in \bbR \} \,,\\
  & \bbE_- :=\left \{ \{z,z'\}: z\not= z' \text{ in } \bbV_-\,, \; |z-z'| \leq h(A_z,A_{z'}) -3\a \right\}\,\,.  \end{align}
   \end{Definition}

  \begin{Lemma}\label{john} 
The graph   $\bbG_-$  percolates $\bbP$--a.s.
 \end{Lemma}
 \begin{proof}  We restrict to the case  $h\searrow $ as the case $h\nearrow $ can be treated similarly.
 By the discussion following Definition \ref{cavallo} (recall the notation there)
 it is enough to prove that the graph $\bar{\bbG}_-$ percolates a.s., where 
 $\bar{\bbG}_-$ is defined as $\bbG_-$  with $A_z$ replaced by $B_z$.
 
  Let  $x\not =y$ be points of $ \s$ such that 
\be\label{risorgive} |x-y| \leq h(E_x, E_y) -\ell_* \,.\en
Equivalently, $\{x,y\}$ is an edge of the graph $\cG(h-\ell_* , \l_*)$ built by means of the marked PPP  $\{ (x,E_x)\,:\, x\in \s\}$.
Let  $z(x)$ and $z(y)$ be the points in $\e \bbZ^d$ such that $x\in R_{z(x)} $ and $y\in R_{z(y)}$.  Trivially, $|z(x)-x|\leq \e\sqrt{d}$,  $|z(y) -y|\leq \e\sqrt{d}$, $B_{z(x)} \leq E_x$ and $B_{z(y)}\leq E_y$. Then,  from Assumption (A5),  \eqref{risorgive}  and  since $10 \a \leq \ell_*$ in  Definition \ref{vinello},
  we get
\[ | z(x)-z(y)| \leq |x-y|+ 2 \e \sqrt{d}\leq  h(E_x, E_y) -\ell_*+\a/50\leq  h(B_{z(x)} , B_{z(y)}) -3 \a \,.
\]
As a consequence, for each edge $\{x,y\}$ in  $\cG(h-\ell_*, \l_*)$, either we have $z(x)=z(y)$ or we have that $\{z(x), z(y)\}$ is an edge of   $\bar{\bbG}_-$.
  Since $\cG( h-\ell_*, \l_*) $ a.s. percolates by (A2), due to the above observation we conclude that  $\bar{\bbG}_-$  a.s.  percolates. 
 \end{proof}

We conclude this section by treating the FKG inequality. On the probability space $(\Theta, \bbP)$ we introduce the partial ordering $\preceq$ as follows: given $\theta_1, \theta_2\in \Theta$ we say that  $\theta_1 \preceq \theta_2 $ if, for all $z \in \ezd$ and $j\in \{1,2,\dots, K\}$, it holds
 \be
 \begin{cases}
 A_z (\theta_1 ) \geq  A_z (\theta _2) \; \text{ and }  T^{(j)}_z (\theta_1) \geq  T^{(j)}_z(\theta_2)  & \text{ if } h \searrow\,,\\
 A_z (\theta_1 ) \leq  A_z (\theta _2) \;\text{ and } T^{(j)}_z (\theta_1) \leq   T^{(j)}_z(\theta_2)  & \text{ if } h \nearrow\,.
 \end{cases}
 \en
 We point out that, due to Assumption (A5),  if $\theta_1 \preceq \theta_2$ then $\bbG_-(\theta_1) \subset \bbG_-(\theta_2)$  and $\bbG_+(\theta_1) \subset  \bbG_+(\theta_2)$. 
Since dealing with i.i.d. random variables, we  also have that the partial ordering $\preceq$ satisfies the FKG inequality: if $F,G$ are increasing events for $\preceq$, then $\bbP(F\cap G)\geq  \bbP(F)\bbP(G)$.

\emph{At this point, we can disregard the original problem and the original random objects. One could start afresh keeping in mind only  Assumptions (A2), (A3), (A4), (A5), 
Definitions \ref{vinello}, \ref{cavallo}, \ref{vichinghi_+}, \ref{def_LR_bis}, \ref{vichinghi_-}, Lemma \ref{john} and the    FKG inequality for the partial order $\preceq$ on $(\Theta, \bbP)$. It then remains to prove   \eqref{problema2d}  in Proposition \ref{carletto} for some fixed positive integer $k$ and  $L$ large enough.}

\section{Tanemura's algorithm for LR crossings in $\bbZ^2$}\label{sec_japan}
  In this section we recall a construction introduced by Tanemura in  \cite[Section 4]{T} to control  the number of LR crossings in a 2d box by stochastic domination with a 2d Bernoulli site percolation. We point out that we had to extend one definition from \cite{T} to treat more general cases (see below for details).

   In order to have a notation close to the one in \cite[Section 4]{T}, \rui{given a positive integer $M$} we consider the  box 
   \be\label{scooby}
   \L :=( [0,M+1] \times [0, M-1] )\cap \bbZ^2\,.
   \en $\L$ has a graph structure, with unoriented edges between points at distance one.
  Let $(x_1, x_2, \dots, x_n)$  be a string of points in $\L$, such that  $\{x_1, x_2, \dots,x_n\}$ is  a  connected subset of  $\L$.
   We \rrr{now}
     introduce a  total  order \rrr{$\prec$} on $\D \{x_1, \dots, x_n\}$
   (in general, given    $A\subset \bbZ^2$, $\D A:= \{ y\in \bbZ^2 \setminus A\,:\, |x-y|=1\text{ for some } x\in A\}$). \rrr{Note that we} have to modify  the definition in \cite[Section 4]{T} which is restricted there to the case that $(x_1, x_2, \dots, x_n)$ is   a  path in $\bbZ^2$.
       For later use, it is more convenient to describe the ordering \rrr{$\prec$} from the largest to the smallest element. We denote by $\Psi $ the anticlockwise rotation  of $\pi/2$ around the origin in $\bbR^2$  (in particular, 
$\Psi (e_1)=  e_2$ and $\Psi (e_2)=-e_1$). We first introduce an order $\prec_k$ on the sites  in $\bbZ^2$ neighboring $x_k$ as  follows. Putting $x_0:= x_1-e_1$, for $k=1,2,\dots, n$ we set 
\[ x_k + \Psi (v) \succ_k x_k + \Psi ^2 (v) \succ_k x_k + \Psi ^3 (v) \succ_k  x_k + \Psi^4 (v)=x_{a(k)}\,,\]
where  $v:=x_{a(k)}-x_k $ and $a(k):= \max \{j: 0\leq j\leq n \text{ and } |x_k-x_j|=1\} $.
 The order \rrr{$\prec$} on $\D  \{x_1, \dots, x_n\}$  is obtained as follows. The largest elements are the sites of $\D  \{x_1, \dots, x_n\}$   neighboring $x_n$ (if any), ordered according to $\succ_n$.
The next elements, in decreasing order, are the sites $\D  \{x_1, \dots, x_n\}$  neighboring $x_{n-1}$ but not $x_n$ (if any), ordered according to $\succ_{n-1}$. As so on, in the sense that in the generic step one has to consider the elements of  $\D  \{x_1, \dots, x_n\}$  neighboring $x_k$ but not $x_{k+1}, \dots, x_n$ (if any), ordered according to $\succ_k$ (see Figure \ref{tanemura1}).

\begin{figure}
\includegraphics[scale=0.4]{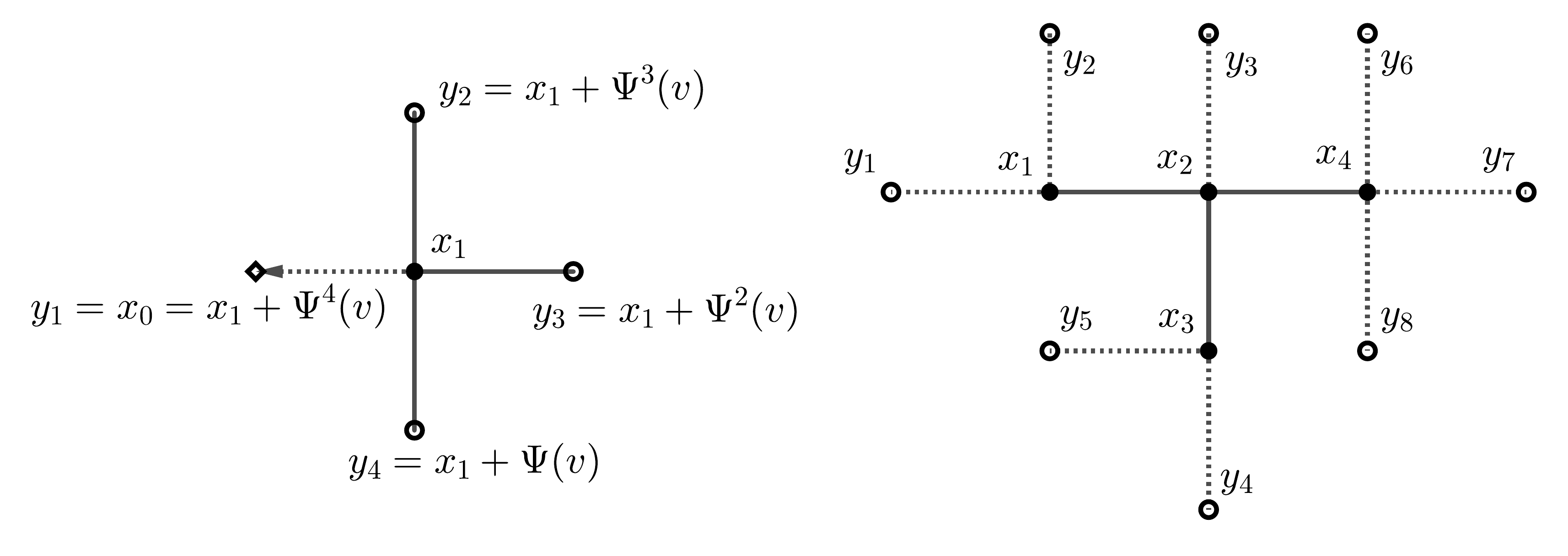}
\captionsetup{width=0.9\linewidth}
\caption{Left: ordering $\prec_1$ on $\D\{x_1\}=\{y_1,y_2,y_3, y_4\}$. We have $a(1):=0$, $v=-e_1$ and $y_1\prec y_2\prec y_3\prec y_4$.  
Right: ordering on the boundary of $\D\{x_1,x_2,x_3,x_4\}=\{y_1,\ldots,y_8\}$. We have $y_1 \prec y_2\prec \cdots \prec y_8$. A dotted segment is drawn between two points $y_i$ and $x_j$ if the point $y_i$ is a neighbor of $x_j$ but not of $x_{j+1}, \dots, x_4$. 
}
\label{tanemura1}
\end{figure}

Suppose that  we have a procedure to decide,  given  $\theta  \in \Theta$ (cf.~Definition \ref{cavallo}),   $x_1,x_2, \dots, x_n, x_{n+1}\in \L$ and $x\in \{x_1, \dots, x_n\}$, if the point $x_{n+1}\in \L$ is occupied (for $n=0$)  and if it is occupied and  linked to  $x$  (for $n\geq 1$), 
knowing the occupation state  of the points $x_1,x_2, \dots, x_n\in \L$ and the presence or absence of  links between them. 
The precise definition of occupation and link is not relevant now, here we assume  that these properties have been previously  defined and that can be checked knowing $\theta$.

We  now   define a random field $\z=\left(\z(x)\,:\, x \in \L\right)$ with $\z(x)\in \{0,1\}$ on the probability space $\bigl( \Theta, \bbQ\bigr)$. $\bbQ$ is a probability measure on $\Theta$, which can differ from the probability $\bbP$ appearing   in Definition \ref{cavallo}
(in the application $\bbQ$ will be a suitable conditioning of $\bbP$).
To define the field $\z$,  we \rrr{build below} the sets $C^s_j= ( E^s_j, F^s_j)$, with  $s \in \{0,1,\dots, M-1\}$  and  $j=1,2,\dots, M^2 $.
The construction will fulfill the following properties:  
\begin{itemize}
\item$E^s_j$ will be a connected subset of $\L$;
\item  there exists $J^s \in \{1,2\dots,M^2\}$ such that, for $j< J^s$, 
 $C^s_{j+1}=( E^s_{j+1}, F^s_{j+1})$ will be obtained from $C^s_j=( E^s_j, F^s_j)$   by adding exactly a point (called $x^s_{j+1}$) either to $E^s_j$ or to $F^s_j$; while, for $j \geq J^s$, $C^s_{j+1}=( E^s_{j+1}, F^s_{j+1})$ will equal $C_j^s=( E^s_j, F^s_j)$; 
\item  $\z \equiv 1$ on $E^s_j$ and $\z\equiv 0$ on $F^s_j$.
\end{itemize}
Note that, because of the above rules, $E^s_j \cup F^s_j =\{x^s_1, x^s_2, \dots ,x^s_j\}$ for $1\leq j \leq J^s$.
 In order to make the construction clearer, in Appendix \ref{app_tanemura} we illustrate the construction in a particular example.

 We now start with the definitions involved in the construction. In what follows, the index $s$ will vary in $\{0,1,\dots, M-1\}$. We also  set $x^s_1:=(0,s)$. We build  the sets $C^0_1$, $C^1_1$,...,$C^{M-1}_1$ as follows.  
If the point  $x^s_1$ is occupied, then we set 
 \be
\z(x^s_1):=1 \text{ and } C^s_1 :=\left( E^s_1, F^s_1 \right):= (\{x^s_1\},\emptyset) \,,\en
otherwise we set 
\be
\z(x^s_1):=0 \text{ and } C^s_1 :=\left( E^s_1, F^s_1 \right):= (\emptyset, \{x^s_1\}) \,.
\en
We then define iteratively  
\be\label{cannolo42}
C^0_2, \;C^0_3, \dots,\; C^0_{M^2},\; C^1_2,\; C^1_3,\dots,\; C^1_{M^2},\dots,
\; C^{M-1}_2,\; C^{M -1}_3, \dots,\; C^{M-1}_{M^2}
\en
  as follows.  If $E^s_1=\emptyset$, then  we declare  $J^s:=1$, thus implying that $C_1^s=C_{2}^s=\cdots= C_{M^2}^s$.
   We restrict now to the case $E^s_1\not =\emptyset$.
Suppose that we have defined all the sets  preceding $C^s_{j+1}$  in the above string \eqref{cannolo42} {(i.e. up to $C^s_j$), that we have not declared that $J^s$ equals some value in $\{1,2,\dots, j\}$ and that we want to define $C^s_{j+1}$.  
  We call $W^s_j$ the points of $\L$ involved in the construction up to this moment, i.e. \[
W^s_j = \{x_1^k:0\leq k \leq M-1\} \cup \{ x^{s'}_r :  0 \leq s'< s, \, 1< r \leq M^2\}
\cup \{ x^{s}_r : 1< r \leq j\}
\,.
\]
As already mentioned, it must be $E^s_0 \subset E^s_1\subset \cdots \subset E^s_j$ and at each inclusion either the two sets are equal or the second one is obtained from the first one by adding exactly a point. We  then write $\bar{E}^s_j$ for the  non--empty string obtained  as follows: the entries of $\bar{E}^s_j$ are the elements of $E^s_j$, moreover if $x^s_a, x^s_b \in E^s_j$ and $a<b$, then $x^s_a$ appears in $\bar{E}^s_j$ before  $x^s_b$. Note that the above property ``$x^s_a, x^s_b \in E^s_j$ and $a<b$'' simply means that the point $x^s_a$ has been added before $x^s_b$ to one of the sets $E^s_0 \subset E^s_1\subset \cdots \subset E^s_j$. Then,  on $\D E^s_j$  we have the ordering $\prec$  (initially defined) associated to the string $\bar{E}^s_j$.
We call  $\cP^s_j$ the following property:  $E^s_j$ is disjoint from  the right vertical face of $\L$, i.e.   $ E^s_j \cap \left(\{M+1\} \times \{0,1,\dots, M-1\}\right)=\emptyset$, and   $(\L\cap  \rrr{\D{E}^s_j})  \setminus W^s_j \not = \emptyset$.   If property $\cP^s_j$ is satisfied, then  we denote by $x^s_{j+1} $  the last element of $(\L \cap \rrr{\D{E}^s_j})  \setminus W^s_j $ \rrr{w.r.t. $\prec$}. We define $k$ as the largest integer $k$ such that  $x_k^s\in E^s_j$ and $|x^s_{j+1} -x^s_k |=1$. If $x^s_{j+1}$ is occupied  and linked to $x^s_k$,
 then we  set
\be
\z(x^s_{j+1}):=1 \text{ and } C^s_{j+1} :=\left( E^s_j \cup \{ x^s_{j+1}\}, F^s_j \right) \,,\en
otherwise we set 
\be
\z(x^s_{j+1}):=0 \text{ and } C^s_{j+1} :=\left( E^s_j, F^s_j \cup \{ x^s_{j+1}\} \right)\,.
\en
On the other hand, if property $\cP^s_j$ is not verified, then we declare $J^s:=j$, thus implying that $C_j^s=C_{j+1}^s=\cdots= C_{M^2}^s$.

It is possible that the set  $\cup _{s=0}^{M-1} \cup _{j=1}^{M^2} \left( E^s_j\cup F^s_j\right)$ does not fill all $\L$. In this case we  set $\z\equiv 0$ on the remaining points.
This completes the definition of the random field $\z$.

Above we have constructed the sets $C^s_j$ in the following order: 
$C^0_1$, $C^1_1$, $\dots$, $C^{M-1}_1$,
$C^0_2$,  $C^0_3$, $\dots,$ $ C^0_{M^2},$ $ C^1_2,$  $ C^1_3,$  $\dots,$  $ C^1_{M^2},$ $\dots,$ $C^{M-1}_2,$  $ C^{M -1}_3,$   $ \dots,$   $ C^{M-1}_{M^2}$. 

We make the following assumption:

\medskip

{\bf Assumption (A)}: \emph{For some $p\in [0,1]$ at every step in the above construction the probability to add 
a point to a set of the form $E^s_j$, conditioned to  the construction performed before such a step,  is lower bounded by $p$.
}

\medskip

 Call  $N_M$ the maximal number of  vertex-disjoint LR crossings  of the box $\L$ for $\z$,  where  $\{x,y\}$ is an edge if $x$ and $y$ are distinct, linked, occupied sites. Here crossings are the standard ones for percolation on \rrr{$\bbZ^2$} \cite{G}. Note that $N_M$ also equals the number of indexes $s\in \{0,1,\dots,M-1\}$ such that  $E^s_{M^2} $ intersects the right vertical face of $\L$.
 By establishing a stochastic domination on a 2--dimensional site percolation in the same spirit of \cite[Lemma 1]{GM} (cf. \cite[Lemma 4.1]{T}), \rrr{due to Assumption (A)  the following holds:}
\begin{Lemma}
Under Assumption (A)   $N_M$  stochastically dominates the
maximal number of vertex-disjoint LR crossings in $\L$ for  a site percolation  on $\bbZ^2$ of parameter $p$.
\end{Lemma}

Due to the above lemma and the results on LR crossings for the Bernoulli  site percolation (cf. \cite[Remark~(d)]{GM}), we get: 
\begin{Corollary}\label{cor_pistacchio}
If  Assumption (A)   is fulfilled with $p>p_c(2)$, where $p_c(2)$ is  the critical probability for the site percolation on $\bbZ^2$,  then 
  there exist constants   $c,c'>0$ such that 
$\bbQ( N_M\geq c  M) \geq 1- e^{-c' M}$ for \rui{any positive integer $ M$}.
\end{Corollary}

\section{Proof of Eq. \eqref{problema2d} as a byproduct of  Tanemura's algorithm and  renormalization}\label{sec_ginepro}
 In this section we prove Eq. \eqref{problema2d} by combining Tanemura's algorithm  and a renormalization scheme inspired by the one in \cite{GM}. For the latter we will  not discuss here all technical aspects, and postpone a detailed treatment to the next sections.

 Given $m\in \bbN_+$ and $z\in \ezd$ we set 
\be\label{mentolino} B(m):=[-m,m]^d \cap \ezd \text{ and  } B(z,m):= z+ B(m)\,.
\en
 Recall that  $U_*:= U_*(\a/2)$  (cf. Assumption (A4) and Definition \ref{vinello}) and that $\D={\rm supp}(\nu)$ and $\sup h=1$ (cf. Warning \ref{aaah}).
 Note also that, since the structural function $h$ is symmetric, by applying twice \eqref{indigestione} we get 
 $h(a, \tilde{a}) \geq h(b, \tilde{a}) -\a/2\geq h(b, \tilde{b}) -\a$ for any $b, \tilde{b}\in \D$ and $a,\tilde a \in U_*$.  
 Hence, we have 
 \be\label{mango}
h(a, \tilde{a}) \geq \sup_{ b, \tilde{b}\in \D} h(b, \tilde b) - \a =1-\a \qquad \forall a, \tilde a \in U_*\,.
\en

  \begin{Definition}[\rui{Seed}]  \label{def_seed} 
Given $z\in \ezd$ and $m\in \bbN_+$,
 we say that 
 $B(z,m)$ is a     \emph{seed} if  $A_x\in U_* $ for all $x\in B(z,m)$.
 \end{Definition}

Recall the graph $\bbG_-=(\bbV_-,\bbE_-)$ introduced in Definition \ref{vichinghi_-}.
Note that any seed is a subset of $\bbV_-$.  Moreover, 
a seed is a region of ``high connectivity'' in the minimal graph $\bbG_-$:
\begin{Lemma}\label{ironman} 
If    $B(z,m)$ is a seed, then $B(z,m)$  is a connected subset   in the graph  $\bbG_-$.  \end{Lemma}
 \begin{proof}  It is enough to show that $|x-y|\leq h(A_x,A_y) -3 \a$ for any $x,y \in B(z,m)$ with $|x-y|=\e$.  
  Since $\e = \a/100 \sqrt{d}$, we get  $|x-y| \leq \a/100$. 
  On the other hand, by \eqref{mango} and since $A_x,A_y \in U_*$, we have 
  $h(A_x,A_y)\geq 1-  \a $. To conclude it is enough to recall that $ 10 \a\leq 1$ (cf. Definition \ref{vinello}).
       \end{proof}
   
%
%
%
%
Recall that $p_c(2)$ is the critical probability for site percolation on $\bbZ^2$.
We now fix some relevant constants and recall the definitions of others:
\begin{itemize}
\item We fix $\e'\in (0,1)$ such that $ 1-6\e'\geq 3/4>p_c(2)$. 
\item We  fix   positive integers $m\leq n$ satisfying the properties stated in Lemma  \ref{pierpilori}  in Section \ref{sec_RN_tools} (their precise value is not relevant to follow the arguments below).  
\item Recall Definition \ref{vinello} for $\e$.
\item We let $N:= n+m+\e$.
\item $L$ is a positive number as in \eqref{problema2d}.
\item  Given $L$, the constant  $M$  (introduced in \eqref{scooby}) is defined as the smallest positive integer such that $ 4N (M+1) > 2L +5 +m+N$ \rui{(in particular,  $M$ and $L$ are of the same order).}
\end{itemize}

In the rest of this section we explain how to get \eqref{problema2d} with $k=4N$. 
\rui{As we will see in a while, such a choice of $k$ guarantees independence properties in  the construction of left-right crossings.}
\smallskip

As in Tanemura's algorithm we take $ \L :=( [0,M+1] \times [0, M-1] )\cap \bbZ^2$ (in general we will use the notation introduced in Section \ref{sec_japan}).   We recall that   $x_1^s:= (s,0)$ for $s=0,1,\dots, M-1$. 
  We set 
  \be\label{deep}
  \bar x:= (x,0,0, \dots,0)\in \bbZ^d \text{  for  } x \in \bbZ^2\,.
  \en  Below, we will naturally associate to each point $x\in \L$ the point $ 4N \bar x $ in the renormalized lattice $4N \bbZ^d\subset \ezd$.

\begin{Definition}[\rui{Conditional probability $\bbQ$}]   We set $\bbQ(\cdot) :=\bbP(\cdot \,|\, D)$ where 
\be\label{dedalo} D:=\{ B(4N \bar{x}_1^s ,m) \text{ is a seed } \forall x_1^s \in \L\}
\en
\end{Definition}
To run Tanemura's algorithm we first  need to define when 
 the point  $x_1^s$ is occupied. Our definition  will imply that the graph \rui{$\bbG\subset \bbG_+$} contains a cluster centered at $4N \bar{x}_1^s$ as in \rui{Fig.~\ref{hobbit1} (left)}, when $d=2$.  In particular, the seed $ B(4N \bar{x}_1^s ,m) $ is connected in $\bbG_+$  by a cluster of points lying  inside the box $ B(4N \bar{x}_1^s ,n) $ to four seeds adjacent to the  faces of such a box in the directions $\pm e_1$, $\pm e_2$ ($e_1,e_2,\dots, e_d$ being the canonical basis of $\bbR^d$).
The precise definition (for all $d\geq 2$) of the occupation of $x_1^s$ is rather technical and explained in Section \ref{moto_GP} (it corresponds to Definition \ref{0occ}, when   $\bar{x}_1^s$  is thought of  as the new origin of $\ezd$).
  We point out that the cluster appearing in \rui{Fig.~\ref{hobbit1} (left)} is contained in a box of radius $N+m<2N$ centered at  \rui{at $4N \bar{x}^s_1$.
   So, to verify that 
$\bar x_1^s$ is occupied, 
it is enough to know the random variables  $A_z$ with $z$  in  the interior
of the box $B( 4N \bar{x}^s_1, 2N)$. Since   the above interior parts are disjoint when varying $s$, we conclude that}  the events $\{ x_1^s \text{ is occupied}\}$ are $\bbQ$--independent when varying $s$ among $\{0,1,\dots,M-1\}$.   Moreover, by Proposition \ref{prop_occ_origin},  $\bbQ( x_1^s \text{ is occupied})\geq 1-4\e'$.

\begin{figure}
\includegraphics[scale=0.07]{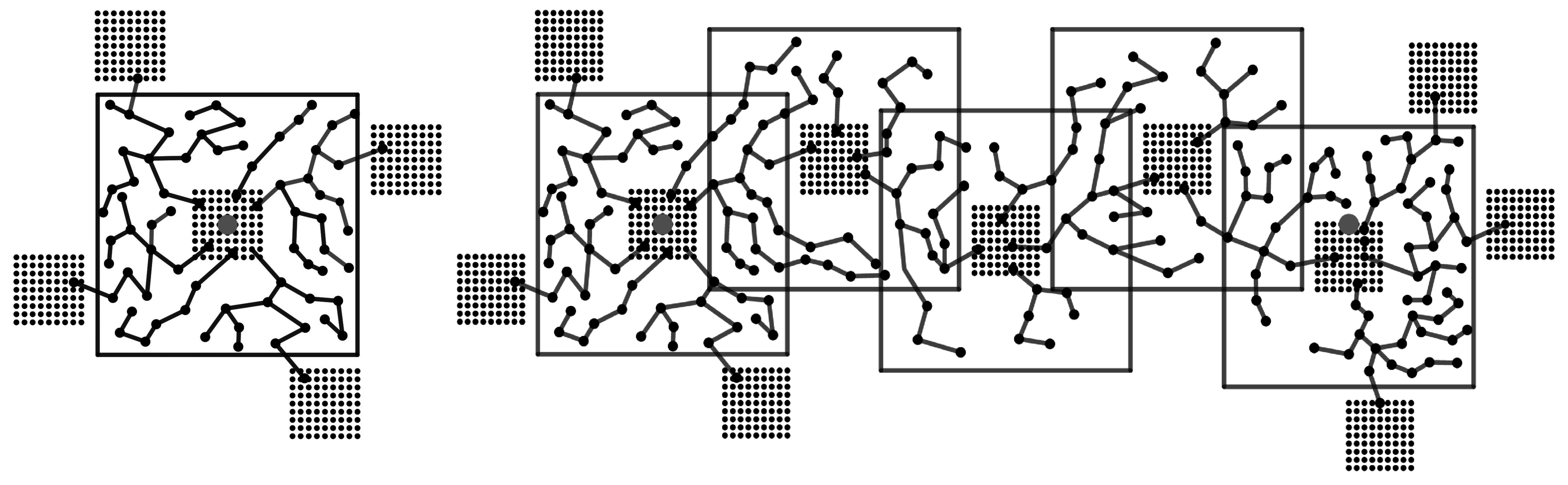}
\captionsetup{width=.9\linewidth}
\caption{Left: cluster contained in $\bbG_+$ when $x_1^s$ is occupied for $d=2$, the centered grey circle corresponds to $4N\bar x_1^s$, the large box has radius $n$, the smaller boxes have radius $m$ and are seeds. Right: cluster of $\bbG_+$ related to the definition of $(0,0)\to (1,0)$ knowing that $x^0_1=(0,0)$ is occupied, the two grey circles are given by $0$ and $4N e_1$.}
\label{hobbit1}
\end{figure}

\medskip

Knowing the occupation state of $x_1^s$, we can 
  define the sets $C_1^0, C_1^1,\dots, C_1^{M-1}$ in Tanemura's algorithm. Let us move to $C^0_2$. Let us assume for example that $x_1^0=(0,0)$ is occupied, hence $C^1_0=(\{x^0_1\},\emptyset)$. Then, by Tanemura's algorithm, one should check if $(1,0)$ is occupied and linked to the origin $x_1^0$ or not. We have first to define this concept. To this aim 
we need to  explore the graph $\bbG_+$ in the direction $e_1$ from the origin.  Roughly, when $d=2$, we say that  $ (1,0)$ is  linked to $(0,0)$ and occupied (shortly, $(0,0) \to  (0,1) $) if the graph $\bbG_+$ contains a cluster similar to the one in \rui{Fig.~\ref{hobbit1} (right)} and Fig.~\ref{hobbit3}, extending the cluster appearing in \rui{Fig.~\ref{hobbit1} (left)}.
Note that there is a seed (called $s_5$ in Fig. \ref{hobbit3}) in the proximity of $ 4N e_1 $ (the grey circle on the right in \rui{Fig.~\ref{hobbit1} (right)} and Fig.~\ref{hobbit3}) connected to four seeds neighboring  the box of radius $n$ concentric to $s_5$, one for each face in the directions $\pm e_1, \pm e_2$.  Hence, we have a local  geometry similar to the one of \rui{Fig.~\ref{hobbit1} (left)}. Moreover,  see Fig.~\ref{hobbit3}, the cluster turns  in direction $e_1$  and connects  inside $\bbG_+$ the seed $s_1$ at $4N\bar x^1_0=0$ to the seed  $s_5$ around $4N e_1$ by passing through the intermediate seeds $s_2,s_3,s_4$. Note that, in order to assure that $s_5$  lays around $4Ne_1$ the intermediate seeds have to be located alternatively up and down.
 The precise definition  (for all  $d\geq 2$) of the event $\{(0,0)  \to  (1,0) \}$, knowing that $x_1^0=(0,0)$ is occupied, is given by Definition \ref{1occ} in Section \ref{moto_GP}. Having defined this concept, we can build the set  $C^0_2$ 
 in Tanemura's algorithm.

%

 We move to $C^0_3$.
 Suppose for example that  $(0,0)  \to (1,0) $  occurs, hence $C^0_2=\bigl( \{(0,0), (1,0) \}, \emptyset\bigr)$. Then, 
 according to Tanemura's algorithm, we need to check if $(2,0)$ is linked to $(1,0)$
 and occupied (shortly, $(1,0) \to   (2,0) $). The definition of the last event   is similar to the one  of ``$(0,0) \to  (1,0) $''. Roughly, by means of three intermediate seeds (the first one given by $s_6$ in Fig.~\ref{hobbit3}) there is a cluster in $\bbG_+$ similar to the one in \rui{Fig.~\ref{hobbit1} (right)} connecting the seed in the proximity of $4N e_1$ to a seed in the proximity $8Ne_1$ and  this last seed   is connected to four seeds adjacent to the faces in the directions $\pm e_1, \pm e_2$ of a concentric box of radius $n$.   
 Let us suppose for example that  $ (1,0) \not \to  (2,0) $. Then $C^0_3=\bigl( \{(0,0), (1,0)\}, \{(2,0)\}\bigr)$.

\begin{figure}
\includegraphics[scale=0.32]{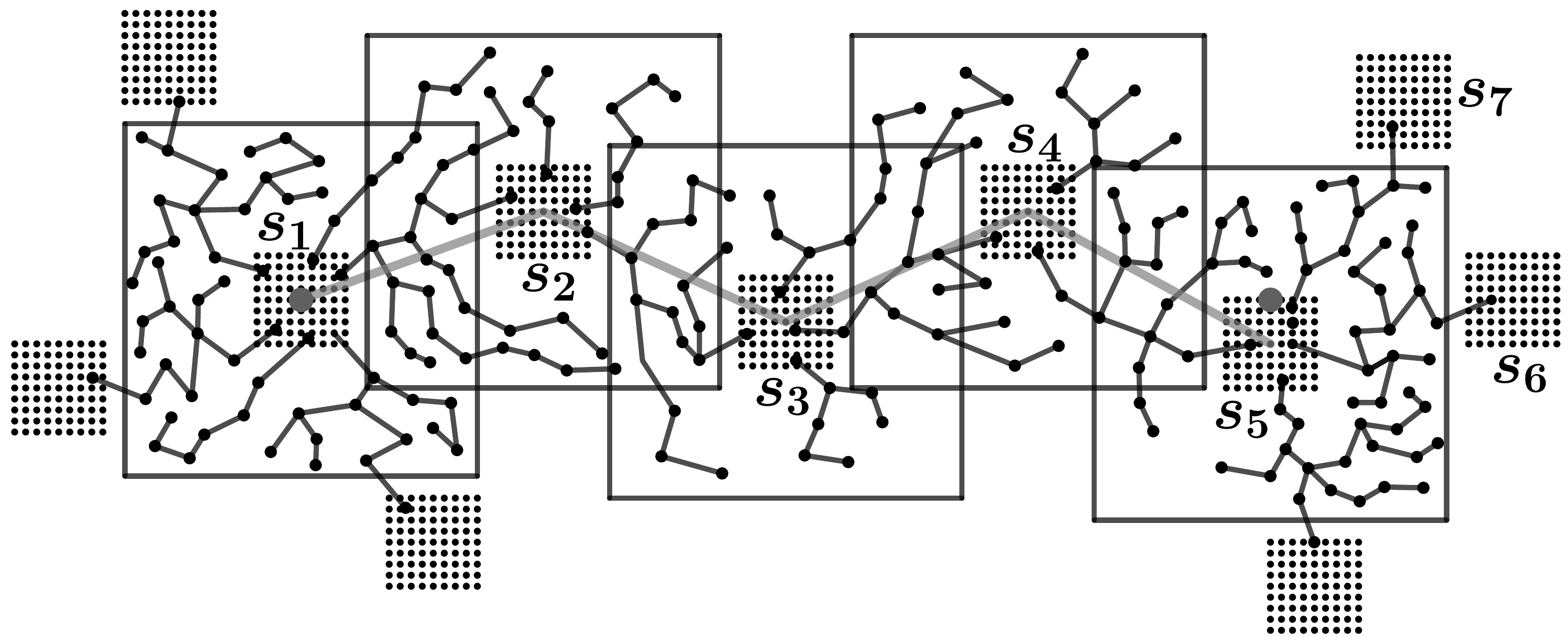}
\captionsetup{width=.9\linewidth}
\caption{\rui{Same cluster as in Figure \ref{hobbit1} (right). $s_1,s_2, \dots, s_7$ are seeds used to progressively extend the cluster along the construction.}}\label{hobbit3}
\end{figure}

 We move to $C^0_4$.  According to Tanemura's algorithm, we need to check if  $(1,0)$ is  linked to $(1,1)$ and occupied, shortly $(1,0)\to (1,1)$. This last  concept is roughly defined as follows: 
 by means of three intermediate seeds (the first one given by $s_7$ in Fig.~\ref{hobbit3}) there is a cluster in $\bbG_+$ connecting the seed $s_5$ in the proximity of $4N e_1$ to a seed in the proximity $4Ne_1+4Ne_2$ and this last seed   is connected to four seeds adjacent to the faces in the directions $\pm e_1, \pm e_2$ of a concentric box of radius $n$ (see Fig.~\ref{hobbit4}). We can then build the set $C^0_4$.

 In general, 
 for any $d\geq 2$ and  $v\in \{\pm (1,0)\,,\, \pm (0,1) \}$, the precise definition of ``$(a,b)\to (a,b)+v$  knowing that $(a,b)$ is occupied" is given by  Definition \ref{1occ}  \rui{apart from} changing origin and direction. We point out that in Definition \ref{1occ}, and in general in Section \ref{moto_GP}, we work with $\bbZ^2 \times\{0\}\subset \bbZ^d$ instead of $\bbZ^2$ ($\bbZ^2 \times\{0\}$ and $\bbZ^2$ are naturally identified).

  Proceeding in this way one defines  the whole sequence $C^0_2, \;C^0_3, \dots,\; C^0_{M^2}$. Then one has to build the sequence $C^1_2,\; C^1_3,\dots,\; C^1_{M^2}$. If $x_1^1=(0,1)$ is not occupied, i.e. $C^1_1=(\emptyset, \{(0,1)\})$, then one sets $C^1_1=C^1_2=C^1_3\dots=C^1_{M^2}$. Otherwise one starts to build a cluster in $\bbG_+$ similarly to what  done above, with the only difference that $x_1^1$ replaces $x_1^0$. As the reader can check, after reading the detailed definitions in Section \ref{moto_GP}, the region of $\ezd$ explored when 
  checking linkages and occupations for the cluster blooming from $x_1^1$ is far enough from the region explored for the  cluster blooming from $x_1^0$, and no spatial correlation emerges. 
 One proceeds in this way to complete Tanemura's algorithm.

 In Section \ref{moto_GP} we will analyze in detail the basic steps of the above construction and  show (see the discussion in Section \ref{francia}) the validity of Assumption (A) of Section  \ref{sec_japan}
with $p=1-6 \e'> p_c(2)$. As a consequence we can apply Corollary \ref{cor_pistacchio} and get that 
  there exist constants    $c_1,c_2>0$ such that 
$\bbQ( N_M\geq c _1 M) \geq 1- e^{-c_2 M}$\rui{, where} $N_M$ is the maximal number of vertex-disjoint LR crossings of $\L$ for  the graph with vertexes  given by occupied sites and having edges between nearest--neighbor linked occupied sites.

As rather intuitive  and  detailed 
 in Appendix \ref{app_ultimatum},  there is a  constant $ C>0$ (independent from $M$) such that   that event $\{N_M\geq c _1 M\}$ implies that the graph $\bbG_+$ has at least $C M$   vertex-disjoint LR crossings of a 2d slice of  size  $O(L)\times O(L) \times O(N)^{d-2}$ (recall that  $L \asymp M$). In particular, as discussed in Appendix   \ref{app_ultimatum},  the bound  $\bbQ(N_M\geq c _1 M) \geq 1- e^{-c_2 M}$ implies the estimate \eqref{problema2d} with $k=4N$ and  new positive constants $c_1,c_2$ there.
%

\begin{figure}
\includegraphics[scale=0.35]{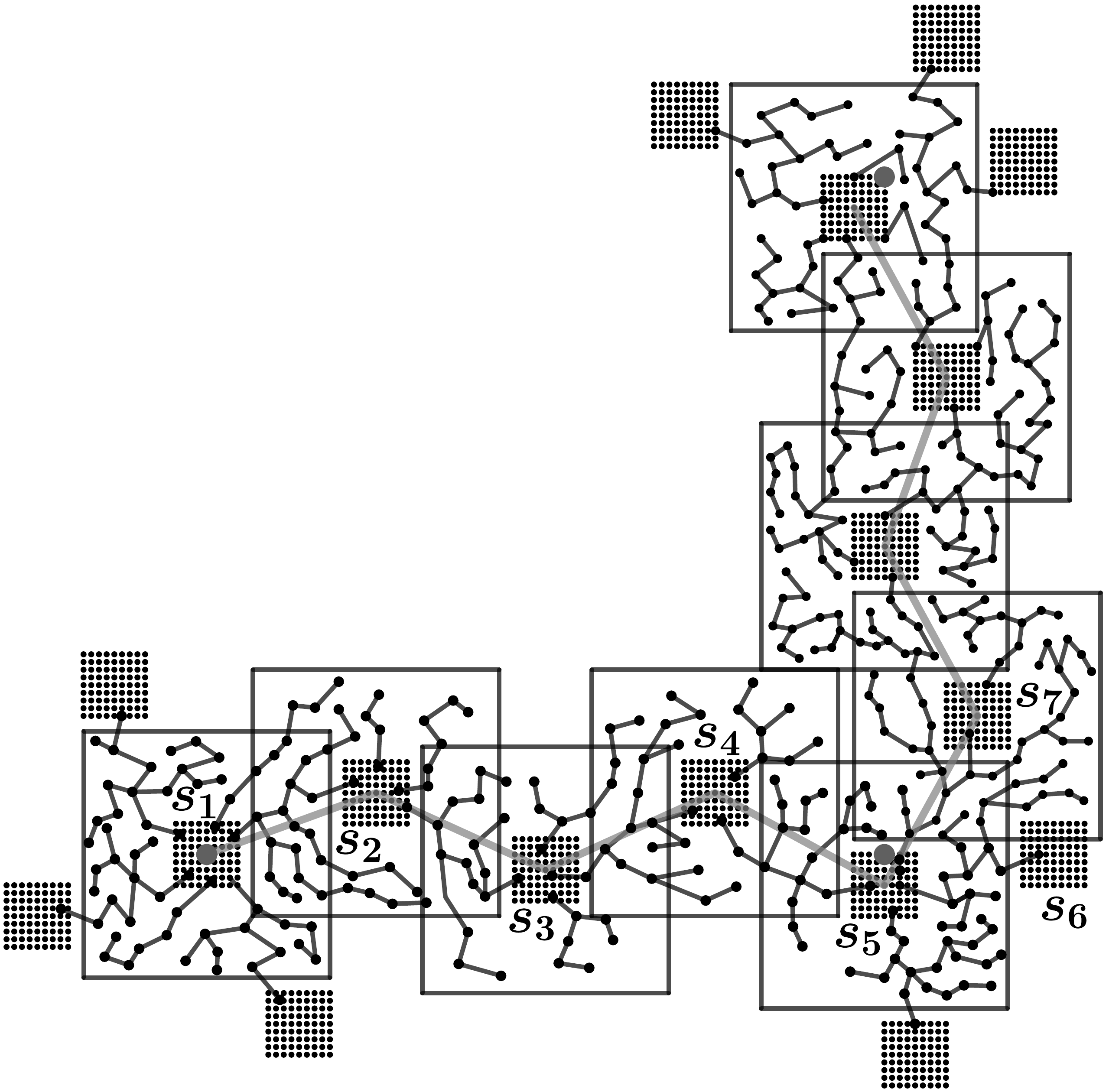}
\captionsetup{width=.9\linewidth}
\caption{Further extension of the cluster in $\bbG_+$ when $(1,0)\to (1,1)$.}
\label{hobbit4}
\end{figure}



\section{Renormalization: preliminary tools}\label{sec_RN_tools}
In  the rest  we will often write $\bbP(E_1,E_2, \dots, E_n)$ instead of $\bbP(E_1\cap E_2 \cap \cdots \cap E_n)$, also  for other probability measures.
 For the readers convenience we recall Definitions \ref{vichinghi_+} and \ref{vichinghi_-} of the graphs  $\bbG_+= (\bbV_+, \bbE_+)$ and  $\bbG_-=(\bbV_-,\bbE_-)$:
 \begin{align*}
    \bbV_-&:=\{z\in \ezd\,:\, A_z \in \bbR \} \,,\\
   \bbE_- &:=\left \{ \{z,z'\}: z\not= z' \text{ in } \bbV_-\,, \; |z-z'| \leq h(A_z,A_{z'}) -3\a \right\}\,,\\
   \bbV_+&:=\{z\in \ezd\,:\, A^{\rm au}_z \in \bbR \} \,,\\
  \bbE_+& :=\left \{ \{z,z'\}: z\not = z' \text{ in } \bbV\,, \;\;\; |z-z'| \leq h(A^{\rm au}_z,A^{\rm au}_{z'}) -\a \right\}\,.
   \end{align*}
We also  introduce the intermediate graph $\bbG$ (trivially, we have $\bbG_-\subset \bbG\subset \bbG_+$):
 \begin{Definition}[\rui{Graph $\bbG= (\bbV, \bbE )$}]  \label{vichinghi}
 On the probability space $(\Theta,\bbP)$ we define the graph $\bbG= (\bbV, \bbE )$
 as 
  \begin{align}
   & \bbV :=\{z\in \ezd\,:\, A_z \in \bbR \} \,,\\
  & \bbE :=\left \{ \{z,z'\}: z\not= z' \text{ in } \bbV_-\,, \; |z-z'| \leq h(A_z,A_{z'}) -2\a \right\}\,\,.  \end{align}
   \end{Definition}

  We introduce the following conventions:

\begin{itemize}

\item  Given  $x\in \bbV  $  and $C\subset \bbV $ with $x\not \in C$, we say that $x$ is \rrr{adjacent} to $C$ inside $\bbG $ if there exists $y\in C$ such that $\{x,y\}\in \bbE$.

\item  Given $A,B ,C \subset \ezd$, we say that ``$A \leftrightarrow B$ in $C$ for $\bbG$''  if there exist $x_1,x_2, \dots, x_k \in C \cap \bbV $ such that $x_1 \in A$, $x_k \in B$ and $\{ x_i ,  x_{i+1}\} \in \bbE$ for all $i: 1\leq i <k$.

\item Given a bounded set $A\subset \bbR^d$  we say that ``$A\lrgh \infty $ for $\bbG$'' if there exists an unbounded  path in $\bbG$ starting at some point in $A$.

\end{itemize}

Similar definitions hold for the graphs   $\bbG_-= (\bbV_-, \bbE_-)$ and $\bbG_+=(\bbV_+,\bbE_+)$.

\begin{Definition}[\rui{Sets $B(m)$, $T(n)$, $T(m,n)$, $K(m,n)$}] \label{sambinaA}
For $m \leq n\in\bbN_+$ we define the following sets:
\begin{align*}
&  B(m):=[-m,m]^d \cap \ezd \text{ and  } B(z,m):= z+ B(m)\,,\\
&  T(n):=\{ x\in \ezd :  n-1  <  \|x\|_\infty \leq n , \,0 \leq x_i \leq x_1 \; \forall i=1,2,\dots, d\}\,, \\
& T(m,n):= \bigl(  [n+\e ,n+\e + 2m]   \times [0,n]^{d-1}\bigr)\cap \ezd
 \,,\\
 & \rui{K(m,n):=\{
x \in \bbV\cap T(n) : \text{$x$  is adjacent in $\bbG$   to  a seed included in   $T(m,n)$}\}\,.}
\end{align*}
 \end{Definition} 
\rui{For the reader's convenience, in the above definition  we have  also  recalled \eqref{mentolino}.
We refer to Figure \ref{messicano1}-(left) for some illustration.  We recall that seeds have been introduced in Definition \ref{def_seed}. The definition of $K(m,n)$ can be restated as follows:
 $K(m,n)$ is given by the points    $x\in \bbV \cap T(n)$ such that, for some 
  $z \in \ezd$, the box   $B(z,m) \subset T(m,n)$ is  a   seed and 
  $\exists y  \in B(z,m)$ with $ \{ x,  y\}\in \bbE $.}

\begin{figure}
\includegraphics[scale=0.3]{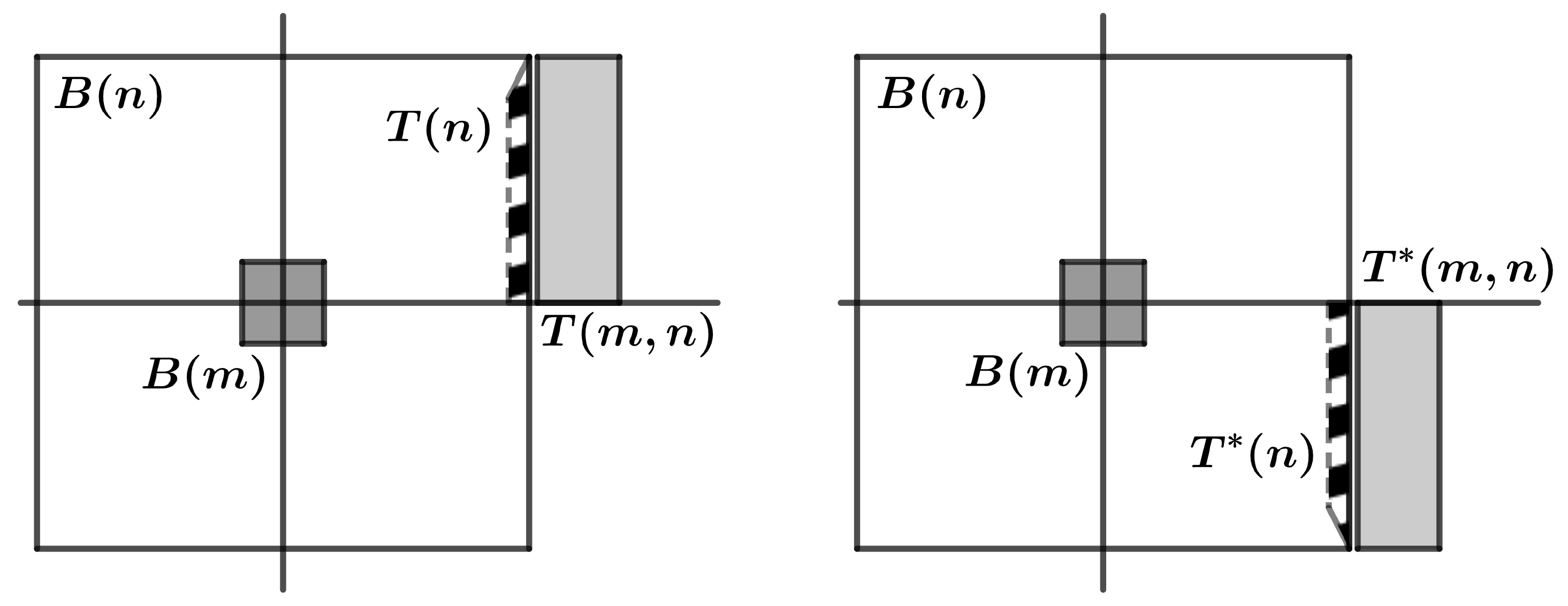}
\captionsetup{width=.9\linewidth}
\caption{Left:  sets  $T(n)$ and $T(m,n)$.  Right:  sets  $T^*(n)$ and $T^*(m,n)$. }\label{messicano1}
\end{figure}
 


%
%

  The following proposition  is the analogous of \cite[Lemma 5]{GM} in our \rui{context}. It will enter  in the proof of   Lemma \ref{pierpilori}  and in the proof of  Proposition \ref{prop_occ_origin} (which lower bounds from below the probability that in  the first step  of the renormalization scheme we enlarge the cluster of occupied sites).
   
 \begin{Proposition}\label{cinquina}
Given $\eta \in (0,1)$, there exist positive integers  $m=m(\eta)$ and $n=n(\eta)$ such that $m>2$, $2m < n$, $2m|n$ and 
\be\label{maggiolino}
 \bbP \bigl( B(m) \leftrightarrow K (m,n) \text{ in $B(n)$ for } \bbG \bigr) > 1-\eta\,.
\en
\end{Proposition}
This proposition  is a consequence of Lemma \ref{john} concerning the percolation of graph $\bbG_-$. Even if
having a seed in a specific place is an event of small probability, the big number
       of possible configurations for the seed entering in the definition of $K(m,n)$ makes the event in  \eqref{maggiolino} of high probability.
   We postpone the proof of Proposition \ref{cinquina} to Section \ref{sio5}.

\medskip
   




It is convenient to introduce the function $h_*: \D \to \bbR$ defined as
\be\label{quiquoqua}
h_*(a): =\sup_{b \in \D} h(a,b)\,.
\en
Moreover, given a finite set $R\subset \ezd$, we define  the  non--random  boundary set 
\be
\partial  R\,:=\{ y \in \ezd \setminus R \,:\, d(y,R ) \leq 1-2\a\}\,,
\en
where $d(\cdot, \cdot)$  denotes the Euclidean distance. Note that the edges of $\bbG$ have length bounded by $1-2\a$.
To avoid ambiguity, we point out that in what follows   the set $\partial  R\cap B(n)$ has to be thought of  as $(\partial  R)\cap B(n)$ and not as $\partial  (R\cap B(n))$.

The following lemma (which is the  analogous of \cite[Lemma 6]{GM} in our \rui{context})
will be crucial in estimating from below the probability to enlarge the occupied cluster   at a generic  step of the renormalization scheme (cf. Section \ref{sec_ginepro}). More specifically, it will allow to prove Proposition  \ref{prop_occ_e1} and to control the further steps in the renomalization scheme as explained in Section \ref{francia}. Recall Definition \ref{vinello} of $U_*$.
\begin{Lemma}\label{pierpilori}
Fix $\e'\in (0,1)$.  Then there exist positive integers 
$m$
and $n$,  with  $m>2$, $2m < n$ and  $2m|n$,
 satisfying  the following property.

 Consider the following sets (see Figure \ref{emma}):
\begin{itemize}
\item Let  $R $ be a finite subset of  $ \ezd$  satisfying
 \be\label{mare100}
B(m) \subset R\,,\qquad 
\left( R\cup \partial R\right)\cap \left( T(n) \cup T(m,n)\right) =\emptyset\,. 
\en  
\item For any $x\in  R\cup \partial  R$, let $\L(x)$ be a  subset of $\{1,2,\dots, K\} $. We  suppose that there exists 
 $k_*\in \{1,2,\dots, K\}$ such that 
 \be\label{monti100}
 k_*\not \in \cup_{x\in D} \L(x) \,,
 \en
where 
  $D\subset \ezd$ is  defined as 
  \be\label{campana} 
  \begin{split} D:= & \left( 
  \partial R \cap B(n) 
\right)\\&  \cup  \left\{  x\in R\,:\,  \exists y\in  \partial  R \cap B(n) 
\text{ with } |x-y| \leq 1-2\a\right\}\,.
\end{split}
\en
\end{itemize}

 Consider the following events:
 \begin{itemize}
 \item 
 Let   $H$ be any event in the  $\s$--algebra $\cF$ generated by the random variables 
$(A_x)_{x\in R \cup \partial  R}$ and  $ (T_x^{(j)})_{x \in R \cup \partial R\,, \, j\in \L(x) }$.

\item 
Let $G$ be the event that there exists  a string  $(z_0,z_1, z_2, \dots, z_\ell)$  in  $\bbV$ such that 
\begin{itemize}
\item[(P1)] $z_0\in R$;
\item[(P2)] $z_1 \in \partial R \cap B(n) $;
\item[(P3)] $z_2, \dots, z_\ell \in B(n) \setminus \bigl( R \cup \partial R\bigr)$;
\item[(P4)] $ z_2,\dots, z_\ell$ is a path in $\bbG$;
\item[(P5)] $z_\ell \in K (m,n)$;
\item[(P6)] \rrr{$T^{(k_*)} _{z_0}\in U_*$} and \rrr{$T^{(k_*)} _{z_1}\in U_*$};
\item[(P7)] $|z_0-z_1|\leq 1-2\a$;
\item[(P8)] $|z_1-z_2|\leq \rrr{h_*( A_{z_2})}-2\a$.
\end{itemize}
 
\end{itemize}
Then $\bbP( G\,|\, H) \geq 1-\e'$.
\end{Lemma}

\begin{figure}
\includegraphics[scale=0.2]{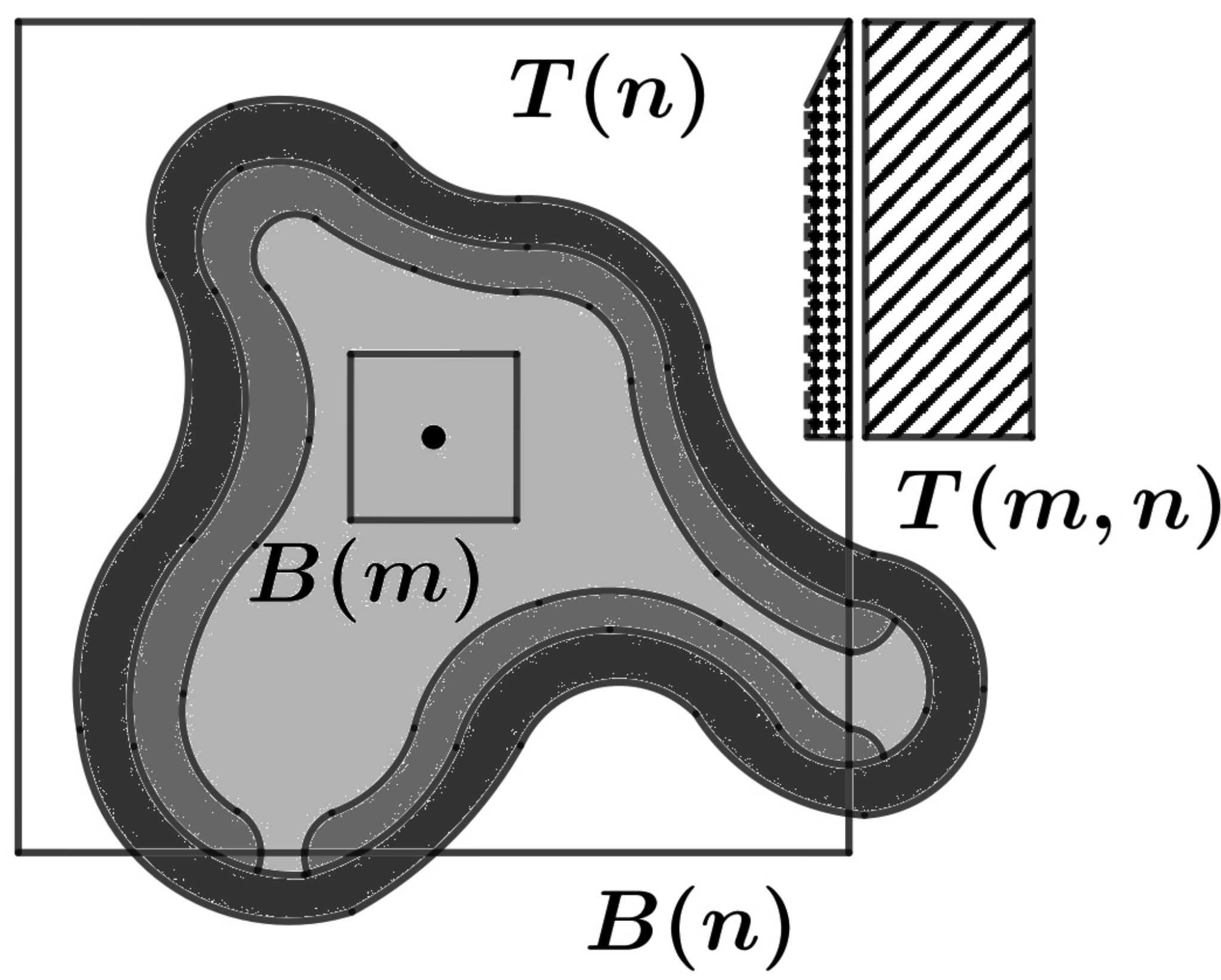}
\captionsetup{width=.9\linewidth}
\caption{$\partial R$ is the  very dark grey contour.  $R$ is given by the  light/dark grey region around the origin.  $D\setminus \bigl( \partial R \cap B(n)\bigr)$ is the dark grey subset  of $R$. }\label{emma}
\end{figure}
We postpone  the proof of Lemma \ref{pierpilori} to Section \ref{trieste65}.
We point out that   the above properties (P6), (P7), (P8) (which can appear a little exotic now) will be crucial to derive the  $\rrr{\bbG_+}$--connectivity issue stated in Lemma \ref{piccolino} \rrr{below}. Indeed, although  $(z_0,z_1, z_2, \dots, z_\ell)$ could   be not  a path in $\bbG$, one  can prove that  it is a path in $\rrr{\bbG_+}$.


\subsection{The sets \rrr{$E\bigl[C,B,i\bigr] $} and  $F\bigl[C,B,B',i\bigr]$}
\label{lenticchie}

In the next  section we will iteratively construct random subsets of $\ezd$ sharing the property to be  connected in $\bbG_+$. We  introduce here the fundamental building blocks, which are given by the sets $E\bigl[C,B,i\bigr] $ and  $F\bigl[C,B,B',i\bigr]$ (they will appear again in Definition \ref{legoland}):
  \begin{Definition}[\rui{Sets $E\bigl[C,B,i\bigr] $ and  $F\bigl[C,B,B',i\bigr]$}] \label{def_triade}
  Given three sets $C, B, B'\subset \ezd $  and given $i\in \{1,2,\dots, K\} $,  we define  the random subsets 
 $E,F\subset \ezd$ as follows:
\begin{itemize}
\item $E $ is given by the points $z_1$ in $ B \cap \partial C$ such that $T_{z_1}^{(i)}\rrr{\in U_*}$ and there exists $z_0\in C $ with $|z_0-z_1|\leq 1-2\a$ and $T_{z_0}^{(i)}\rrr{\in U_*}$;
\item $F$ is given by the points $z\in B'$ such that there exists a path $(z_2,\dots,z_k)$ inside $\bbG$ where $z_k=z$, all points $z_2,\cdots,z_k$ are in $B'\setminus (C\cup\partial  C)$ and \rrr{$|z_1-z_2|\leq h_*(A_{z_2})-2\a$} for some $z_1\in E$.
\end{itemize}
To stress the dependence from $C,B, B', i$, we will also write   $E\bigl[C,B,i\bigr] $ and  $F\bigl[C,B,B',i\bigr]$.
  \end{Definition} 
The proof of the next two lemmas is given in Section \ref{patroclo}.
  \begin{Lemma}\label{ciak2011} Let
  \rrr{$E=E\bigl[C,B,i\bigr] $} and  $F=F\bigl[C,B,B',i\bigr]$ be as in Definition \ref{def_triade}. Let $\hat E,\hat F \subset \ezd$. 
  \begin{itemize}
  \item[(i)] \rui{If the event $\{E=\hat E\}\cap \{F=\hat F\}$  \kui{occurs}, then 
      $\hat E,\hat F$    satisfy}     \be\label{chip}
    \hat E \cap \hat F=\emptyset \,, \qquad 
     \hat E\subset \bigl(  B \cap \partial C\bigr)\,, \qquad \hat F\subset B'\setminus (C\cup\partial  C)\,.
     \en 
     \item[(ii)]
     If $\hat E,\hat F$   satisfy  \eqref{chip}, then  the event $\{E=\hat E\}\cap \{F=\hat F\}$ 
     belongs to the $\s$--algebra generated by 
    \begin{itemize}
    \item[$\bullet$] $T^{(i)}_z$ with $z\in  \bigl( B \cap \partial C \bigr) \cup D$, where      \[D:=\left\{  x\in C\,:\,  \exists y\in B \cap \partial C
\text{ with } |x-y| \leq 1-2\a\right\}\,;\]
 \item[$\bullet$] $A_z$ with $z$ belonging to  some  of the following sets:
  \[ \hat F\,,  \quad  \bigl( B'\setminus (C\cup\partial  C) \bigr) \cap \partial \hat F\,,  \quad  \bigl( B'\setminus (C\cup\partial  C) \bigr) \cap     \partial \hat E\,.
\]    \end{itemize}
\item[(iii)]
As a consequence, given $R\subset \ezd$, the event $\{E\cup F= R\}$ belongs to the $\s$--algebra generated by $\{ T^{(i)}_z\,:\, z\in  \bigl( B \cap \partial C \bigr) \cup D\} \cup \{A_z\,:\, z \in R\cup \partial R\}$.
\end{itemize}
  \end{Lemma}
    
\begin{Lemma}\label{piccolino} 
Given sets  $C, B, B'\subset \ezd $  and  an index $i\in \{1,2,\dots, K\}$, we  define 
$E:= \rrr{E\bigl[C,B,i\bigr]}$ and $F:=F\bigl[C,B,B',i\bigr]$.
If $C\subset \rrr{\bbV_+}$ is  connected in  the graph   $\rrr{\bbG_+=(\bbV_+, \bbE_+)}$, then the  set $C':=C\cup E\cup F$
 is contained in $\rrr{\bbV_+}$ and is connected in  the graph   $\rrr{\bbG_+}$.
\end{Lemma}


\section{Renormalization scheme}\label{moto_GP}
  \rrr{As in Section \ref{sec_ginepro},   we  set   $N:=n+m+\e$.
From  now \rui{on}  $\e'$   is a fixed  constant in $(0,1)$ such that $1-6\e'\geq 3/4>p_c(2)$, $p_c(2)$ being the critical probability for Bernoulli site percolation on $\bbZ^2$. Moreover, we choose $m,n$ as in Lemma \ref{pierpilori}. }

\rrr{Recall Definition \ref{sambinaA} of $T(n)$, $T(m,n)$. We  define
\be \label{tittina}
T^*(m,n):= f\left(  T(m,n)\right)\text{ and } T^*(n):= f\left( T(n) \right)\,,
\en where $f:\bbR^d\to \bbR^d$ is the  isometry  $f(x_1,x_2, \dots, x_d) := (x_1,- x_2, \dots, -x_d)$ (see Figure \ref{messicano1}).
 Given $a \in \bbR  ^d$, we define  $g(\cdot|a) :  \bbR  ^d \to  \bbR  ^d$ as the isometry
\be \label{gigi} g(x|a):= ( x_1, -\text{sgn}(a_2) x_2, \dots, - \text{sgn}(a_d) x_d )\,,
\en where $\text{sgn}(\cdot)$ is the sign function, with the convention that $\text{sgn}(0)=+1$. Note that $f(x)=g(x|0)$.
}

Let $e_1, e_2, \dots, e_d$ be the canonical basis of $\bbR^d$.  
We denote by $L_1, L_2,L_3,L_4$ the isometries of $\bbR^d$ given respectively by  $\mathds{1}, \theta, \theta^2, \theta ^3$, where $\mathds{1}$ is the identity and $\theta$ is the unique rotation such that 
$\theta(e_1)= e_2$, $\theta (e_2)=-e_1$, $\theta (e_i)=e_i$ for all $i=3,\dots,d$.
We define $B_0'\subset \ezd$ as 
\begin{equation}\label{gommina}
B_0':= B(n) \cup \left( \cup_{j=1}^{4} L_j \bigl( T(m,n) \bigr) \right)\,.
\end{equation}
\rrr{Hence, for $d=2$, $B_0' $ is the region of $\ezd$ given by the largest square  and the four peripheral  rectangles  in Figure \ref{fig_C1_2}--(left).}
For $j=1,2,3,4$ we call $K^{(j)} (m,n)$ the random set of points defined  similarly to  $K  (m,n)$ 
(cf. \rui{Definition \ref{sambinaA}})
but with $T(m,n)$  and $T(n)$  replaced by $ L_j \bigl( T(m,n) \bigr) $ and $L_j\bigl( T(n)\bigr)$, respectively.

\begin{Definition}[\rui{Set $C_1$ and success-events $S_0$, $S_1$}]\label{gandalf}
We define $C_1$ as the set of points $x \in B_0'$  such that \[
\{x\} \leftrightarrow B(m) \text{ in $B_0'$ for }  \bbG  \,.\]
Furthermore, we define  the success-events $S_0$ and $S_1$ as
\begin{align*}
& S_0:=\{\,\text{$B(m)$ is a seed}\, \}\,,\\
&  S_1:=\{\,\text{$C_1$ contains a point of $ K^{(j)}(m,n)$
for each $j=1,2,3,4$}\,\}\,.
\end{align*}
\end{Definition}

    \begin{Definition}[\rui{Occupation of the origin}] \label{0occ}
  We say that  the origin $0\in \rrr{\ezd}$ is occupied if the  event  \rrr{$S_0\cap S_1 $} takes place.
\end{Definition}

We refer to Figure \ref{fig_C1_2}--(left) for an example of the set $C_1$ when $\rrr{S_0\cap S_1}$ occurs. We note that the event $S_0$ implies that $B(m) \subset \bbV$, hence $B(m) \subset C_1 $.

\begin{figure}
\includegraphics[scale=0.50]{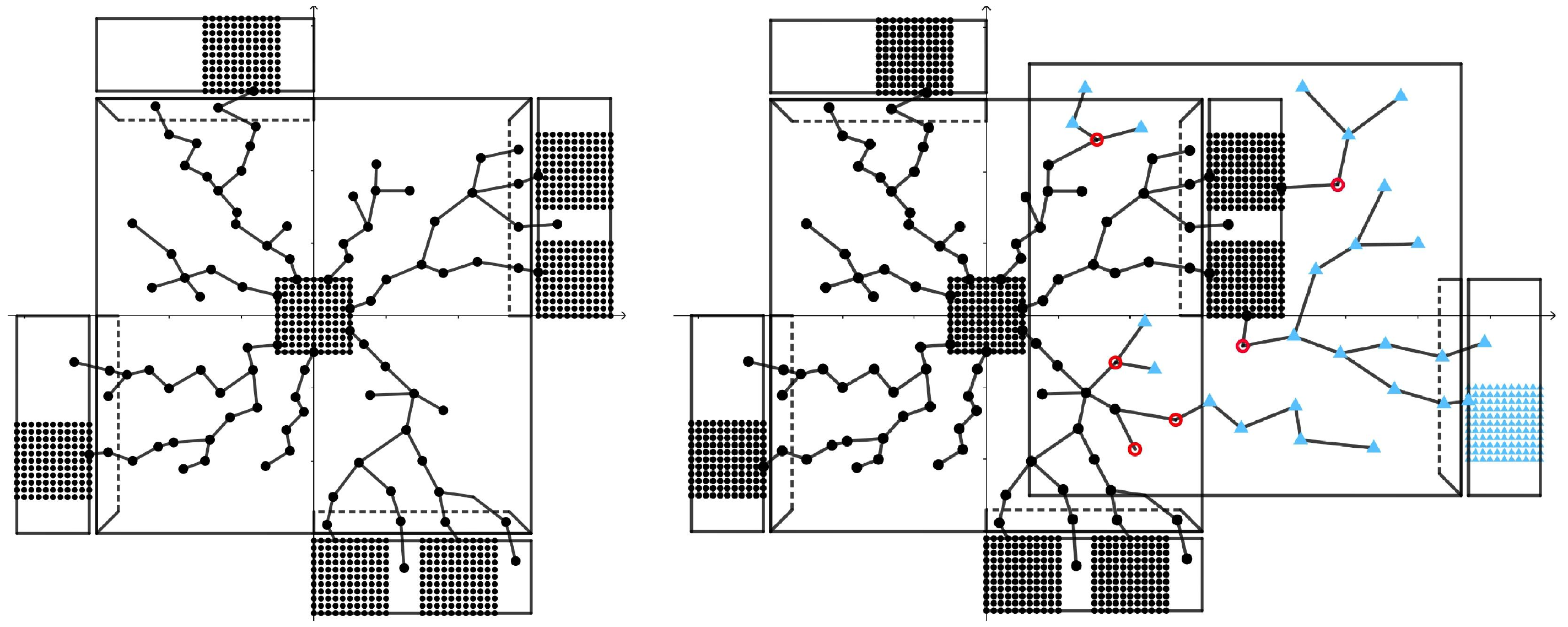}
\captionsetup{width=.9\linewidth}
\caption{Left: the set $C_1$ when  \rrr{$S_0\cap S_1$} occurs. Right: The set $C_2$ when  \rrr{$S_0\cap S_1\cap S_2$} occurs.  Points in $C_1$ correspond to circles, while points in $C_2\setminus C_1$ correspond to \rui{red rings if in $E_1$ and blue triangles if in $F_1$.}}
\label{fig_C1_2}
\end{figure}

\begin{Remark}\label{gomitolo1} If  the event $S_0$ occurs \rrr{(e.g. if the origin is occupied)}, then $C_1$ is a connected subset of $\bbG$ (and therefore of $\rrr{\bbG_+}$) by Lemma \ref{ironman}.
\end{Remark}

\rrr{When the origin is occupied}, for $i=1,2,3,4$ we define    \rrr{$c^{(i)}$} as  the minimal (w.r.t. the lexicographic order) point $z$   in $\ezd$ such that $B(z,m)$ is a seed contained in $C_1\cap L_i\bigl( T(m,n)\bigr)$.  We point out that such a seed exists by Lemma \ref{ironman} and the definition of $S_1$.   
 It is simple to check that, when $S_1$ takes place, 
 \begin{equation}\label{povo}
\rrr{ c^{(1)} _1}=N \text{ and }\rrr{c^{(1)} _j}\in [m,n-m] \text{ for }2\leq j \leq d\,,
 \end{equation}
  where \rrr{$c^{(1)} _j$}  denotes the $j$--th coordinate of \rrr{$c^{(1)}$}. Similar formulas hold for \rrr{$c^{(i)}$}, $i=2,3,4$. 
  \rrr{For later use, we set}
  \be \rrr{b^{(1)}:= c^{(1)}\,.} \en
 \begin{Proposition}\label{prop_occ_origin} 
 \rrr{It holds  $\bbP( 0\text{ is occupied}\, |\,S_0)   \geq  1-4\e'$.}
\end{Proposition}
We postpone the proof of the above proposition to Section \ref{puffo1}. \rrr{If the origin is not occupied, then we stop our construction. Hence, from now on we assume that $0$ is occupied without further mention. We fix a unitary vector, that we take equal to  $e_1$ without loss of generality, and we explain how we attempt to extend $C_1$ in the direction $e_1$.  In order to shorten the presentation, we will define geometric objects only in the successful cases relevant to continue the construction (in the other cases, the definition can be chosen arbitrarily). Figure \ref{aquile} will  be useful to locate objects. }

 Below, for \rrr{$i=2,\dots, 7 $}, we will iteratively define points $b^{(i)}$.  Moreover,
 for \rrr{$i=1,2,3$},
    we    will iteratively define sets $T_i(n)$ and $T_i(m,n)$ obtained from $T(n)$ and $T(m,n)$ by an $i$--parametrized orthogonal map. \rui{Apart from} the case \rrr{$i=4$}, many objects will be defined similarly.  Hence, we isolate some special definitions to which  we will refer in what follows. We stress that we collect these generic definitions below, but we will apply them  only  when describing the construction step by step in the next subsections. \rrr{Recall Definition \ref{def_triade}}.
    
\begin{iDefinition}[\rui{Sets $K_i(m,n)$,  $B_i' $}] \label{coca?} Given $b^{(i)}$, $T_i(n)$ and $T_i(m,n)$, 
we define   $K _i  (m,n)$ as  the set of points   $x \in   b^{(i)}+T_i(n) $ which are   \rrr{adjacent} inside $\bbG$   to  a seed contained in   $b^{(i)}+ T_i (m,n)$.  Moreover,  we define $B_i' := b^{(i)}+ \bigl( B(n)\cup T_i(m,n)\bigr)$.
\end{iDefinition}


\begin{iDefinition}[\rui{Sets $E_i,F_i,C_{i+1}$}] \label{legoland}
We set
\begin{align*}
 & E_i:=\rrr{E\bigl[ C_i,\, B\bigl( b^{(i)}, n\bigr), i]}\,, \\
 &  F_i:=F\bigl[ C_i,\,   B\bigl( b^{(i)}, n\bigr),B_i' ,i]\,,\\
 & C_{i+1}:=C_i\cup E_i \cup F_i\,.
  \end{align*}
\end{iDefinition} 

\begin{iDefinition}[\rui{Success-event $S_{i+1}$}] \label{ara}
We call $S_{i+1}$ the success-event that $C_{i+1}$  contains at least one
   vertex  in $K_i(m,n)$.
   \end{iDefinition}

\begin{iDefinition}[\rui{Property $\mathfrak{p}_i$}] \label{lattino}
We say that property $\mathfrak{p}_i$ is satisfied if  the sets 
$C_i \cup \partial C_i$ and $b^{(i) }+ \bigl (T_i(n) \cup T_i(m,n) \bigr)$ are disjoint.
\end{iDefinition}
 In several steps below we will claim  without further comments that property $\mathfrak{p}_i$ is satisfied. This property will correspond to  the second property in  \eqref{mare100} in the applications of Lemma \ref{pierpilori} in Section \ref{puffo1}, which (when not immediate)   will be checked in  Section \ref{puffo1} and   Appendixes  \ref{natale45}, \ref{natale46}. \begin{Remark}\label{aiko}
If $S_{i+1}$ occurs, then  $C_{i+1}$ contains a point $x \in   b^{(i)}+T_i(n) $ which is   \rrr{adjacent} inside $\bbG$   to  a seed $B(z,m)$ contained in   $b^{(i)}+ T_i (m,n) \subset B_i' $.
 Let us suppose  that  also  property $\mathfrak{p}_i$ in Definition \ref{lattino} is satisfied. Then $(C_i \cup E_i )\subset (C_i \cup \partial C_i)$ does not intersect  $b^{(i)}+T_i(n) $, thus implying that $x\in F_i$ \rrr{($E_i$ and $F_i$ were defined  in Definition \ref{legoland})}. Since the above seed $B(z,m)$ is contained in $b^{(i)} + T_i(m,n)$, which is contained in $B_i' \setminus (C_i\cup \partial C_i)$ due to property $\mathfrak{p}_i$, by Lemma \ref{ironman} and Definition \ref{def_triade} we conclude that $F_i \subset C_{i+1}$ contains the above seed $B(z,m)$.
\end{Remark}
  We now continue with the construction of increasing clusters and success-events. 
  
  \begin{figure}
\begin{center}
\includegraphics[scale=0.42]{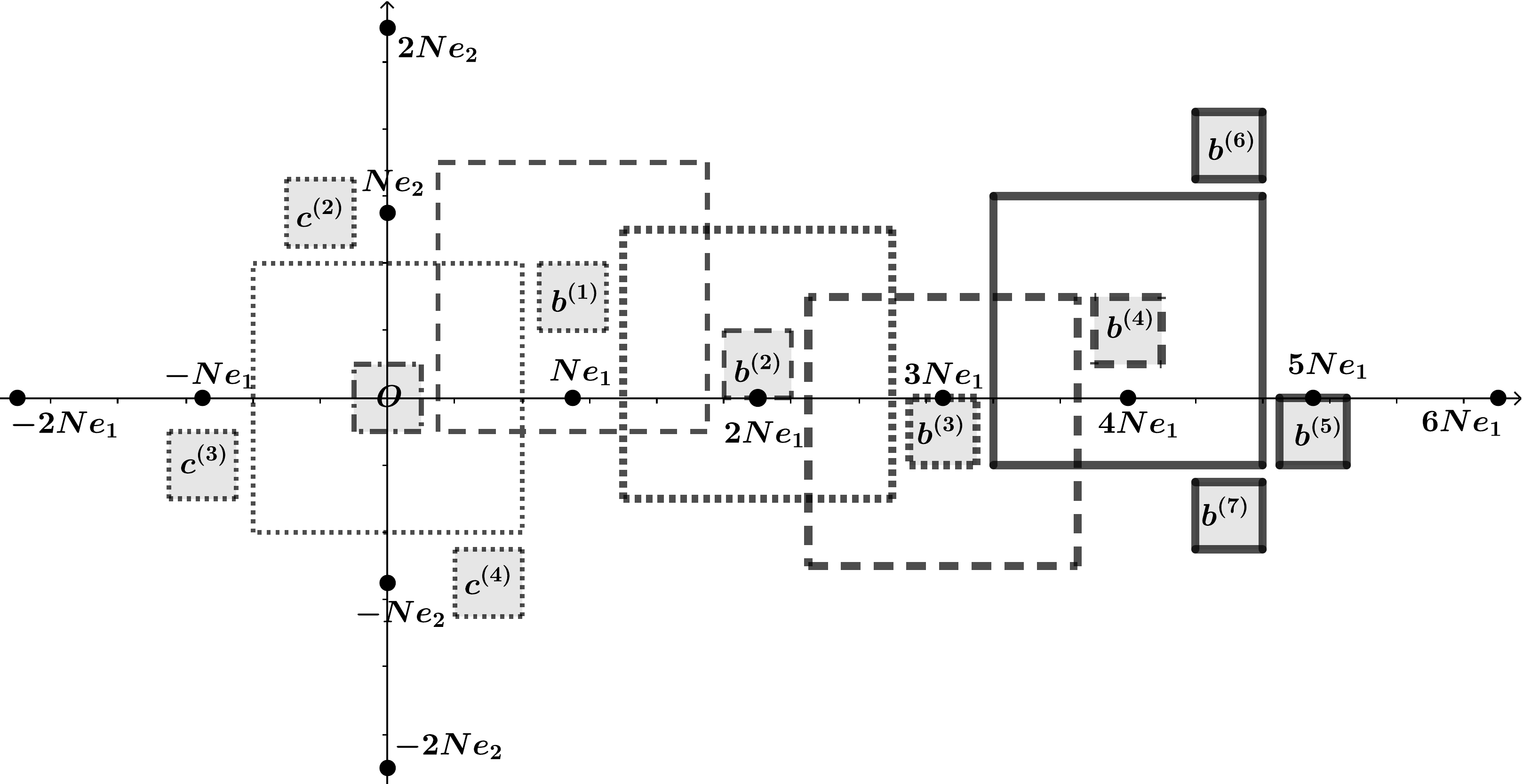}
\end{center}
\captionsetup{width=.9\linewidth}
\caption{Colored small boxes are the   seeds  $B(m)$ and $B(b^{(i)}, m)$, while bigger boxes are given by  $B(n)$ and $B(b^{(i)}, n)$. }
\label{aquile}
\end{figure}

 \subsection{\rrr{Case $i=1$}}
\rrr{
We define $T_1(n) := T^*(n) = g\bigl( T(n)| b^{(1)}\bigr) $ and $T_1 (m,n)= T^*(m,n)= g\bigl( T(m,n)|b^{(1)}\bigr) $ (cf. \eqref{tittina} and \eqref{gigi}).
 We apply  (i)--Definition   \ref{coca?},   (i)--Definition \ref{legoland}, (i)--Definition \ref{ara} and (i)--Definition \ref{lattino} for $i=1$. In particular, this defines  the sets  $K_1(m,n), B_1', E_1$, $F_1$, $C_2$  and the success-event $S_2$. See Figure \ref{fig_C1_2}--(right)}.

When  \rrr{$S_0\cap S_1\cap S_2$} occurs,   the set  \rrr{$C_{2}$ intersects the box $B(2Ne_1, N)$ as we now show. Indeed,   one can  prove  property $\mathfrak{p}_1$ using \eqref{povo}. Hence, by Remark \ref{aiko},  the  event $S_0\cap S_1\cap S_2  $ implies that $F_1 \subset C_{2}$ contains   a seed inside  $ b^{(1)}+ T^*(m,n)$  (see Figure \ref{fig_C1_2}--(right)). One can easily check that 
 $  b^{(1)}+  T^*(m,n)  \subset B(2N e_1,N)$. To this aim we observe that,  by \eqref{povo}, if $x\in  b^{(1)}+  T^*(m,n)  $,} then 
$x_1\in[2N-m,2N+m]\subset2N+[-N,N]$ and 
$x_j\in  [b_j^{(1)}-n, b_j^{(1)}]\subset[m-n,n-m]\subset[-N,N]$ for $2\leq j \leq d$.

 We  define \rrr{$b^{(2)}$} as the minimal point $z\in\e\bbZ^d$ such that $B(z,m)$ is a seed contained in $\rrr{C_{2}}\cap \bigl( b^{(1)}+T^*(m,n)\bigr)$.   By the above discussion, $b^{(1)}+T^*(m,n) \subset B(2N e_1, N)$.
By \eqref{povo} and since $\rrr{b^{(2)}_j}\in 
  [b^{(1)}_j-n,b_j^{(1)}]$ for $j\not=1$, we get   for \rrr{$i=2$}
 \begin{equation}\label{povoino}
 b^{(i)} _1=\rrr{i N}, \qquad b^{(i)}_j\in 
  [-n+m, n-m]
  \text{ for  }j\not =1\,.
 \end{equation}

 \subsection{Case \rrr{$i=2$}} We assume that  $S_0\cap S_1  \cap S_2$ occurs. 

We set 
$ \rrr{T_2}(m,n):=g\left(T(m,n)| \rrr{b^{(2)}} \right) $  and $\rrr{T_2} (n):=g\left(T(n)| \rrr{b^{(2)}}\right)$. 
   We  apply  (i)--Definition 
 \ref{coca?},   (i)--Definition \ref{legoland}, (i)--Definition \ref{ara}   and (i)--Definition \ref{lattino} for \rrr{$i=2$.} 
 In particular this defines  \rrr{$K_2(m,n), B_2',E_2,F_2, C_3, S_3$}. It is simple to check that property \rrr{$\mathfrak{p}_2$} is satisfied. If also \rrr{$S_3$} occurs, by Remark \ref{aiko} we can define 
 \rrr{$b^{(3)}$}  as the minimal point in $\e\bbZ^d$ such that $B(z,m)$ is a seed contained in \rrr{$C_3\cap\left( b^{(2)}+T_2 (m,n)\right)$}.

Let us localize some objects. Due to Claim \ref{locus2-3} in Appendix \ref{app_locus},  
 $b^{(\rrr{2})}+T_{\rrr{2}}(m,n)\subset B(3Ne_1,N)$ (see Figure~\ref{aquile}).  In particular, if $S_{\rrr{3}}$ occurs, then $C_{\rrr{3}} \cap B(3Ne_1,N)$ contains the above seed $B(z,m)$. If $S_{\rrr{3}}$ occurs,  due to \eqref{povoino} with \rrr{$i=2$,} for $j\not =1$ we have  
 \be\label{matiz}
  \begin{cases}
b_j^{(\rrr{3})}\in [b^{(\rrr{2})}_j-n+m,b_j^{(\rrr{2})}-m]\subset [-n+m,n-2m] & \text{  if $b^{(\rrr{2})}_j\geq 0$}\,,\\
b_j^{(\rrr{3})}\in [b^{(\rrr{2})}_j+m,b_j^{(\rrr{2})}+n-m]\subset [-n+2m,n-m], & \text{ if  $b^{(\rrr{2})}_j< 0$}\,.
\end{cases}
\en
Due to the above bounds and since $b^{(\rrr{3})}_1=3 N$, we get \rrr{that  \eqref{povoino} holds also for $i=\rrr{3}$}.
 
 \subsection{Case $\rrr{i=3}$} \rrr{We assume that the event $\rrr{S_0\cap  S_1\cap S_2 \cap   S_3}$ occurs.} 
 
 We define   $ T_{\rrr{3}}(m,n):=g\bigl(T(m,n)\,|\, b^{(\rrr{3})}\bigr) $ and  $T_{\rrr{3}}(n):=g\bigl(T(n)\,|\, b^{(\rrr{3})})$.
  We  apply  (i)--Definition 
 \ref{coca?},   (i)--Definition \ref{legoland}, (i)--Definition \ref{ara} and (i)--Definition \ref{lattino} for $i=\rrr{3}$. In particular, this defines  \rrr{$K_3(m,n), B_3', E_3,F_3, C_4, S_4$}.   Property \rrr{$\mathfrak{p}_3$} is satisfied. 
If also \rrr{$S_4$} occurs, by Remark \ref{aiko} we can  define $b^{(\rrr{4})}$ as minimal point in $\e\bbZ^d$ such that $B(z,m)$ is  a seed in $b^{(\rrr{3})}+ T_{\rrr{3}}(m,n)$. 

Let us localize some objects. Due to Claim \ref{locus2-3}   in Appendix \ref{app_locus},   $b^{(\rrr{3})}+T_{\rrr{3}}(m,n)\subset B(4Ne_1,N)$ (see Figure \ref{aquile}). In particular, if \rrr{$S_4$} occurs, then $\rrr{C_4} \cap B(4Ne_1,N)$ contains the above seed $B(z,m)$.   When \rrr{$S_4$} occurs,  due to \eqref{povoino} for \rrr{$i=3$} and 
reasoning as in \eqref{matiz}, we get that \eqref{povoino} holds also  for \rrr{$i=4$}.

\subsection{Case \rrr{$i=4$}}
We  assume that the event \rrr{$S_0\cap S_1\cap \cdots \cap S_4$} occurs.
The idea now is to connect the cluster $C_{\rrr{4}}$ to seeds adjacent to 
 the remaining three faces of the  cube $b^{(\rrr{4})}+B(n)$ in directions $e_1$ and $\pm e_2$  \rrr{(note that we have already the seed $B\bigl(b^{(3)},m\bigr)$ in the direction $-e_1$)}.
To this aim we set 
 \begin{align*}
 & \hat T_1(n):= g\bigl(T(n)\,|\,b^{(\rrr{4} )} \bigr) \qquad \qquad  \hat T_1(m,n):= g\bigl(T(m,n)\,|\,b^{(\rrr{4} )} \bigr)
  \\ 
      & \hat T_2(n):=( h\circ \theta) \bigl( \hat T_1(n)  \bigr) 
 \qquad \;\;\;\;\;\hat T_2(m,n):=( h\circ \theta) \bigl( \hat T_1(m,n) \bigr)
    \\
  & \hat T_3(n):= ( h\circ \theta^3) \bigl(  \hat T_1(n) \bigr)   \qquad \;\;\;\;   \hat T_3(m,n):=( h\circ \theta^3) \bigl( \hat T_1(m,n) \bigr)
  \end{align*} 
  where $\theta$ is the rotation introduced before \eqref{gommina} and the map $h$ is defined as  $h(x_1,x_2, \dots,x_n):= (|x_1|,x_2,\dots,x_n)$. 
\rrr{The use of the above sets $\hat T_j(n), \hat T_j(m,n)$  is due to the fact that we want to avoid to construct seeds in regions that have
already been explored during the constructions of the previous sets of type $C_r$.
Indeed points in the sets  $\hat T_j(n), \hat T_j(m,n)$    have first coordinate not smaller than $4N=b^{(4)}_1$ and this assures the desired property.}
%
%
  
 We also set 
 \be\label{thorin} B_{\rrr{4}}':= b^{(\rrr{4})}+\bigl(B(n)\cup \bigl[\cup_{j=1,2,3}  \hat T_j(m,n)\bigr]\bigr)\,.\en
For $j=1,2,3$ we
 call $K^{(\rrr{4}+j)}(m,n)$ the set of points $x \in b^{(\rrr{4})}+ \hat T_j(n) $ which are \rrr{adjacent} inside $\bbG$ to a seed contained in $b^{(\rrr{4})}+ \hat T_j(m,n)$ \rrr{(the sets $K^{(1)}(m,n)$,..., $K^{(4)}(m,n)$ had already been introduced to define the success-event $S_1$ in Definition \ref{gandalf})}.  \rrr{Note that we do not apply (i)--Definition \ref{coca?} for $i=4$, since the set $B_{\rrr{4}}'$ has already been introduced in \eqref{thorin} in a different form}.
We apply \rrr{here} only (i)--Definition \ref{legoland} with $\rrr{i=4}$ to define \rrr{the sets $E_4,F_4,C_5$}.

\begin{Definition}[\rui{Success-event  $S_{5}$ and points $b^{(5)}$, $b^{(6)}$, $b^{(7)}$}]
We call
  $S_{5}$ the success-event that $C_{5}$ contains at least one vertex inside $K^{(4+j)}$ for all $j=1,2,3$. When $S_{5}$ occurs,   for $j=1,2,3$ we define $b^{(4+j)}$ as  the minimal point in $\ezd$ such that $b^{(4+j)}+ B(m)$   is a seed in $C_{5} \cap \bigl( b^{(4)}+\hat T_j(m,n)\bigr)$. The  existence of such a seed can be derived by the same arguments of Remark \ref{aiko} since $C_{4}\cup \partial C_{4}$ and $b^{(4) }+\bigl (\hat T_j(n) \cup \hat T_j(m,n) \bigr)$ are disjoint for $j=1,2,3$ (as \rrr{discussed} in   Section \ref{sole88}).
 \end{Definition}

Let us localize the above objects \rrr{when also $S_5$ occurs}. Due to \eqref{povoino} for \rrr{$i=4$} and 
reasoning as in \eqref{matiz}, we get that \eqref{povoino} holds also  for $i=\rrr{5}$.  
For $i=\rrr{6,7}$ we have
\be\label{suono}
 b_1^{(i)}\in 4N+[m,n-m]\,,\quad
\begin{cases}  b_2^{(\rrr{6})} =b_2^{(\rrr{4})}+ N \,,\\
b_2^{(\rrr{7})} =b_2^{(\rrr{4})}- N \,,
\end{cases}
b^{(i)}_j \in [-n+m,n-m] \text{ for }j\geq 3\,
\en
(for $j\geq 3$ one has to argue as in \eqref{matiz}).
 Moreover,  due to     Claim \ref{locus4} in Appendix \ref{app_locus},  if 
 \rrr{$S_0\cap S_1\cap \cdots \cap S_5$}
  occurs,  then $b^{(\rrr{4})}+\hat T_1(m,n) \subset B( 5N e_1, N)$, 
$b^{(\rrr{4})}+\hat T_2(m,n) \subset B(4N e_1+ N e_2, N)$ and  $b^{(\rrr{4})}+\hat T_3(m,n) \subset B(4N e_1- Ne_2, N)$. In particular, the same inclusions hold for  the seeds $b^{(\rrr{5})}+ B(m)$,
$b^{(\rrr{6})}+ B(m)$ and $b^{(\rrr{7})}+ B(m)$, respectively. 

%
%
%
%
   \begin{Definition}[\rui{Occupation and \kui{linkage} of $e_1$}] \label{1occ}
Knowing that  the origin $0\in \rrr{\ezd} $ is occupied, we say that the site  $e_1$ is linked to $0$ and occupied  \rrr{(shortly, $0\to e_1$)}
 if \rrr{also the  event   $ \cap _{i=2}^{ 5} S_i $} takes place.
\end{Definition}

\begin{Proposition}\label{prop_occ_e1} If $0$ is occupied and  $e_1$  is linked to $0$ and occupied,  then the sets \rrr{$C_2,C_3, \dots, C_5$} are connected in $\rrr{\bbG_+}$. Moreover, 
 \be\label{soglia2} \bbP( e_1\text{ is linked to $0$ and occupied}\, |\, \text{$0$ is occupied} )   \geq \rrr{1-6 \e'}\,.
 \en
\end{Proposition}
We postpone the proof of the above proposition to Section \ref{puffo1}.

\subsection{Further comments on the construction of the occupied clusters}\label{francia}
\rrr{We start by explaining  what to do in the case of a non-success by  treating an example.  Suppose that we have a success until the definition of $C_5$:  $0$ is occupied and  $e_1$  is linked to $0$ and occupied. Suppose that, according to Tanemura's algorithm,  we want to extend $C_5$ along the first direction in order to get a cluster with a seed in the proximity of $8N e_1$. To this aim we set
$T_5(n):= g ( T(n)| b^{(5)})$ and apply (i)-Definitions \ref{coca?}, \ref{legoland}  and \ref{ara} with $i=5$ (see Fig. \ref{aquile}). This defines the sets $E_5,F_5, C_6$ and the success-event $S_6$. If $S_6$ does not occur, then we extend $C_6$ trying to develop the cluster along the route from $4N e_1$ to $4Ne_1 +4N e_2$ (if $e_2$ is the direction prescribed by Tanemura's algorithm). To do this we use the seed centered at $b^{(6)}$. We define 
\[ 
g_2 (x|a):= \left(-\text{sgn}(a_1) x_1, x_2, -\text{sgn}(a_3) x_3,\dots, -\text{sgn}(a_d) x_d\right)
\]
and set 
$T_6(n):=g_2\bigl(T(n) | b^{(6)}\bigr)$, $T_6(m,n):=g_2\bigl(T(m,n) | b^{(6)}\bigr)$.  We then  apply (i)-Definitions \ref{coca?}, \ref{legoland}  and \ref{ara} with $i=6$. This defines the sets $E_6,F_6, C_7$ and the success-event $S_7$, and  we proceed in this way.}

\rrr{In order to check the validity of Assumption (A) (cf.  Section \ref{sec_japan}) for the construction outlined in Section \ref{sec_ginepro}, one applies iteratively Lemma \ref{pierpilori} as done in the proof of Proposition \ref{prop_occ_e1}. 
Since we explore uniformly bounded regions, by taking $K$ large enough in Definition \ref{cavallo}, we can  apply iteratively  Lemma \ref{pierpilori} assuring condition \eqref{monti100} to be fulfilled simply by using some index $k_*\in \{1,2,\dots, K\}$ not already used in the region under exploration.}

\section{Proof of   \eqref{selva0}, Corollary \ref{cor1} and Corollary \ref{cor2}}\label{bin_MA}
In this section we prove  Corollaries \ref{cor1} and  \ref{cor2}. 
 Before, we state and prove the phase transition mentioned in \eqref{selva0}:
\begin{Lemma}\label{silvestro} Let $d\geq 2$, \rui{$\l>0$ and let $h(a,b)$ be  given by $h(a,b)=(a+b)^\g$ with $\g>0$, or $h(a,b)=\min(a,b)$  or  $h(a,b)=\max(a,b)$.} Consider the graph $\cG=\cG( h,\l)$ built on the  $\nu$--randomization of a  PPP on $\bbR^d$ with \rui{intensity} $\l$, where $\nu$ has bounded support and $\nu(\{0\})\not =1$.  Then there exists \rui{ $\l_c\in (0,+\infty)$}  such that \eqref{selva0} holds.
\end{Lemma}
\begin{proof} Since the  superposition of  independent PPP's is a new PPP  with   \rui{intensity} given by the sum of the original \kui{intensities}, 
by a standard coupling argument  the map $\l \mapsto P\left(  \cG( h ,\l) \text{ percolates} \right)\in [0,1] $ is non-decreasing.   By the 0-1 law, the above map takes value $0,1$.  Hence, to prove \eqref{selva0}, it is enough to prove that the above  map equals $0$  for $\l$ small and equals $1$ for $\l$ large.

We first prove that $P\left(  \cG( h ,\l) \text{ percolates} \right)=1$ for $\l$ large. To this aim, 
 we fix a positive constant $c$ such that   $ c< \sup\bigl ({\rm supp}(\nu)\bigr) $ (this is possible since $\nu(\{0\})\not =1$).
Note that  $\cG( h,\l)$ contains the subgraph $\tilde \cG$ with vertex set  $\tilde \xi :=\{ x\in \xi\,:\, E_x \geq c\}$ and edges $\{x,y\}$ with $x\not = y$ in $\tilde \xi$ such that $|x-y| \leq h(E_x,E_y)$. Given $x\not =y$ in $ \tilde \xi$ with  $|x-y| \leq \rui{h(c,c)}$,  $\{x,y\}$ is an edge of $\tilde \cG$ as $h(E_x,E_y) \geq \rui{h(c,c)}$. Hence, $\cG( h,\l)$ contains the Boolean graph model built on $\tilde \xi$ with deterministic radius 
$\left( \rui{h(c,c)}\right)/2>0$. As   $\tilde \xi$ is a PPP with \rui{intensity} $\l p$ where  $p:=\nu ( [c,+\infty) )>0$,  we get that (see \cite{MR}) there exists \rui{$\tilde\l_c\in(0,+\infty)$} such that $\tilde\cG$ percolates a.s. if $\l p>\tilde\l_c$. Hence, if \kui{$\l>\tilde\l_c/p$}, the graph $\cG(h,\l)$ percolates a.s.. 

We now prove that $P\left(  \cG( h,\l) \text{ percolates} \right)=0$ for $\l$ small. 
Since $\nu$ has bounded support, we can now fix a finite constant 
 $ c'> \sup\bigl ({\rm supp}(\nu)\bigr) $. Then 
  $\cG( h,\l)$ is contained in  the Boolean graph model built on \kui{$ \xi$} with deterministic radius 
$\left( \rui{h(c',c')}\right)/2$. 
 Since, for $\l$ small, the latter a.s. does not percolate, we get the thesis.\end{proof}

\subsection{Proof of Corollary \ref{cor1}}  Due  to Theorem  \ref{teo1} we only need to check  Assumptions (A1),...,(A5).
 Note that  Assumptions (A1), (A3) and (A5) follow  immediately from   the  hypotheses of Corollary \ref{cor1} and the definition of $h$.  As 
pointed out in Section \ref{moda},  if  $h$ is continuous, ${\rm supp}(\nu)$ is bounded  and (A5) is satisfied (as in the present setting), then (A4) is automatically satisfied by compactness. 

It remains to prove Assumption (A2). To this aim,
 we  fix $\l_1\in (\l_c, \l)$. Given $c>0$ we define
 $p_0:=\nu(\{0\})$, $p_1=p_1(c):=\nu\left( (0,c)\right)$ and 
  $p_2=p_2(c):= \nu( [c,+\infty) )$, $\nu_1:= \nu\left( \cdot | \, (0,c) \right)$ and 
  $\nu_2:= \nu\left( \cdot | \, [c,+\infty)\right)$. Then $\nu=p_0 \d_0 +p_1 \nu_1 + p_2 \nu_2$ ($\d_0$ is the standard  Dirac measure at $0$). Trivially, $\lim _{c\downarrow 0} p_1=0$ and $\lim _{c\downarrow 0} p_2=1-p_0$. 
 We  choose  $c>0$ small enough to have  $\frac{1-p_0}{p_2} \l_1<\l$.

 Call $\hat \cG= \cG( \l_1,h;\mu)$ the graph with structural function
 \rui{$h$} 
 built on the $\mu$--randomization of a PPP $\xi$ with \rui{intensity} $\l_1$,
 where 
 \[ \mu:=p_0\d_0+ (1-p_0) \nu_2=p_0\d_0+ (p_1+p_2) \nu_2\,. \]
We have that $\mu $ stochastically dominates $\nu$ since, for all $a\geq 0$, it holds
 \[
\mu\left( [a, +\infty)\right) =  
\begin{cases}
1= \nu \left( [a, +\infty)\right)
 & \text{ if } a=0\,,\\
 1-p_0\geq  \nu \left( [a, +\infty)\right) &  \text{ if } 0<a<c  \,, \\
 (1- p_0) \frac{ \nu \left( [a, +\infty)\right)}{ 1-p_0-p_1}\geq  \nu \left( [a, +\infty)\right)  & \text{ if } a\geq c\,.
  \end{cases}
  \]
  Due to the above stochastic domination there exists a coupling between the $\mu$-randomization of the  PPP $\xi$ with \rui{intensity} $\l_1$ and the $\nu$-randomization of the  PPP $\xi$ with \rui{intensity} $\l_1$ such that the  \rui{marks} in the former   are larger than or equal to  the   \rui{marks} in the latter.
    As $\cG(h, \l_1)$ percolates a.s. since $\l_1>\l_c$,   by the above  coupling and   since $h(\cdot, \cdot)$ is jointly increasing,  $\hat \cG$ percolates a.s..
    
    Trivially, given $\rho\in(0,1) $, $\hat \cG$ can be described also as  the graph with vertex set $\xi$  as above and edge set   given by the unordered pairs $\{x,y\}\subset \xi $ with $x\not=y$   and 
\be  \label{bagnoli1}
 \bigl | \rho x  -\rho y  \bigr |  
\leq 
\rui{\rho h(E_x,E_y) = h(E_x,E_y)- (1-\rho) h(E_x,E_y) }\,, 
  \en 
  where  the marks $E_x$'s have law $\mu$.
Note that,  as $x\not =y$, at least one between the marks $E_x,E_y$ is nonzero \rui{(and both of them are non zero if $h(a,b)=\min(a,b)$)} and therefore lower bounded by $c$ \rui{a.s. (by definition of $\mu$)}.
 This implies that   
\rui{$(1-\rho) h(E_x,E_y) \geq (1-\rho) c^\g=: \ell_*$ if $h(a,b)=(a+b)^\g$ and $(1-\rho) h(E_x,E_y) \geq (1-\rho) c=: \ell_*$ if $h(a,b)= \min(a,b)$ or $h(a,b)=\max(a,b)$. At this point,}
 by applying the  map $x\mapsto \rho x$, we get that the image of $\hat \cG $ is contained in the graph $\bar \cG:=\cG(h-\ell_*, \l_2; \mu )  $ with structural function $h-\ell_*$ built on the $\mu$--randomization of a PPP with \rui{intensity}  $ \l_2:=\l_1  \rho^{-d}$. As $\hat \cG$ percolates a.s., the same holds for its $\rho$-rescaling contained in $\bar \cG$. This proves that $\bar \cG  $  percolates a.s..

 Recall that $\frac{1-p_0}{p_2} \l_1<\l$. We now choose $\rho $ very near to $1$ to have 
$\l_*:=\frac{1-p_0}{p_2}\l_2=\frac{1-p_0}{p_2} \l_1  \rho^{-d}$ smaller than $\l$.
To conclude with (A2) we only need to show that $\cG(h-\ell_*, \l_*)$ percolates a.s.. To this aim we observe that
\[
\nu=\frac{p_2}{p_1+p_2}\mu + \frac{1-p_2}{p_1+p_2}\bar \mu\,, \qquad \bar \mu:=\frac{p_1}{1-p_2}  p_0 \d_0 + \frac{1-p_0}{1-p_2}p_1 \nu_1\,.
\]
Note that $\nu$ is a convex combination of the probability measures $\mu$ and $\bar \mu$.
Hence,
 the marked vertex set of  $\cG(h-\ell_*, \l_*)$, which is the $\nu$--randomization of a PPP with \rui{intensity} $\l_*$, can be obtained as superposition of  two independent marked point processes given by   
 the $\mu$--randomization of a  PPP with \rui{intensity} $\frac{p_2}{p_1+p_2} \l_*=\l_2$ and the $\bar \mu $-randomization of a PPP with \rui{intensity} $\frac{1-p_2}{p_1+p_2} \l_*$.
The subgraph of  $\cG(h-\ell_*, \l_*)$ given by the points in the first marked PPP and their edges has the same law of $\bar  \cG$. As  $\bar  \cG$  percolates a.s. we get that $\cG(h-\ell_*, \l_*)$  percolates a.s., i.e. (A2) is satisfied. 

\subsection{Proof of Corollary \ref{cor2}} Due  to Theorem  \ref{teo1} we only need to check that Assumptions (A1),...,(A5) are satisfied.

Assumptions (A1) and (A3) follow  immediately from   the hypothesis of the corollary. Also  Assumption (A5) is trivial as   $h(a,b)=\z- 2\max\{ |a|,|b|\}$ for $ab\geq 0$. As 
pointed out in Section \ref{moda},  since  $h$ is continuous, ${\rm supp}(\nu)$ is bounded  and (A5) is satisfied, then (A4) is automatically satisfied by compactness.

 Let us prove Assumption  (A2).  
To this this aim  we fix $\ell>0$ such that  $\z-\ell>\z_c$ (this is possible as $\z>\z_c$). Then, by \eqref{selva}, $\cG(h-\ell ,\l )$ percolates a.s..
  Moreover, 
given $\g \in (0,1)$, $\cG(h-\ell ,\l )$ can be described also as  the graph with vertex set $\xi$ given by a PPP with \rui{intensity} $\l$ and edge set    given by the pairs $\{x,y\}\subset \xi $ with $x\not=y$   and 
\be  \label{bagnoli}
 \Big | \frac{x}{\g} -\frac{ y}{\g} \Big |  
\leq 
 \frac{\z-\ell}{\g}-\frac{1}{\g}  ( |E_x|+ |E_y| +|E_x-E_y|)
 \,,
 \en 
where the marks come from the $\nu$-randomization of the PPP $\xi$.
Note that the r.h.s. is upper bounded by $(\z-\ell)/\g-  ( |E_x|+ |E_y| +|E_x-E_y|)$. 

We now  fix $\g$ very near to $1$ (from the left) to have   $(\z-\ell) /\g<\z$. Hence, we can write $(\z-\ell) /\g=\z-\ell_*$ for some $\ell_*>0$.
Due to the above observations, if $\{x,y\}$ is an edge of $\cG(h-\ell ,\l )$, then $|x/\g-y/\g|\leq \z-\ell_*- ( |E_x|+ |E_y| +|E_x-E_y|)$. In other words, the graph $\cG(h-\ell ,\l )$ is included in the new  graph $\hat \cG= (\xi , \hat \cE)$, with edge  set $\hat \cE$  given by the pairs  $\{x,y\}\subset \xi $ with $x\not=y$   and 
\be \label{pompei}
\bigl| x/\g - y/\g \bigr | \leq \z-\ell_*-  ( |E_x|+ |E_y| +|E_x-E_y|)  \,.
\en
As already observed  $\cG(h-\ell ,\l )$ percolates a.s., thus implying that $\hat \cG$ percolates a.s..
 On the other hand, due to \eqref{pompei}, the graph  obtained by spatially rescaling  $\hat \cG$ according to the map $x\mapsto x/\g$ 
 has the same law of the graph  $\cG( h-\ell_*, \l \g^d )$.  Hence, the latter percolates a.s..  To conclude the derivation of (A2) it is enough to take $\l_*:=  \l \g^d$.
  


\section{Proof of Proposition  \ref{cinquina}}\label{sio5}
Recall Definition  \ref{sambinaA}.

\begin{Definition}[\rui{Sets $A(n)$ and $T_{\s,J}(n)$}] \label{sambina}
For $m \leq n\in\bbN_+$, $z \in \ezd$, $\s \in \{-1,1\}^d$, $J\in \{1,2,\dots, d\}$     we define the following sets \rrr{(see Figure
 \ref{protex})}
\begin{align*}
& A(n):=\{ x\in \ezd\,:\, n-1 < \|x\|_\infty \leq n\}\,,\\
& T_{\s,J}(n):=\{ x\in \ezd : n-1 < \|x\|_\infty \leq n, \,0 \leq \s_i x_i \leq \s_J  x_J \; \forall i=1,2,\dots, d\}
 \,.
\end{align*}
 \end{Definition} 
 
 \begin{figure}
\includegraphics[scale=0.35]{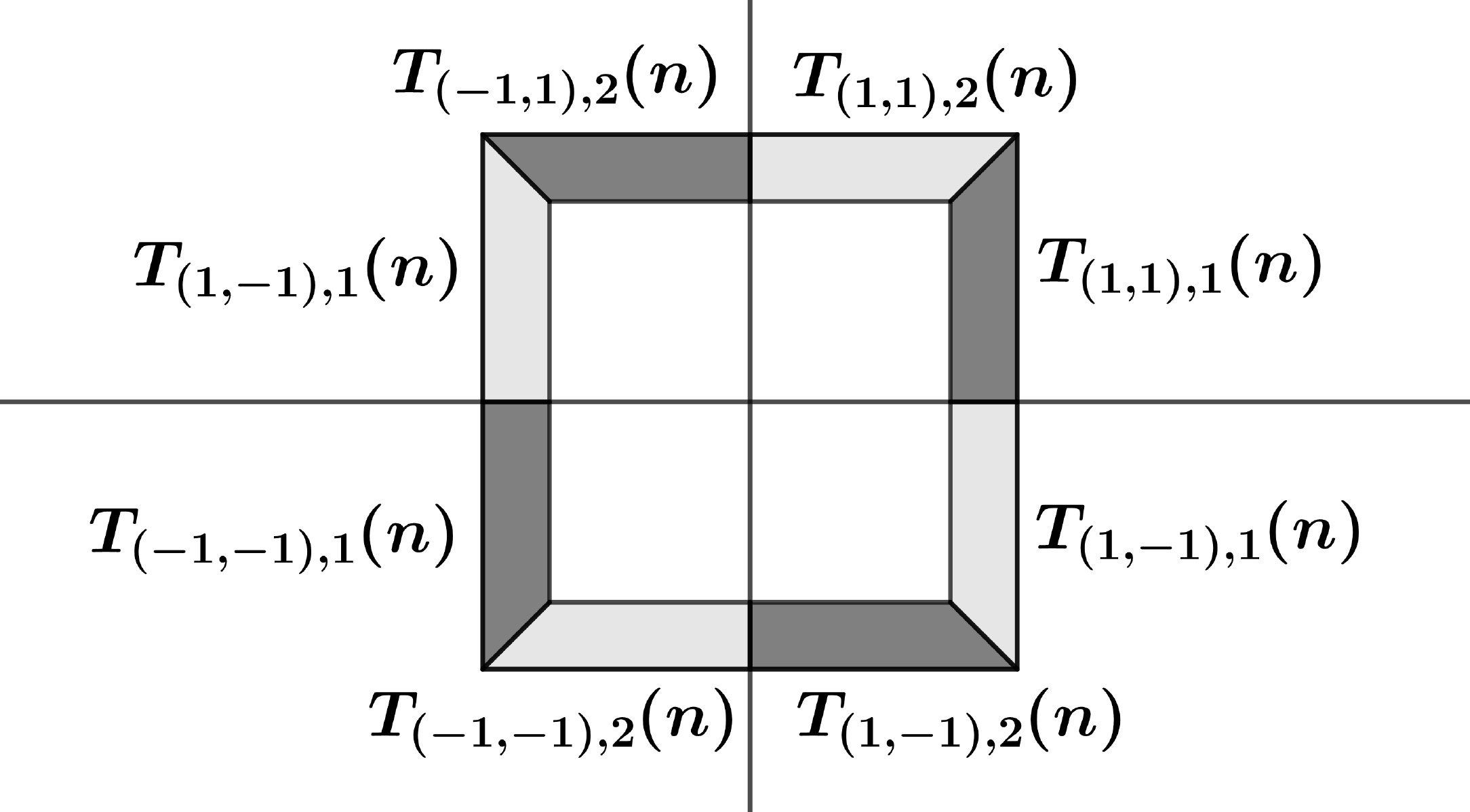}
\captionsetup{width=.9\linewidth}
\caption{The sets $T_{\s,J}$ for $d=2$. The largest square has radius $n$ and the smallest one has radius $n-1$. The annulus $A(n)$ is the union of the $T_{\s,J}$'s, hence it corresponds to the (dark or light) grey region. } 
\label{protex}
\end{figure}

Note that $T_{\underline{1},1}(n)= T(n)$, where 
 $\underline{1}:=(1,1,\dots, 1)$.  The following fact can be easily checked (hence we omit its proof):

%
%

 \begin{Lemma}\label{euclide} 
 We have the following properties:
 \begin{itemize}
  \item[(i)] 
 $
 A(n)= \cup _{\s\in \{-1,1\}^d}\cup_{J=1}^d 
 T_{\s,J}(n)$;
 \item[(ii)]  given $(\s,J)$ the map
$ \psi_{\s,J}( x_1,x_2, \dots, x_d):=(y_1,y_2, \dots, y_d)$,
 where 
 \[
 y_k:= 
 \begin{cases}
  x_J \s_1 \text{ if } k=1 \,,\\
 x_1 \s_J  \text{ if } k=J\,,  \\
  x_k \s_k  \text{ otherwise}\,, \end{cases}
 \]
 is an isometry from $T(n)$ to $T_{\s,J}(n)$ and it  is the identity when $\s = \underline{1}$ and $J=1$.
 \end{itemize}
 \end{Lemma}
 
 The following lemma and its proof are inspired by \cite[Lemma 3]{GM} and its proof. \rui{To state it properly we introduce the set $U_n$ (recall the constant $\l_*$ introduced in  (A2), the definition of $h_*$ in \eqref{quiquoqua} and that $d(\cdot,\cdot)$ denotes the Euclidean distance):
 \begin{Definition}[Set $U_n$]
  Given \kui{positive integers} $n>m$, we denote by 
$U_n$  the set of  points $x\in A(n)$ such that  $B(m) \lrgh x$ in $ B(n)$ for  $ \bbG_-$ and 
\be\label{rotta}
d\bigl( x, B(n)^c \bigr)  \leq h_*(A_x)-3\a\,.
\en 
 \end{Definition}}
\begin{Lemma}\label{analogo3} Let  $m $ and $ n$ be positive integers such that  $n>m$. Then, for each integer $k$, it holds
\begin{equation}\label{LP}
\sum_{n=m+1}^\infty \bbP(|U_n|<k,\,B(m) \lrgh \infty\text{ for } \rrr{\bbG_-})<e^{  c(d) \l_* k}
\end{equation}
for a positive constant $c(d)$ depending only on the dimension.
\end{Lemma}
\begin{proof}   We claim that   the event $\{B(m) \lrgh \infty\text{ for } \rrr{\bbG_-}\}$ implies that   $|U_n|\geq 1$. 
To prove our claim we observe that, 
 since the edges in $\rrr{\bbG_-}$ have length at most  $1-3\a $ \rrr{(see Warning \ref{aaah})}, 
the event $\{B(m) \lrgh \infty\text{ for } \rrr{\bbG_-} \}$ implies that there exists $x \in A(n) $ such that $B(m) \lrgh x \text{ in } B(n) $ for $ \rrr{\bbG_-}$  and  $\{x , y\}\in \rrr{\bbE_-}$ for some $y\in   B(n)^c\cap \rrr{\bbV_-}$. Indeed, it is enough to take 
any path from $B(m)$ to $\infty$ for $ \rrr{\bbG_-} $ and define $y$ as the first visited point in $B(n)^c$ and $x$ as the 
point visited before $y$. Note that the  property  $\{x,y\}\in \rrr{\bbE_-} $ implies \eqref{rotta}\rrr{.}  Hence $x\in U_n$. This concludes the proof of our claim.
Due to the above claim   we have 
\be\label{torino1}
\bbP(|U_n|<k,\,B(m) \lrgh \infty\text{ for } \rrr{\bbG_-}) \leq \bbP ( 1\leq |U_n | < k) \,.
\en

\smallskip
We now want to estimate  $\bbP ( U_{n+1}=\emptyset\,|\, 1\leq |U_n | <k)$ from below (the result will be  given in  \eqref{torino2} below).

For each $x \in U_n$ we denote by 
$I_{n+1}(x)$ the set of points   $y$  in $A(n+1)$ such that $|x-y|\leq 1-3\a$. 
We call $G_n$ the event that  $\rrr{\bbV_-}$ has no points in $\cup _{x \in U_n} I_{n+1}(x)$. 
 We now claim that  $G_{n} \subset \{U_{n+1}=\emptyset\}$. To prove our claim let $z$ be in 
  $ U_{n+1}$.  Then there is a path in  $\rrr{\bbG_-}$ from  $z$ to some  point in $B(m)$ visiting only points in  $B(n+1)$. We call $v$ the last point in the path inside $A(n+1)$  and $x$ the next point in the path. Then  $x\in A(n)$ and  all the points visited by the path after $x$ are in $B(n)$. Hence,  $B(m) \lrgh x$ in $ B(n)$ for $ \rrr{\bbG_-}$.  Moreover, since $\{ x, v\}\in \rrr{\bbE_-}$, property \eqref{rotta} is verified. Then $x\in U_n$ and $\rrr{\bbV_-} $ has some point (indeed $v$) in $I_{n+1}(x)$. 
 In particular, we have shown that, if $U_{n+1}\not=\emptyset$, then $G_n$ does not occur, thus proving our claim.

 Recall that 
  the graph $\rrr{\bbG_-}$ depends only on the random field  $(A_z)_{z\in \ezd}$ and that $\bbP(\rrr{A_z\not \in \bbR})= e^{-\l_* \e^d}$ for any $z\in \ezd$.
We call $\cF_n$ 
 the $\s$--algebra 
 generated by the random variables $A_z$ with  $z\in B(n)$.
   Note that the set   $\cup_{x\in U_n} I_{n+1}(x)$
    and the event  $\{1\leq |U_n |< k\}$ are $\cF_n$--measurable. Moreover, 
   on the event $\{1\leq |U_n |< k\}$, the set  $\cup_{x\in U_n} I_{n+1}(x)$  has cardinality  bounded by $c(d) k \e^{-d} $, where $c(d)$ is a positive constant depending only on $d$. By the independence of the $A_z$'s  we conclude that  that $\bbP$--a.s. on the event  $\{1\leq |U_n | < k\}$ it holds
\be\label{girino}
\begin{split}
\bbP ( G_n\,| \, \cF_n)& =\bbP( \rrr{A_z\not\in \bbR} \; \forall z\in   \cup_{x\in U_n}  I_{n+1}(x) \,|\, \cF_n)
\\
& \geq \bbP(\rrr{A_0\not\in \bbR})^{c(d) k \e^{-d}}=   e^{- c(d) \l_* k } \,.
\end{split}\en
Hence, since $G_n \subset \{ U_{n+1}=\emptyset\}$, by \eqref{girino} we conclude that 
\be\label{torino2}
\bbP ( U_{n+1}=\emptyset\,|\, 1\leq |U_n | < k) \geq \bbP ( G_n \,|\, 1\leq |U_n |< k) \geq \exp \{ - c(d)  \l_* k\}\,.
\en
As a byproduct of \eqref{torino1} and \eqref{torino2} we get 
\begin{multline}
e^{ - c(d)  \l_*  k}  \bbP(|U_n|<k,\,B(m) \lrgh \infty\text{ for }\rrr{\bbG_-}) 
\leq e^{ -  c(d)  \l_*   k}  \bbP ( 1\leq |U_n | < k)\\
\leq
\bbP ( U_{n+1}=\emptyset\,|\, 1\leq |U_n | < k)  \bbP ( 1\leq |U_n | < k)\\
=\bbP ( U_{n+1}=\emptyset\,,\, 1\leq |U_n | <k)  \,.
\end{multline}
Since the events $\{ U_{n+1}=\emptyset\,,\, 1\leq |U_n |<k\}$ are disjoint, we get \eqref{LP}.
\end{proof}
We now present the analogous of \cite[Lemma 4]{GM}.  \rui{To this aim we introduce the set $V_n$:
\begin{Definition}[Set $V_n$] Given \kui{positive integers} $n>m$, we  call $V_n$ 
 the set of points $x\in T(n)$ satisfying \eqref{rotta} and  such that $B(m) \lrgh x$  in $ B(n)$ for $ \bbG_-$.
\end{Definition}}
\begin{Lemma}\label{analogo4} Let $w:= 2^d d$. Then, for any $\ell \in \bbN$, it holds 
\be\label{minerva}
\liminf _{n \to \infty}\bbP( |V_n | \geq \ell ) \geq 1- \bbP (B(m)\not \lrgh \infty \text{ for } \rrr{\bbG_-}  )^{1/w }  \,.
\en
\end{Lemma}
\begin{proof} 
Let $\s,J$ be as in Definition \ref{sambina}.
If in the definition of $V_n$  we take  $T_{\s, J}(n) $  instead of $T(n) $,  then we call $V_{\s,J,n}$ the resulting  set.  Note that $V_{\underline{1}, 1,n}=V_n$.
%
By Lemma \ref{euclide}--(i)  we get that $|U_n| \leq \sum _{(\s,J) } |V_{\s,J,n}|$, hence
\be
\{ |U_n| < w \ell\} \supset \cap _{(\s,J)} \{ |V_{\s, J,n}  | < \ell \} \,.
\en
 By the FKG inequality \rrr{(cf. Section \ref{scremato})} and since each event $\{ |V_{\s, J,n}  | < \ell \} $ is decreasing, and by the isometries given in Lemma \ref{euclide}--(ii), we have \[ \bbP (  |U_n| < w \ell ) \geq \prod _{(\s,J)} \bbP( |V_{\s, J,n} | < \ell )=\bbP( |V _n | < \ell )^w\,. \]
The above bound implies that $ \bbP( |V_n | \geq \ell )\geq 1- \bbP ( |U_n|< w \ell ) ^{1/w}$.
On the other hand we have
\be
\begin{split} \bbP ( |U_n| < w\ell) &  \leq  \bbP ( |U_n| < w\ell\,,\, B(m) \lrgh \infty \text{ for } \rrr{\bbG_-} )\\
&+ \bbP (B(m)\not \lrgh \infty \text{ for } \rrr{\bbG_-})
\end{split}
\en
and by Lemma \ref{analogo3} the first term in the r.h.s. goes to zero as $n \to \infty$, thus implying the thesis.
\end{proof}

We can finally give the proof of Proposition \ref{cinquina}. 
\begin{proof}[Proof of Proposition \ref{cinquina}]
By Lemma \ref{john} $\rrr{\bbG_-}$  percolates $\bbP$--a.s., hence we can fix  an integer $m>2$ such that  
\be \label{bracciano}
\bbP( B(m) \not \lrgh \infty \text{ for } \rrr{\bbG_-} )< (\eta/2)^w\,, \qquad w:=d2^d\,.
\en
 Then, by Lemma \ref{analogo4}, 
 for any $\ell \in \bbN$ we have 
\be\label{crimine} \liminf _{n \to \infty}\bbP( |V_n | \geq \ell ) \geq 1- \bbP (B(m)\not \lrgh \infty \text{ for } \rrr{\bbG_-} )^{1/w} > 1-\eta/2\,.
\en

We set $\rho := \bbP ( B(m) \text{ is  a seed}) \in (0,1)$ and fix an integer $M$   large enough that $(1-\rho)^M<\eta/2$. 
 We set $\ell:= (2m)^{d-1} 3^{d-1}M\e^{-d}$ and, by \eqref{crimine}, we can fix $n$
large enough that $\bbP( |V_n | \geq \ell ) >1-\eta/2$, $2m<n$ and $2m|n$.

\rui{The main idea behind the proof is the following: by the above choice of constants and since points in $\ezd$ have distance  at least  $\e$, points in  $V_n\subset T(n)$  must be enough spread that with high probability some point $x\in  V_n$ is in the proximity of a seed $S$ contained in the slice $\{ z\in \ezd\,:\, n+\e\leq z_1 \leq n+\e+2m\}$. Then,  since $x $ satisfies \eqref{rotta}, we will show that  $x$ must be  adjacent inside $\bbG$ to the seed $S$ and  hence $x\in K(m,n)$. Using  that 
$B(m) \lrgh x$  in $ B(n)$ for $ \bbG_-\subset \bbG$, we will then conclude that $B(m) \leftrightarrow K (m,n)$ in $B(n)$ for $\bbG$. }

\rui{Let us implement the above scheme}.
Since $2m|n$ we can partition $[0,n]^{d-1}$ in  non--overlapping $(d-1)$--dimensional closed boxes $D_i^*$, $i \in \cI$,  of side length $2m$ (by ``non--overlapping'' we mean  that the interior parts are disjoint). We set $D_i:= D_i^*\cap \ezd $. Note that
$T(n) \subset  \cup _{i \in \cI} (n-1,n]\times D_i$ and  $T(m,n) = \cup _{i \in \cI} ([n+\e, n+\e+2m]\cap \e \bbZ) \times D_i$. 
%

By construction, any set $(n-1,n]\times D_i$ contains  at most $(2m)^{d-1} \e^{-d}$ points $x\in T(n)$.  Since  $\ell= (2m)^{d-1} 3^{d-1}M\e^{-d}$, the event $\{|V_n | \geq \ell\}$ implies that there exists $\cI_*\subset \cI$ with $|\cI_*| =3^{d-1}M$  fulfilling the following property: 
for any $k\in \cI_*$ there exists $x\in V_n$ with $x\in(n-1,n]\times D_k$.  We can choose univocally  $\cI_*$   by defining it as the set of  the first (w.r.t.  the lexicographic order) $M$ indexes $k\in \cI $ satisfying the above property. We now thin $\cI_*$ since we want to deal with disjoint sets $D_k$'s. To this aim we observe   that each $D_k$ can intersect at most $3^{d-1}-1$ other sets of the form $D_{k'}$. Hence, there must exists $\cI_\natural \subset \cI_*$ such that $D_k\cap D_{k'}=\emptyset $ for any $k\not = k'$ in $\cI_\natural$ and such that $|\cI_\natural| = M$ (again $\cI_\natural$ can be fixed deterministically by using the lexicographic order). We introduce the events
\be\label{schiavo} G_k:=\{ ([n+\e,n+\e+ 2m] \cap \e \bbZ) \times D_k   \text{ is a seed} \}\,.\en
We claim that 
\be\label{cedric}
\bbP \left(\, \{ |V_n| \geq \ell \} \cap  (\cup_{ k \in \cI_\natural} G_k)  \,\right) \geq 1-\eta\,.
\en
\rui{Before  proving our claim we show that the event in  \eqref{cedric} implies the event in \eqref{maggiolino}, thus allowing to conclude. Hence, 
let} us now suppose that $ |V_n | \geq \ell $ and that the event $ G_k $ takes place for some $k \in \cI_\natural$. We claim that necessarily 
$B(m) \leftrightarrow K  (m,n)$  in $B(n)$ for $\bbG$.  Note that the above claim and \eqref{cedric} lead to \eqref{maggiolino}.
We prove our claim. As discussed before \eqref{schiavo}, since $k\in \cI_\natural$  there exists  $x\in V_n\cap( (n-1,n]\times D_k)$.
Let $S $ be the   seed $([n+\e,n+\e+ 2m] \cap \e \bbZ) \times D_k $.
By definition of $V_n$, \rrr{$ d\bigl( x, B(n)^c \bigr) \leq h_*(A_x) -3\a$}  and $B(m) \lrgh x$ in $ B(n)$ for $ \rrr{\bbG_-}$.
  Call $x'$ the point in  $\partial B(n)$
  such that $|x-x'| = d\bigl(x, B(n)^c\bigr)$.  Note that $x'\in \{n\}\times D_k $ as $V_n \subset T(n)$. Let $y:=x'+\e e_1$. Then $y\in S$ and therefore  \rrr{$A_y \in U_*=U_*(\a/2)$} (as $S$ is a seed) and $|x'-y|=\e  \leq  \a/100$ \rrr{(cf. Definition \ref{vinello})}. 
  Then, using \eqref{indigestione} with $\d=\a/2$, the symmetry of $h$ and that $A_y \in U_*$,  we get $h_*(A_x)=\sup _{a \in \D} h(A_x,a ) \leq h (A_x,A_y)+\a/2$. Hence, we obtain
 \begin{multline}
       |x-y|  \leq |x-x'| +|x'-y| 
     \leq d\bigl(x, B(n)^c\bigr)+ \a/100 \leq  h_*(A_x)-3\a+ \a/100 \\
     \leq h(A_x,A_y)+ \a/2 -3\a+\a/100 \leq  h(A_x,A_y)-2\a\,. 
     \end{multline}
    We have therefore shown that  $B(m) \lrgh x$ in $ B(n)$ for $ \rrr{\bbG_-}$ for some $x\in T(n)$ with  $\{x,y\}\in \bbE$ for some  $y \in S$ \rrr{(cf. Definition \ref{vichinghi})}. As a consequence, $x\in K(m,n)$. Since $  \rrr{\bbG_-}  \subset \bbG$, we get that  $B(m) \leftrightarrow K  (m,n)$    in $B(n)$ for $ \bbG$.

\rui{It remains now to prove \eqref{cedric}}. 
To this aim we call $\cF_n$ the $\s$--algebra 
 generated by the r.v.'s $A_z$ with   $z\in B(n)$. We observe that the event $\{  |V_n | \geq \ell \}$ belongs to $\cF_n$,   the set $\cI_\natural $ is $\cF_n$--measurable and  w.r.t. $\bbP(\cdot| \cF_n)$ the events  $\{ G_k:k \in \cI_\natural \}$ are independent (recall that  $D_k\cap D_{k'}=\emptyset $ for any $k\not = k'$ in $\cI_\natural$)
 and each $G_k$ has  probability $\rho:= \bbP ( B(m) \text{ is  a seed})$. Hence,  $\bbP$--a.s. on the event $\{ |V_n | \geq \ell\}$ we can bound 
\be\label{curdo}
 \bbP( \cup_{  k \in \cI_\natural} G_k \,|\, \cF_n) \geq 1-(1-\rho)^M> 1-\eta/2\,. \en Note that the last bound follows from our choice of $M$.
Since, by our choice of $n$, $\bbP( |V_n | \geq \ell ) >1-\eta/2$, we conclude that 
the l.h.s. of \eqref{cedric} is lower bounded by $(1-\eta/2)^2>1-\eta$. This concludes the proof of \eqref{cedric}.
 \end{proof}


\section{Proof of Lemma \ref{pierpilori}}\label{trieste65} 
\rrr{Note that $ \g:= \bbP(T^{(j)}_0\in U_*)>0$ due to  Assumption (A4) and Definition \ref{cavallo}}.
We can fix a positive constant $c(d)$ such that  the ball $\{y \in \bbR^d\,:\, |y| \leq 2\}$  contains at most $c(d) \e^{-d}$ points of $\ezd$.
We then choose $t$ large enough that $ 
 (1- \g^2) ^{(t \e^d /c(d)) -1 }
  \leq \e'/2$. Afterwards 
  we choose $\eta>0$ 
small enough so that  $ (1-p) ^{-t} \eta \leq \e'/2 $, where
\be \label{pippi}
p:=\bbP(\rrr{A_x\in \bbR})=1-\exp\{- \l_* \e^d\}<1\,.
\en
  Then we take  $m=m( \eta)$ and $n=n( \eta)$ as in Proposition \ref{cinquina}. In particular, \eqref{maggiolino} holds and moreover
\be\label{danza}
[1- (1-p) ^{-t} \eta ]\, [ 1-  (1- \g^2 ) ^{(t \e^d /c(d)) -1}  ]\geq (1-\e'/2)^2 >1 -\e'\,.
\en
\begin{Remark}\label{caffeina}
As $\eta \leq \e'/2$, from  \eqref{maggiolino} we get that 
\be \label{fenec}
 \bbP \bigl( B(m) \leftrightarrow K (m,n) \text{ in $B(n)$ for } \bbG \bigr) > 1-\e'\,.
\en
This will be used in other sections.
\end{Remark}

\begin{Lemma}\label{falchetto}
In the same context of Lemma \ref{pierpilori}
 let  \begin{equation*}
\begin{split}
V_{R}:= \{x \in \partial  R \cap B(n) \,:\, & \exists y\in B(n) \setminus ( R\cup  \partial  R ) 
 \text{ such that } \\ &
|x-y| \leq \rrr{h_*(A_y)}-2\a \text{ and } \\
&\{y\} \leftrightarrow K(m,n)  \text{ in  $B(n) \setminus ( R\cup  \partial R )$ for $\bbG$}\}\,.
\end{split}
\end{equation*}
Then   we have (recall \eqref{pippi})
\begin{equation}\label{navetta}
\bbP \bigl( |V_{R}| > t\bigr) \geq 1- (1-p) ^{-t} \eta\,.
\end{equation}
\end{Lemma}
We postpone the proof of Lemma \ref{falchetto} to Subsection \ref{sec_moka18}. \rui{We point out that our set $V_R$ plays the same role of the set $U$ in \cite[Lemma 2]{GM}.  As detailed in  the proof of Lemma \ref{falchetto}, given a path $(x_0,x_1, \dots, x_k)$  from $B(m)$ to $K(m,n)$ inside $\bbG$ with all vertexes in $B(n)$ and called $x_\ell$ the last vertex  of the path inside $\partial R$, it must be  $x_\ell\in V_R$, while $x_{\ell+1}$  plays the role of $y$ in the definition of $V_R$ for $x:=x_\ell$.}
%
%
%

\begin{Remark}\label{iacopo}
The random set $V_R$   depends  only on $ A_x$ with $x\in B(n) \setminus ( R\cup  \partial R )$ and $A_x$ with $x\in T(m,n)$. Indeed, to determine $K(m,n)$, one needs to know the seeds inside $T(m,n)$.
\end{Remark}

Given $x\in \partial R $ we define   $x_*$ as the minimal (w.r.t. lexicographic order) point $y\in R$ such that $|x-y|\leq 1-2\a$. Note that $x_*$ exists for any  $x\in V_R$ since $V_R \subset \partial  R$. 
Let us show  that $F\subset G$, where
\[ F:=\{  
\exists x \in V_R \text{ with } \rrr{T^{(k_*)} _x \in U_*\,,\; T_{x_*}^{(k_*)} \in U_*}
\}
\,.
\]
To this aim,  suppose  the event  $F$ to  be fulfilled and take $ x \in V_R$ with  \rrr{$ T^{(k_*)} _x \in U_*$ and $ T_{x_*}^{(k_*)} \in U_*
$}. Since $x\in V_R$, by definition of $V_R$  there exists $ y\in B(n) \setminus ( R\cup  \partial  R )$ 
 such that 
$|x-y|\leq\rrr{ h_*( A_y)} -2\a$ and  there exists a path $(y,z_3, z_4,  \dots, z_\ell)$  inside  $\bbG $ connecting $y$ to 
$ K(m,n)$ with vertexes in 
  $B(n) \setminus ( R\cup  \partial  R )$. We  set 
 $z_0:=x_*$, $z_1=x$, $z_2:=y$. Then the event $G$ is satisfied by the string $(z_0, z_1, \dots, z_\ell)$. This proves that $F\subset G$.

Since $F\subset G$  we can estimate
\be\label{alba}
 \bbP( G\,|\ H)   \geq \bbP  \bigl(  |V_R| > t \,,\, F \,|\, H  \bigr)  = \sum_{\substack{B \subset \partial  R \cap B(n) :\\ |B|>t }}  
   \bbP  \bigl(  V_R =B\,,\, F _B\,|\, H  \bigr) \,,
   \en
   where 
   \[ F_B:= \{  
\exists x \in B \text{ with } \rrr{T_x ^{(k_*)}\in U_* \,,\;  T_{x_*}^{(k_*)} \in U_*}
\}\,.
\]
The event $F_B$ is determined by the random variables $\{ T^{(k_*)}_x \}_{ x\in D}$.
 In particular  (cf. Remark
 \ref{iacopo}) the event $\{ V_R =B\}\cap F_B$ is determined by  \[
 \begin{cases}
 T^{(k_*)}_x & \text{ with } x\in D\,,\\  
 A_x & \text{ with } x\in B(n) \setminus ( R\cup  \partial  R ) \text{ and with } x \in T(m,n)\,.
 \end{cases}
 \]
Since by assumption  $H$ is $\cF$--measurable, and due to conditions  \eqref{mare100} and  \eqref{monti100}, we conclude that the event
$\{ V_R =B\}\cap F_B$ and $H$ are independent. Hence $ \bbP  \bigl(  V_R =B, F _B\,|\, H  \bigr)= \bbP  \bigl(  V_R =B,F _B\bigr)  $. In particular, coming back to \eqref{alba}, we have
\be\label{olleboiccic}
\bbP( G\,|\ H )    \geq   \sum_{\substack{B \subset \partial R \cap B(n) :\\ |B|>t }}  
 \bbP( V_R=B\,,\, F_B)
\,.
\en

To deal with  $ \bbP( V_R=B, F_B)$ we observe that the events $\{ V_R=B\}$ and $ F_B$ are independent 
(see Remark \ref{iacopo}),  hence we get 
 \be\label{theverde}
 \bbP(  V_R=B\,,\,  F_B )=
  \bbP(  V_R=B)\bbP( F_B )\,.
\en
It remains to lower bound $\bbP(F_B)$.   We first show that there exists a subset  $\tilde B\subset B$ such that 
\be \label{stima}
|\tilde B | \geq  |B| \e^d /c(d) -1
\en
and  such that  all points of the form  $x$ or $x_*$, with $x\in \tilde B$,  are distinct. We recall that the positive constant $c(d)$ has been introduced at the beginning of \rrr{Section}  \ref{trieste65}.
To build the above set $\tilde B$ we recall that  $B \subset \partial R $ and that,  for any $x \in B$, it holds   $|x-x_*|\leq 1- 2\a$ and $x_* \in  R$.  As a consequence, given $x,x'\in B$, $x_* $ and $x'_*$ are distinct if $|x-x'| \geq 2  $ and moreover any point of the form $x_*$  with $x\in B$ cannot coincide with a point in $B$.  Hence it is enough to exhibit a subset $\tilde B \subset B$ satisfying \eqref{stima} and  such that  all points in $\tilde B$ have reciprocal distance at least  $2$.  We know that  the ball  $\bbB$ of radius $2$ contains at most $c(d) \e^{-d}$ points of $\ezd$. The  set $\tilde B$ is then  built as follows: choose  a point $a_1 $ in  $B_1:=B$ and define $B_2:= B_1\setminus (a_1+\bbB)$,  then choose  a point  $a_2 \in B_2$ and  define $B_3:= B_2\setminus (a_2+\bbB)$ and so on until possible (each $a_k$ can be chosen as the minimal point w.r.t. the lexicographic order). We call $\tilde B:=\{a_1, a_2, \dots, a_s\}$ the set of chosen points.  Since $|B_k| \geq  |B| - (k-1) c(d)\e^{-d}$, we get that $s=|\tilde B|$ is bounded from below by  the maximal integer $k$ such that  $|B| >(k-1) c(d)\e^{-d}$, i.e. $\lfloor |B| \e^d/c(d)\rfloor  >k-1$.
  Hence,  $s=|\tilde B|\geq  \lfloor  |B| \e^d/c(d)\rfloor $.  By the above observations, $\tilde B$ fulfills the desired properties.

  Using $\tilde B$ and independence, we have 
  \be\label{agg}
\begin{split}
\bbP(F_B)&=1-\bbP(\cap_{x\in B}\{ \rrr{T_x^{(k_*)}\in U_*,\,T_{x_*}^{(k_*)}\in U_*}\}^c) \\
& \geq 1-\bbP(\cap_{x\in \tilde B}\{\rrr{
T_x^{(k_*)}\in U_* ,\,T_{x_*}^{(k_*)}\in U_*}\}^c)
\\ &=1-\prod_{x\in \tilde B}(1-\bbP(\rrr{T_x^{(k_*)}\in U_*})\bbP(\rrr{T_{x_*}^{(k_*)}\in U_*}))
\\&=1-(1-\g^2)^{|\tilde B|}\,.
\end{split}
\en

As a byproduct of  \eqref{olleboiccic}, 
 \eqref{theverde}, \eqref{stima} and \eqref{agg} and finally using   \eqref{navetta} in Lemma \ref{falchetto} we get  
\be
\begin{split}
 \bbP( G\,|\ H )    & \geq   \sum_{\substack{B \subset \partial  R \cap B(n) :\\ |B|>t }}  \bbP ( V_R=B)  \left(1-(1- \g^2) ^{|\tilde B|}\right)
\\
& \geq   \left(1-(1- \g^2) ^{(t \e^d /c(d) ) -1 }\right) \sum_{\substack{B \subset \partial  R \cap B(n) :\\ |B|>t }}  \bbP ( V_R=B)  
\\
&=  \left(1-(1- \g^2) ^{(t \e^d /c(d)) -1 }\right)  \bbP ( |V_R|>t)  \\
&\geq \left[  1- (1- \g^2 ) ^{(t \e^d /c(d) )-1}  \right]\, \left [ 1- (1-p) ^{-t} \eta
\right] \,.
\end{split}
\en
Finally, using \eqref{danza} we   conclude the proof of Lemma \ref{pierpilori}.
%
\subsection{Proof of Lemma \ref{falchetto}}\label{sec_moka18} 
%
%
%
Suppose that $B(m) \leftrightarrow K  (m,n)$ in $B(n)$ for $ \bbG$. Take a path $(x_0,x_1, \dots, x_k)$ from $B(m)$ to $ K  (m,n)$ inside $ \bbG$ with all vertexes $x_i$ in $B(n)$.
Recall that $K  (m,n) \subset T(n)$ and $R\cup  \partial R $  is disjoint from $T(n)$ by \eqref{mare100}. In particular,    $R\cup  \partial R $  is disjoint from $K (m,n)$. 
Since $B(m)\subset R$,  the path starts at  $R$. Let $x_r$ be  the last point of  the path contained in $R$. Since $R$ is disjoint from $K  (m,n)$ and $x_k\in K (m,n)$, it must be  $r<k$. Necessarily, $x_{r+1}\in \partial R $. Call $x_\ell$ the last point of the path contained in $\partial  R$.  It must be $\ell< k$ since 
$\partial   R$ is disjoint from $K (m,n)\ni x_k$. 
We claim that  $x_\ell \in V_R$ and \rrr{$A_{x_\ell} \in \bbR $}.  To prove our claim we observe that the last property follows from the fact that 
all points $x_0,x_1, \dots, x_k$ are in $ \bbV$. Recall that these points  are also in $B(n)$. Hence $x_\ell \in  \partial R \cap B(n)$.   Since $\{x_\ell, x_{\ell+1}\}\in \bbE$,  we have \rrr{$|x-y|\leq h_*( A_y )-2\a$} with $x:= x_{\ell}$ and $y:= x_{\ell+1}$.  Finally, it remains to observe that 
$(x_{\ell+1}, \dots, x_k)$ is a path  connecting $x_{\ell+1}$ to $x_k \in K(m,n)$ in 
  $B(n) \setminus ( R\cup  \partial R )$ for $\bbG$. Hence, $x_\ell \in V_R$.

We have proved that if $B(m) \leftrightarrow K (m,n)$ in $B(n)$ for $ \bbG$, then $V_R$ contains at least a  vertex of $\bbV$. As a byproduct with 
 \eqref{maggiolino}  \rrr{we} therefore have
\be\label{lalla1}
\eta> \bbP \bigl( B(m) \not  \leftrightarrow K (m,n) \text{ in $B(n)$ for } \bbG \bigr) \geq \bbP( V_R\cap \bbV=\emptyset)  
\,\rrr{.}
 \en
 On the other hand, we can bound
 \be \label{lalla2}
 \bbP( V_R\cap \bbV=\emptyset )  \geq \bbP( V_R\cap \bbV=\emptyset\,,\, |V_R| \leq  t)\,.
 \en
 Note that $V_R$  and $(A_x)_{x\in \partial  R}$ are independent (see Remark \ref{iacopo}).
Hence
\be\label{lalla3}
\begin{split}
&  \bbP( V_R\cap \bbV=\emptyset \,,\, |V_R| \leq  t)= \sum _{
\substack{
B\subset \partial R\,:\,\\
|B|\leq  t}}   \bbP( V_R=B\,,\, \rrr{A_x\not\in \bbR} \; \forall x\in B)\\
& =\sum _{
\substack{
B\subset \partial  R\,:\,\\
|B|\leq   t}}   \bbP( V_R=B) \bbP( \rrr{A_x \not\in \bbR} \; \forall x\in B)\\
& =\sum _{
\substack{
B\subset \partial  R\,:\,\\
|B|\leq   t}}   \bbP( V_R=B) (1-p)^{|B|} 
\\
&\geq \sum _{
\substack{
B\subset \partial  R\,:\,\\
|B|\leq   t}}   \bbP( V_R=B) (1-p)^{t}=\bbP( |V_R|\leq t ) (1-p)^{t}\,.  
 \end{split}
 \en
 By combining \eqref{lalla1}, \eqref{lalla2} and \eqref{lalla3} we get that $\eta \geq \bbP( |V_R|\leq t ) (1-p)^{t}$, which is equivalent to \eqref{navetta}.

\section{Proof of Lemma \ref{ciak2011} and Lemma \ref{piccolino}}\label{patroclo}
\subsection{Proof of Lemma \ref{ciak2011}}
  Item (i) is trivial and Item (iii) follows from Items (i) and (ii).   Let us assume \eqref{chip} and prove Item (ii).  We claim  that $\{E=\hat E\}\cap \{F=\hat F\}$ equals  $\{E=\hat E\}\cap W$, where $W$ is the event that 
   \begin{itemize}
  \item[(a)] for any $z\in \hat F$ there are points $z_2,\cdots,z_{k-1},z_k=z$  in $ \hat F$ 
  such that \rrr{$A_{z_i}\in \bbR$ for $i=2,\dots,k$},  \rrr{$|z_i-z_{i+1}| \leq h (A_{z_i},A_{z_{i+1}}) -2\a$} for $i=2,\dots, k-1$ and such that 
   \rrr{$|z_1-z_2|\leq h_*(A_{z_2})-2\a$} for some $z_1\in \hat  E $;
   \item[(b)] if  $z \in  B'\setminus (C\cup\partial  C)  $    is such that 
  $\exists  z'\in \hat E$ with \rrr{$|z-z'|\leq h_*(A_{z})-2\a$}, then $z\in \hat F$; 
   \item[(c)] for any $z\in\bigl( B'\setminus (C\cup\partial  C)\bigr) \cap \partial \hat F$    
    there is no $z' \in \hat F$   such that \rrr{$A_{z'}\in \bbR$ and $|z-z'| \leq  h(A_z,A_{z'})-2\a$}.
   \end{itemize}  Before proving our claim, we observe that it allows to conclude the proof of the lemma.
 \rui{Indeed,  due to the definition of $E$,
 the   event $\{E=\hat E\}$    belongs to the $\s$--algebra in  Item (ii) of the lemma. Moreover,  }  
  as  the point $z$ appearing in Item (b) must be in $\partial \hat E$, the event $\{E=\hat E\}\cap W$  belongs to the  \rui{same $\s$--algebra} due to the explicit description \rui{of $W$} given above.
  
   It remains to derive our claim. Due to Item (a), the event  $  \{E=\hat E\}\cap W $ implies that $\{E=\hat E\}\cap \{F\supset \hat F\}$. On the other hand, suppose that   the event $ \{E=\hat E\}\cap W$ takes place and  let  $z\in F$. By Definition \ref{def_triade}   there exists a path $(z_2,\dots,z_k)$ inside $\bbG$ where $z_k=z$, all points $z_2,\cdots,z_k$ are in $B'\setminus (C\cup\partial  C)$ and \rrr{$|z_1-z_2|\leq h_*(A_{z_2})-2\a$} for some $z_1\in \hat  E$.
By Item (b),   $z_2 \in \hat F$. Let $j$ be the maximal index in $\{2,3,\dots, k\}$ such that $z_2,z_3, \dots, z_j \in \hat F$.  Suppose that $j<k$.  As $\{z_j,z_{j+1}\}\in \bbE$, we get that $z_{j+1}\in \partial \hat F$. Since $z_j \in \hat F$, $z_{j+1}\in \bigl( B'\setminus (C\cup\partial  C)\bigr) \cap \partial \hat F$ and 
$\{z_j,z_{j+1}\}\in \bbE$,  we get a contradiction with Item
 (c).    Then, it must be   $j=k$, thus implying that   $z=z_j$ and therefore $z\in \hat F$.
 Up to now, we have proved that $  \{E=\hat E\}\cap W \subset \{E=\hat E\}\cap \{F= \hat F\}$.  
 We observe that, given $z, z_1,z_2,\dots, z_k$ as in Definition \ref{def_triade}, it must be $z_2,\dots, z_k\in F$. This observation and the above Items (a), (b), (c) imply the opposite inclusion 
 $\{E=\hat E\}\cap \{F= \hat F\} \subset   \{E=\hat E\}\cap W $.  

\subsection{Proof of Lemma \ref{piccolino}}
  Recall Definitions  \rrr{\ref{vichinghi_+} and} \ref{vichinghi}.
If $z\in E$, then \rrr{$T_{z}^{(i)}\in\bbR $}  and therefore $z\in \rrr{\bbV_+}$. If $z\in F$, then $z\in \bbV$ (by definition of $F$) and therefore $z\in  \rrr{\bbV_+}$. This implies  that $E, F\subset  \rrr{\bbV_+}$, hence $C'\subset  \rrr{\bbV_+}$.

Since $C$ is connected in $ \rrr{\bbG_+}$ and since $\bbG\subset  \rrr{\bbG_+}$ (in particular the string $(z_2, \dots, z_k)$ appearing in the definition of $F$ is a path in $ \rrr{\bbG_+}$), to prove the connectivity of $C'$ in $ \rrr{\bbG_+}$ it is enough to show the following: 
\begin{itemize}
\item[(i)]
if  $z_0,z_1 \in \rrr{ \bbV_+}$   satisfy $ \rrr{T^{(i)} _{z_0}\in  U_*}$, $ \rrr{ T^{(i)} _{z_1}\in U_*}$ and  $|z_0-z_1|\leq 1-2\a$, then $\{z_0,z_1\}\in  \rrr{\bbE_+}$;
\item[(ii)]
if   $z_1,z_2 \in  \rrr{\bbV_+}$ satisfy  $ \rrr{T^{(i)} _{z_1}\in U_*}$, \rrr{$A_{z_2}\in\bbR$}  and  $\rrr{|z_1-z_2|\leq h_*( A_{z_2}) -2\a}$, then $\{z_1,z_2\}\in  \rrr{\bbE_+}$.
\end{itemize}

 Using the assumptions of Item (i)  we get \rrr{(recall \eqref{mango} and Assumption (A5))}
   \be\label{lorenzo18}  
\rrr{  |z_1-z_0|  \leq 1-2\a \leq  h(  T_{z_1}^{ (i)} ,T_{z_0}^{ (i)})-\a\leq h( A^{\rm au}_{z_1}   ,A ^{\rm au}_{z_0} )-\a\,.
}
\en
\rrr{The above estimate implies that  $\{z_0,z_1\}\in \bbE_+$}.

 Using the assumptions of Item (ii) we get 
\be\label{pierpaolo18}
   |z_1-z_2| \leq  \rrr{h_*( A_{z_2}) -2\a  \leq h( T^{(i)}_{z_1},A_{z_2}) +\a/2-2\a \leq  h( A^{\rm au}_{z_1}   ,A ^{\rm au}_{z_2} )-\a }\,.
\en
\rrr{Note that in the second bound we have used that $T^{(i)}_{z_1}\in U_*=U_*(\a/2)$ (see Assumption (A4)), while in the third bound we have used Assumption (A5). Eq.~\eqref{pierpaolo18} implies that $\{z_1,z_2\}\in \bbE_+$.}

\section{Proof of Propositions \ref{prop_occ_origin} and \ref{prop_occ_e1}}\label{puffo1}
By  iteratively  applying Lemma \ref{piccolino} and using  Remark \ref{gomitolo1} as starting point, we get that \rrr{$C_2, \dots, C_{5}$} are connected subsets in $\rrr{\bbG_+}$, if the associated success-events are satisfied.
The lower bounds $\bbP( 0\text{ is occupied}\, |\,S_0)   \geq 1-4 \e'$ and \eqref{soglia2} follow   from the inequalities
\be\label{catenina}
\bbP( S_{i+1}| S_0, S_1,  \cdots, S_i) 
\geq 
\begin{cases}
1- 4 \e' & \text{ for }i=0\,,\\
1-\e' & \text{ for } i\in 
\rrr{\{1,2,3\}} \,,\\
1- 3 \e '& \text{ for }\rrr{i=4}\,,
\end{cases}
\en
by applying  the chain rule and the Bernoulli's inequality (i.e. $(1-\d)^k \geq 1- \d k$ for all $k\in \bbN$ and $\d\in [0,1]$).
\rui{Apart from} the case $i=0$, the proof of \eqref{catenina} can be obtained by applying Lemma \ref{pierpilori}. Below we will treat in detail the cases $\rrr{i=0,1}$. For the other cases we will 
give some comments, and show the validity of conditions \eqref{mare100} and \eqref{monti100} in Lemma \ref{pierpilori}.
In what follows we will introduce  points $\tilde b_1, \tilde b_2,\dots$. We stress that  the subindex $k$ in $\tilde b_k$ does not  refer to the $k$-th coordinate. We write $(\tilde b_k)_a$
for the $a$--th coordinate of $\tilde b_k$.
\subsection{Proof of \eqref{catenina} with $i=0$} \label{svezia}
 We want to show that 
 $\bbP (S_1|S_0) \geq 1-4 \e'$.   Since $S_0$ and $S_1$  are  increasing events w.r.t. $\preceq$, by the FKG inequality (see Section  \ref{scremato}) we have $\bbP(S_1|S_0) \geq \bbP(S_1)$. To show that $\bbP (S_1) \geq 1-4 \e'$, we note that the event $W_j:=\{B(m) \leftrightarrow K ^{(j)} (m,n)  \text{ in $B(n)$ for } \bbG\}$ implies that $C_1$ contains a point of $ K^{(j)} (m,n) $. Hence,
$\cap_{j=1}^{4} W_j\subset S_1$ and therefore (see Remark \ref{caffeina} \rrr{which is based on Proposition \ref{cinquina}}) 
$ \bbP (S_1^c) \leq \bbP\left( \cup_{j=1}^{ 4 } W_j^c\right) \leq 4 \e'$.
\subsection{Proof of \eqref{catenina} with $i=1$}
\label{venezia} We want to show that 
 $\bbP (S_2|S_0, S_1) \geq 1- \e'$. 
 \begin{Lemma}\label{barrio1}
Given $B(m)\subset  R_1\subset B_0'$,  the event $S_0\cap \{C_1=R_1\}$ is determined by the random variables $\{A_x\}_{x\in R_1\cup \partial R_1}$. 
\end{Lemma}
\begin{proof} 
The claim is trivially  true for the event $S_0$.  It is therefore enough to show that, if $S_0$ takes place,  then the event $\{C_1=R_1\}$   is equivalent to the following: (i) 
for any $x\in R_1$ there is a path from $x$ to $B(m)$  inside  $R_1$ for $\bbG$ and (ii)    any $x\in \partial R_1\cap B_0'$ is not \rrr{adjacent}  to $R_1$ in $\bbG$, i.e. there is no $y\in R_1$ such that $\{x,y\}\in \bbE$. Trivially  the event $\{C_1=R_1\} $ implies (i) and (ii). On the other hand, let us suppose that (i) and (ii) are satisfied, in addition to $S_0$. Then (i) implies that $R_1\subset C_1$. Take, by contradiction,  $x\in C_1\setminus R_1$. By definition of $C_1$ there exists a path from $x$ to $B(m) $ in $B_0'$ for $\bbG$. Since $x\not \in R_1$ and  $B(m) \subset R_1$, there exists  a last point $x'$  in $R_1^c$ visited by the path. Since the path ends in $B(m) \subset R_1$, after  $x'$ the path visits another point $y$ which must belong to $R_1$. Hence we have $\{x',y\}\in \bbE$ (and therefore $x'\in \partial R_1 \cap B_0'$) and $y\in R_1$, thus contradicting (ii).
\end{proof}

We proceed with the proof  that $\bbP(S_2|S_0,  S_1)\geq 1-\e'$ by applying Lemma \ref{pierpilori}. Recall that $T_1(n)=T^*(n)$, $T_1(m,n)= T^*(m,n)$ and recall  Definition \ref{coca?} of $K_1(m,n)$.
We can write
\begin{multline} \label{sam0}
\bbP(S_2|S_0, 
S_1)\\
=\sum _{R_1, \tilde b_1} \bbP(S_2| S_0, S_1, C_1=R_1,   b^{(1)}=\tilde b _1 )  \bbP( C_1=R_1,  b^{(1)}=\tilde b_1 |S_0, S_1)\,,
\end{multline}
where in the above sum  $R_1  \subset \ezd$ and $\tilde b _1 \in \ezd$ are taken such that 
$\bbP( S_0, S_1, C_1=R_1,  b^{(1)}=\tilde b _1)>0$.
We now  apply Lemma \ref{pierpilori}  (with the origin there replaced by $\tilde b_1$) to lower bound  $\bbP(S_2|H_1)$ by $1- \e'$, where \rrr{the event $H_1$ is given by $S_0\cap  S_1\cap\{ C_1=R_1\}\cap\{  b^{(1)}=\tilde b _1\}$}.

\smallskip

  We first  check condition \eqref{mare100}. Note that 
  $B(m)\subset R_1\subset B_0'$ and $B(\tilde{b}_1,m)\subset R_1\cap T(m,n)$ as $\bbP(H_1)>0$. Hence we have  $(R_1 \cup \partial R_1) \subset
(B_0'\cup \partial B_0')$.
We point out that, given  $x\in  B_0'\cup \partial B_0'$, it must be $x_1 \leq n+ \e + 2m + 1-2\a  $.  On the other hand, given 
$x\in \tilde b_1+ \bigl( T^*(n) \cup T^*(m,n)\bigr)$, it must be $x_1\geq 2n+m+\e-1$. 
 As $2m < n$ and $2m|n$, we have $n\geq 4m$ and therefore $n>m+2$. Hence   $x$ cannot belong to both sets. In particular, 
 we have the analogous of \eqref{mare100}, i.e. 
$ B(\tilde b_1 , m) \subset R_1$ 
and $  \left( R_1\cup \partial  R_1 \right)\cap \bigl(
  \tilde{b}_1+\bigl( T^*(n) \cup T^*(m,n)\bigr)
  \bigr)=\emptyset$.
Condition \eqref{monti100} is trivially satisfied by taking $\L_1(x) := \emptyset $ for all $x\in R_1 \cup \partial R_1$ and $k_*=1$.

\smallskip

 We now  prove that $H_1$ belongs to the $\s$--algebra  $\cF_1$ generated by $(A_x)_{x\in R_1\cup \partial R_1}$. 
Due to Lemma \ref{barrio1}, the event $S_0\cap \{C_1=R_1\}$ is determined by $\{ A_x\}_{x\in R _1\cup \partial R_1}$. If  the event  $S_0\cap \{C_1=R_1\}$ takes place, then  the event   $S_1\cap \{ b^{(1)}=\tilde b_1\}$ becomes equivalent to the following: 
(1) 
  $B(\tilde b_1, m)$ is a seed and,  for any other seed $B(z, m)\subset R_1\cap T(m,n)$, $\tilde b_1$ is  lexicographically smaller than $z$; 
(2)  for any $j=2,3,4$ the set $R_1$ contains a point $x\in L_j(T(n))$ \rrr{adjacent}  for $\bbG$  to a seed  contained in $R_1\cap  L_j\bigl( T(m,n)\bigr)$.
Note that in Item (2) we have used Lemma \ref{ironman}, thus implying that  if a seed $B(z,m)$  is \rrr{adjacent} for $\bbG$ to a point $x\in R_1$, then any point of $B(z,m)$ is  connected for $\bbG$ to $x$, and therefore $B(z,m)\subset R_1$ as $C_1=R_1$.
The above  properties  $(1),(2)$ can be checked when  knowing $\{A_x\}_{x\in R_1\cup \partial R_1}$. Hence, $H_1$ belongs to the $\s$--algebra  $\cF_1$.

\smallskip

Due to  the above observations, 
 we can  apply Lemma \ref{pierpilori} with conditional event $H_1$,    $\tilde b_1 $ as new origin, $R_1$ as new set $R$,  $\L_1(x):=\emptyset $ for any $x\in R_1 \cup \partial R_1$ as new function $\L(x)$
   and   $k_*:=1$. We get that $ \bbP( G_1|H_1) \geq 1-\e'$, where 
  $G_1$ is  the event corresponding to the  event $G$ appearing in Lemma \ref{pierpilori} (replacing  also $K(m,n)$   by $K_1(m,n)$).  
To show that $ \bbP( S_2|H_1) \geq  1-\e'$, and therefore that  $\bbP(S_2|S_0,  S_1)\geq 1-\e'$ by \eqref{sam0},  it is enough to show that  $ G_1 \cap H _1\subset   S_2\cap H_1$. 
To this aim 
let us suppose that $G_1 \cap H_1 $ takes place. Let (P1),...,(P8) be the properties entering in the definition of $G$ in Lemma \ref{pierpilori}, when replacing  $R$, $B(n)$ and $K(m,n)$   by  $R_1$, $B(\tilde b_1,n)$ and   $K_1(m,n)$, respectively. 
To \rui{conclude} it is enough to show that $z_\ell\in C_2$ since $z_\ell \in  K_1 (m,n) $ by (P5). Note that by $H_1$,  (P1), (P2), (P6) and (P7) we have that $z_0\in C_1$ and $z_1\in E_1$, while by  $H_1$, (P3), (P4) and (P8) we get that $z_2, \dots, z_\ell \in F_1$. Since $C_2:=C_1\cup E_1\cup F_1$, we have that $ z_\ell\in C_2$.

\subsection{Generalized notation}

In order to define objects once and for all, given $\rrr{1}\leq i \leq \rrr{4}$ and given sets $R_1, R_2, \dots, R_i$ and points $\tilde b_1, \tilde b_2, \dots, \tilde b_{i}$ we set
\begin{align}
&  H_{i}:=
\left(\cap _{k=0} ^{ i} S_k\right)
 \cap \left( \cap _{k=1} ^{i}  \bigl\{  C_k = R_k \bigr\}  \right)
  \cap
   \left( \cap_{k=1} ^{i} \bigl\{ b ^{(k)}=\tilde{b}_{k}  \bigr\}\right)\,,
\label{acca}\\
& \L_i(x) :=\bigl\{ k\,:\, 1\leq k \leq i-1  \,, \; x\in B( \tilde b_k ,n+1) \bigr\}\,,   \quad  \forall x\in R_i \cup \partial R_i \,, \label{lambada}\\
&  \cF_{i}:= \s\left(  \{A_x\}_{x\in R_{i}\cup\partial R_{i}} , \{T_x^{(k)}\}_{ x\in R_{i}\cup \partial R_{i},\, k\in \L_{i}(x)}\right)\label{sigmund}\,.
\end{align}
Note that the above definitions cover also the objects $H_1, \L_1, \cF_1$ introduced in Section \ref{venezia}.

For later use, it will be convenient to write also $H_i[R_1, \dots, R_i;\tilde b_1,  \dots, \tilde b_{i}]$ (instead of $H_i$) to stress the dependence on  $R_1, \dots, R_i,\tilde b_1,  \dots, \tilde b_{i}$.

\subsection{Proof of \eqref{catenina} with \rrr{$i= 2,3$}}\label{sole87}
The proof follows the main arguments presented for the case \rrr{$i=1$}.  One has to condition similarly to \rrr{\eqref{sam0}} and afterwards apply Lemma \ref{pierpilori} 
  with  $k_*:=i$ and  $B(m)$, $B(n)$, $T(n)$, $T(n,m)$, $K (m,n)$,     $R$, $H$  and $\L(x)$   replaced by 
 $B(\tilde b_i,m)$,  $B(\tilde b_i,n)$,  $ \tilde b_i +T_i(n)$, $\tilde b_i+T_i(m,n)$, $K_i (m,n)$,  $R_i$, $H_i$  and $\L_i(x)$, respectively.   The fact that $H_i\in \cF_i$ can be obtained as for the case \rrr{$i=1$} by writing $H_i= H_{i-1}\cap  S_i \cap \{C_i=R_i\} 
 \cap
 \{ b^{(i)}= \tilde b _i \}$ and using the iterative result that $H_{i-1}\in \cF_{i-1}$ together with Lemma \ref{ciak2011}. 

To check    \eqref{mare100} is straightforward but cumbersome and we refer to Appendix \ref{natale45} for the details.
  
  \subsection{Proof of \eqref{catenina} with \rrr{$i=4$}}\label{sole88} 
%
%
For $j=1,2,3$, we define the event  $W_j:=\{\text{\rrr{$C_5$} contains at least one vertex inside $K^{(\rrr{4}+j)}$}\}$. Then $\rrr{S_5}= \cap_{j=1}^3 W_j$ and 
\begin{equation}
 \label{eq1}
 \bbP(\rrr{S_5}\,|\, \rrr{S_0,S_1,\cdots,S_{4}}) \geq 1-\sum _{j=1}^3   \bbP(W_j^c\,|\, \rrr{S_0,S_1,\cdots,S_{4}})\,.
\end{equation}
Hence, we only need to show that $\bbP(W_j\,|\, \rrr{S_0,S_1,\cdots,S_{4}})\geq 1- \e'$. We use Lemma \ref{pierpilori} and the same strategy used in the previous steps.
 Recall \eqref{acca}, \eqref{lambada}, \eqref{sigmund}.
 We have to    lower bound by $1-\e'$ the conditional probability $\bbP(W_j\,|\,H_{\rrr{4}})$ when $\bbP(H_{\rrr{4}})>0$.
   To this aim we apply Lemma \ref{pierpilori} with $R:=R_{\rrr{4}}$, $\L(x):=\L_{\rrr{4}}(x)$ and with $B(n)$, $T(n)$,  $T(m,n)$, $k_*$  replaced by $B(\tilde b_{\rrr{4}},n)$, $\tilde b_{\rrr{4}}+\hat T_j(n)$, $\tilde b_{\rrr{4}}+ \hat T_j(m,n)$ and $\rrr{4}$,  respectively.  The validity of \eqref{monti100}  is trivial. To check    \eqref{mare100} is straightforward but cumbersome and we refer to  Appendix \ref{natale46} for the details.


\appendix

\section{Example of Tanemura's algorithm }\label{app_tanemura}
To make the comprehension of Tanemura's construction easier we consider the following example guided by Figure \ref{tanemura2}. Suppose that $M=4$ (hence $\L=([0,5]\times [0,3])\cap \bbZ^2$) and that we have a field $\g\in \{0,1\}^\L$ defined on $(\O, \bbQ)$. We call $x\in \L$  occupied if $\g(x)=1$ and we say that  two occupied sites $x,y\in \L$ are  linked if $|x-y|=1$. As example, consider the configuration in Figure \ref{tanemura2}--(a) in which empty dots correspond to occupied points while  crosses correspond to  non--occupied  points. We start looking at the occupied points in the left vertical face of $\L$ and we define $C_1^j:=(E_1^j,F_1^j)=(\{x_1^j\},\emptyset)$ for $j=0,2,3$, while $C_1^1:=(E_1^1,F_1^1)=(\emptyset,\{x_1^1\})$. Then we focus on $C_1^0$ and we start constructing its extensions $\{C_s^0\}_{s=2}^{M^2}$. In particular (see Figure \ref{tanemura2}--(b)) we have 
$C_s^0:=(E_s^0,F_s^0)=(\{x_1^0,\ldots,x_s^0\},\emptyset)$ for $s=1,\ldots,4$, $C_5^0:=(E_5^0,F_5^0)=(\{x_1^0,\ldots,x_4^0\},\{x_5^0\})$ and $C_s^0:=(E_s^0,F_s^0)=(\{x_1^0,\ldots,x_4^0,x_6^0,\ldots,x_s^0\},\{x_5^0\})$ for $s=6,7,8$. 
 For the reader's convenience, we comment the construction of $C_6^0$. The point of $\L$ involved until the construction of $C_5^0$ are given by
\[
W^0_5=\{x_1^0,x_1^1,x_1^2,x_1^3, x^0_2, x^0_3, x^0_4,x^0_5\}\,.
\]
We have $E^0_5=\{x^0_1, x^0_2, x^0_3, x^0_4\}$,   $\bar{E}^0_5=( x^0_1, x^0_2, x^0_3, x^0_4)$ and $(\L\cap \D E^0_5)\setminus W^0_5=\{(1,1), (2,1), (3,1)\}$.
Property $\cP^0_5$ is trivially satisfied. Hence we denote by $x^0_6$ the last element of $(\L\cap \D E^0_5)\setminus W^0_5$ w.r.t. the ordering $\prec$ of $\D E^0_5$ induced by the string $\bar{E}^0_5$. Since $x^0_6$ is occupied and linked to $x^0_4$, we have 
$C^0_6=(E^0_6,F^0_6):=(E^0_5\cup\{x^0_6\},F^0_5)=(\{x_1^0,\ldots,x_4^0,x_6^0\},\{x_5^0\})$. The full dots  in Figure \ref{tanemura2}-(b) are occupied points that have been visited by the algorithm during the construction of $\{C_s^0\}_{s=1}^{M^2}$ and hence they are in $E_{M^2}^0$, while the grey big cross corresponds to  the  non--occupied site visited by the algorithm, which therefore belongs to $F_{M^2}^0$. We draw a dotted edge between $x_j^0$ and $x_h^0$, with $j<h$, if $x_h^0$ has been visited by the algorithm when analyzing the boundary of $x_j^0$.

 Let us focus now on Figure \ref{tanemura2}--(c). After having completed the construction of $\{C_s^0\}_{s=1}^{M^2}$, we start the construction of $\{C_s^1\}_{s=1}^{M^2}$ from $C_1^1:=(E_1^1,F_1^1)=(\emptyset,\{x_1^1\})$. Since $E_1^1$ is empty, we define $C_s^1:=C_1^1$ for all $s=2,\ldots,M^2$. Then we proceed with the construction of $\{C_s^2\}_{s=1}^{M^2}$ applying the same procedure used for $\{C_s^0\}_{s=1}^{M^2}$ and recalling at each step that we can visit only points that have not been analyzed in the previous steps. To distinguish $\{C_s^2\}_{s=1}^{M^2}$ from $\{C_s^0\}_{s=1}^{M^2}$, we have drawn continuous edges instead of the dotted edges introduced before.  For the reader's convenience, we comment the construction of $C^2_6$. We have $C^2_5=(E^2_5,F^2_5)=(\{x^2_1, x^2_2, x^2_3, x^2_4\}, \{x^2_5\})$,  $\bar E^2_5=(x^2_1, x^2_2, x^2_3, x^2_4)$ and $(\L\cap \D E^2_5)\setminus W^2_5=\{(1,3)\}$. Trivially, property $\cP^2_5$ is satisfied. We then set $x^2_6:= (1,3)$. Since $x^2_6$ is occupied and linked to $x^2_2$, we set $C^2_6= (E^2_6, F^2_6):=(E^2_5\cup\{x^2_6\},F^2_5)$.
Finally we consider $C_1^3:=(E_1^3,F_1^3)=(\{x_1^3\},\emptyset)$ and we note that all points inside $\L$ of  the boundary of $x_1^3$ have already been  visited. Hence we define $C_s^3:=C_1^3$ for $s=2,\ldots,M^2$. There is an occupied point in $\L$ that has not been visited (see the empty dot in Figure \ref{tanemura2}--(c)) and hence we define the field $\z$ equal to 0 on that point. As we can see in Figure \ref{tanemura2}--(c), at the end of the algorithm, we have obtained three trees (one of which is made only by a root) and we have two LR crossings since two trees have a branch that intersects the right vertical face of $\L$.

\begin{figure}
\includegraphics[scale=0.37]{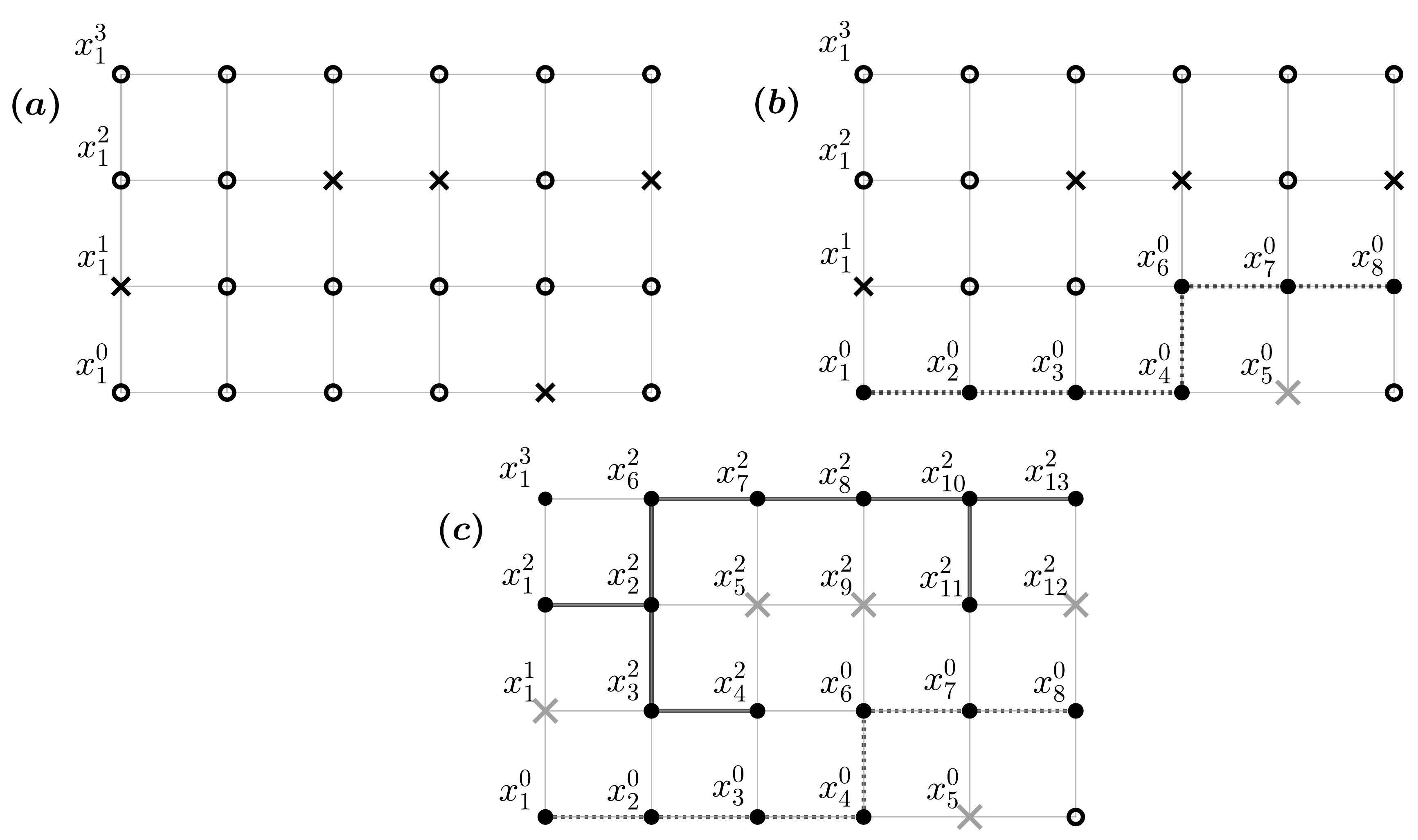}
\caption{}
\label{tanemura2}
\end{figure}

\section{Locations}\label{app_locus}
\begin{Claim}\label{locus2-3} Given $i\in \{\rrr{2,3}\}$, 
if  $S_1\cap \cdots \cap S_i$ occurs, then 
 $b^{(i)}+T_i(m,n)\subset B(\rrr{(i+1)}Ne_1,N)$.  
 \end{Claim}
\begin{proof} Let 
 $x\in (b^{(i)}+T_i(m,n))$. Then by  \eqref{povoino}
$x_1\in(b_1^{(i)}+N+[-m,m])=\rrr{i}N+N+[-m,m]\subset \rrr{(i+1)}N+[-N,N]$. Using again \eqref{povoino},   for $j\geq 2$ we get the following cases: (a) 
 if $b_j^{(i)}\geq 0$, then $x_j\in[b_j^{(i)}-n,b_j^{(i)}]\subset [-n,n-m]\subset[-N,N]$; (b)
 if $b_j^{(i)}< 0$, then $x_j\in[b_j^{(i)},b_j^{(i)}+n]\subset [-n+m,n]\subset[-N,N]$.
\end{proof}

\begin{Claim}\label{locus4}
If $S_1\cap \cdots \cap S_{\rrr{4}}$ occurs,  then $b^{(\rrr{4})}+\hat T_1(m,n) \subset B( 5N e_1, N)$, 
$b^{(\rrr{4})}+\hat T_2(m,n) \subset B(4N e_1+ N e_2, N)$ and  $b^{(\rrr{4})}+\hat T_3(m,n) \subset B(4N e_1- Ne_2, N)$.
\end{Claim}
\begin{proof} The inclusion concerning $b^{(\rrr{4})}+\hat T_1(m,n)$ can be derived as in the proof of Claim \ref{locus2-3}. 
Consider now $x\in (b^{(\rrr{4})}+\hat T_r(m,n))$ for $r\in\{2,3\}$. Then by \eqref{povoino} with $i=\rrr{4}$ we get that  $x_1\in  b_1^{(\rrr{4})} +[0,n]\subset 4N+[-N,N]$ and (a) $ x_2\in b_2^{(\rrr{4})}+N+[-m,m] \subset N+[-n,n]$ for $r=2$; (b) $ x_2\in b_2^{(\rrr{4})}-N+[-m,m] \subset -N+[-n,n]$
 for $r=3$. Using \eqref{povoino} with $i=\rrr{4}$ one finally gets that $|x_j |\leq N$ for $j\geq 3$, as in the proof of Claim \ref{locus2-3}. 
\end{proof}

\subsubsection{Validity of  condition \eqref{mare100} in Section \ref{sole87}}\label{natale45}   The  inclusion in \eqref{mare100} is trivially satisfied in all steps. We concentrate only on the second property in \eqref{mare100}, concerning disjointness. 
 
 To proceed we first recall that \eqref{povoino} holds for $i=\rrr{2,3}$.
 To get the disjointness in \eqref{mare100}  for $i\in\{\rrr{2,3}\}$ we argue as follows. We observe that 
 $R_{i}\cup\partial R_i\subset \cup_{k=0}^{i-1}(B'_k\cup\partial B_k')$ and points in $\cup_{k=0}^{i-1}(B'_k\cup\partial B_k')$ have their first coordinate not bigger than $\rrr{i}N+m+1$ (cf. Fig.~\ref{aquile}, \eqref{povo} and  \eqref{povoino} for $i-1$ instead of $i$). On the other hand, points in $\tilde b _i+\bigl( T_i(n)\cup T_i(m,n)\bigr)$ have their first coordinate not smaller than $\rrr{i}N+n-1$ (cf.~\eqref{povoino}). Since $n-m>2$
we get that   $\tilde b_i+\bigl( T_i (n)\cup T_i(m,n)\bigr) $ and $R_{i}\cup\partial R_{i}$ are disjoint for $i\in\{\rrr{2,3}\}$.

  \subsubsection{Validity of  condition \eqref{mare100} in Section \ref{sole88}}\label{natale46}  
The disjointness of $R_{\rrr{4}} \cup \partial R_{\rrr{4}} $ and $ \tilde b_{\rrr{4}} +\bigl( \hat T_1(n)\cup \hat T_1(m,n) \bigr)$ follows easily from the fact that   $x_1 \leq (\tilde b_{\rrr{4}})_1 + m+1$ for points $x$ in the first set, while  $x_1\geq   (\tilde b_{\rrr{4}})_1 +n -1$ for points $x$ in the second set.

 We  now show that $R_{\rrr{4}} \cup \partial R_{\rrr{4}} $ and $ \tilde b_{\rrr{4}} +\bigl( \hat T_3(n)\cup \hat T_3(m,n) \bigr)$ are disjoint (the result for $ \tilde b_{\rrr{4}} +\bigl( \hat T_2(n)\cup \hat T_2(m,n) \bigr)$ is similar).
 Suppose first that  $(\tilde b_{\rrr{3}} )_2 \geq 0$.  Then $(\tilde b_{\rrr{4}})_2 \in (\tilde b_{\rrr{3}})_2 + [-n,\rrr{-m}]$. 
  By construction, if $x\in R_{\rrr{4}} $ with $x_1 \geq 4N-m$, then $x\in \tilde b_{\rrr{3}}+T_{\rrr{3}}(m,n)$ and therefore $ x_2 \geq 
 (\tilde b_{\rrr{3}})_2 -n$. 
 Take now  $y\in \tilde b_{\rrr{4}}+\bigl( \hat T_3(n)\cup \hat T_3(m,n) \bigr)$. Then $y_1 \geq  (\tilde b_{\rrr{4}})_1=4N$ and   $y_2 \leq (\tilde b_{\rrr{4}})_2 -n+1 \leq 
 (\tilde b_{\rrr{3}})_2-m-n+1$.
  Suppose by contradiction  that $y\in R_{\rrr{4}} \cup \partial R_{\rrr{4}}$. Then there exists $x\in R_{\rrr{4}}$ such that $|x-y|<1$.  This implies that $x_1 \geq y_1-1\geq 4N-1$. Hence, by the initial observations, $x_2 \geq (\tilde b_{\rrr{3}})_2 -n $.  This last bound, together with 
   $y_2 \leq 
 (\tilde b_{\rrr{3}})_2-m-n+1$ and $|x_2-y_2|\leq 1$, leads to a contradiction as $m>2$.
  
Suppose  that  $(\tilde b_{\rrr{3}})_2 <  0$. Then $(\tilde b_{\rrr{4}})_2 \in (\tilde b_{\rrr{3}})_2 + [m, n-m]$. 
  By construction, if $x\in R_{\rrr{4}}$ with $x_1 \geq 4N-m$, then $x\in \tilde b_{\rrr{3}}+T_{\rrr{3}}(m,n)$ and therefore $ x_2 \leq 
 (\tilde b_{\rrr{3}})_2 +n$. 
 Take  $y\in \tilde b_{\rrr{4}}+\bigl( \hat T_3(n)\cup \hat T_3(m,n) \bigr)$. Then $y_1 \geq  (\tilde b_{\rrr{4}})_1=4N$ and   $y_2 \geq  (\tilde b_{\rrr{4}})_2 +n-1 \geq  
 (\tilde b_{\rrr{3}})_2+m  + n-1$.
  Suppose by contradiction  that $y\in R_{\rrr{4}} \cup \partial R_{\rrr{4}}$. Then there exists $x\in R_{\rrr{4}}$ such that $|x-y|<1$.  This implies that $x_1 \geq y_1-1\geq 4N-1$. Hence, by the initial observations, $x_2 \leq  (\tilde b_{\rrr{3}})_2 +n $.  This last bound, together with 
   $y_2 \geq 
 (\tilde b_{\rrr{3}})_2+m  + n-1$ and $|x_2-y_2|\leq 1$, leads to a contradiction as $m>2$.


 \section{Derivation of   \eqref{problema2d}  from the bound $\bbQ(N_M\geq c _1 M) \geq 1- e^{-c_2 M}$
 in Section \ref{sec_ginepro}}\label{app_ultimatum}
In this appendix  we explain  in detail how to derive    \eqref{problema2d} with $k:=4N$ from the bound $\bbQ(N_M\geq c _1 M) \geq 1- e^{-c_2 M}$
obtained in Section \ref{sec_ginepro}.  Below $c_3,c_4,..$ will be positive constant, independent from $M$ and $L$.  Since we are interested in $L$  large enough, without loss we think of $M$ as larger than $3$. To help in following the arguments below, we have provided Figure \ref{traguardo}.

\begin{figure}
\includegraphics[scale=0.25]{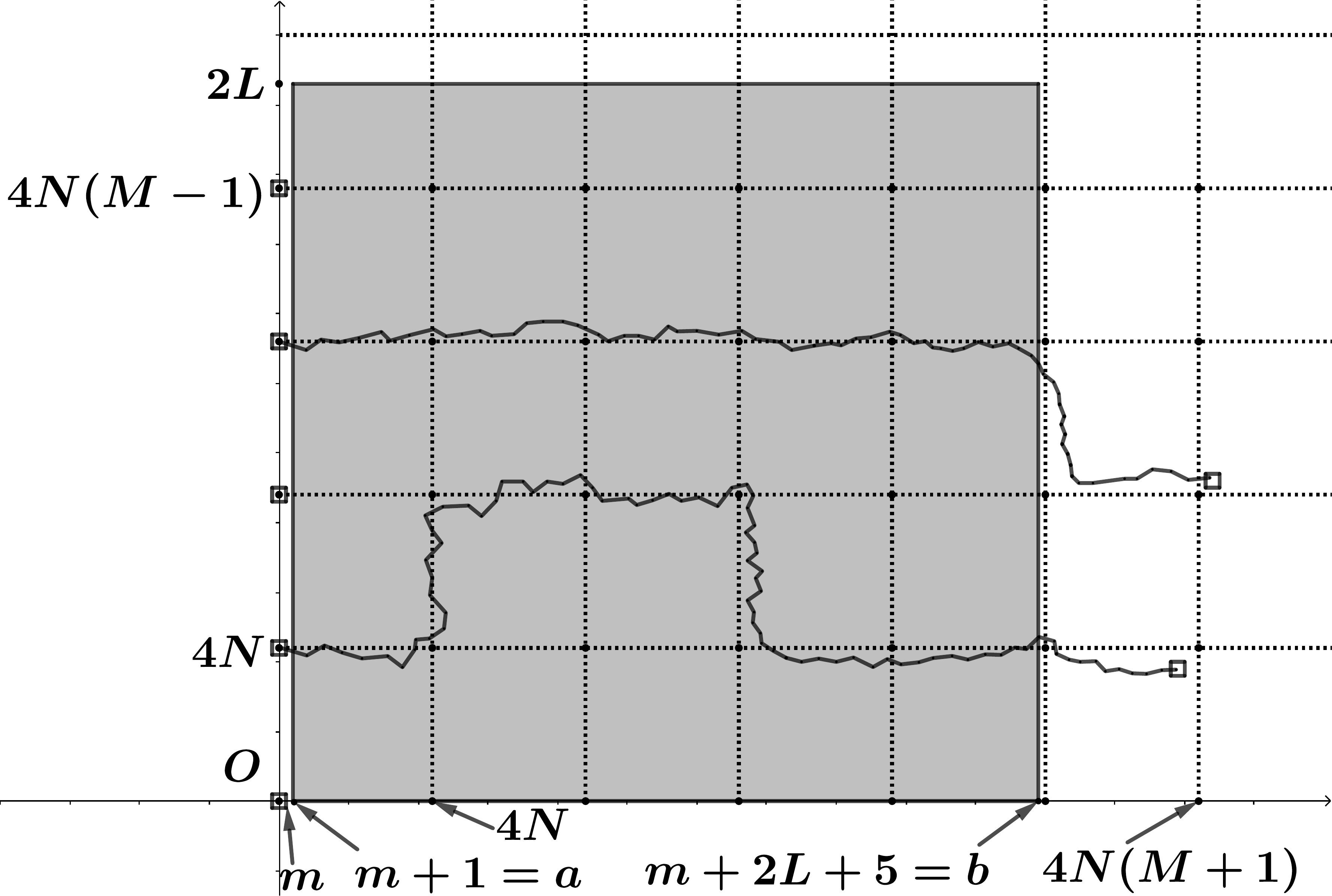}
\captionsetup{width=.9\linewidth}
\caption{$M=5$, $d=2$. Black circles are points in 
$ \{ 4N \bar x\,:x \in \L\}\subset 4 N \bbZ^d$. The grey box represent the slice $S$. Paths   are  in $\bbG_+$   }
\label{traguardo}
\end{figure}

Due to the  translation invariance of $\bbP$, we can traslate all geometrical objects by the vector $(L+3+m)e_1 + L e_2$. In particular, instead of the slice \eqref{rondine} we will consider the  translated one \be\label{slime}
S:= [a,b] \times [0,2L] \times [-4N, 4N)^{d-2}\,, \qquad a:= m+1\,,\; b:=2L+5+m \,. 
\en
Recall that $ \L :=( [0,M+1] \times [0, M-1] )\cap \bbZ^2$  and recall the notation \eqref{deep}. The map $\L \ni x \mapsto 4 N \bar{x} \in 4N \bbZ^d$ is our  natural immersion of $\L$ into  the rescaled lattice $4N \bbZ^d$. The right vertical face of $\L$ is then mapped inside the set $\{ x\in \bbZ^d\,:\, x_1= 4N (M+1)\}$. Recall that $M$ is the minimal integer such that  \be\label{underwater}
4N (M+1) > b+N\,.\en

Without loss of generality, when referring to the LR crossings of the box $\L$ for $\z$ we restrict   to crossings such that only the first and the last points intersect the vertical faces of $\L$ (which would not change the random number $N_M$).
We fix a set $\G'$ of vertex--disjoint LR crossings  of $\L$ for $\z$  with cardinality $N_M$. Then we define $\G$ as the set of paths $(x_1,x_2, \dots, x_k)$ in $\G'$ such that $x_i$ has second coordinate in $[1,M-2]$  for each $i$ (we stress that the index $i$ here labels points and not coordinates).
Note that, since $\L$ is bidimensional, $|\G|\geq |\G'|-2$. 

%

 Take a LR crossing $(x_1,x_2, \dots, x_k)$ in $\G$.   By the discussion in 
 Section \ref{sec_ginepro},
    we get that there is a path $\g$ in $\rrr{\bbG_+}$ from $B(4N \bar x_1,m)$ to $B(4 N \bar x_k,N)$  without self-intersections. Moreover,  this path is included in the  region  $\cR$ obtained as union of the boxes $B(4N  \bar{x}_i , \rrr{2N})$ with $i=1,\dots, k$ \rrr{(see Fig.~\ref{aquile})}. 
    \rrr{As the second coordinate of any point $\bar{x}_i $ varies in $[1,M-2]$,}  the second coordinate of  any point in  $    B(4N  \bar{x}_i ,  \rrr{2N})$   is in \rrr{$[4N-2N, 4N(M-2)+2N] \subset [0,2L]$. Note that the last inclusion follows from the bound      $4NM\leq b+N=2L + 5+m+N$ due to 
 the minimality of $M$ in \eqref{underwater}.  }
  In addition,  the box  $ B(4N \bar x_1, m)$ lies in  the halfspace $\{(z_1,\dots, z_d) \in \e \bbZ^d\,: z_1 \leq m \}$, while the box $ B(4N \bar x_k, N)$  lies  in the halfspace $\{(z_1,\dots, z_d) \in \e \bbZ^d\,: z_1 \rrr{\geq b+1}\}$ \rrr{as   \eqref{underwater} implies that $4N (M+1) \geq b+N+1$}. As a consequence we can extract from the above path $\g$  a new  path $\tilde \g$  for  $\rrr{\bbG_+}$  lying in
\rrr{$\cR \cap   \{(z_1,\dots, z_d) \in \e \bbZ^d\,: 0\leq z_2 \leq 2L  
\} \subset  \e \bbZ \times [0,2L] \times [-4N,4N)^{d-2}$, whose first and last points have 
  first  coordinate  $\leq m$ and $\geq  b+1$, respectively.} 
  
  \rrr{At cost to further refine $\tilde \g $, we have that $\tilde \g$ has the following property $\cP$: $\tilde \g$ is  a LR crossing  of the box $[a,b]\times [0,2L]\times [-L,L]^{d-2}$  (cf. Definition 
  \ref{def_LR_bis}) whose vertexes, \rui{apart from} the first and last one, are contained in $S= [a,b] \times [0,2L] \times [-4N, 4N)^{d-2}$, while the first and last one are contained, respectively, in $(-\infty,a)\times [0,2L]\times [-4N, 4N)^{d-2}$ and $(b,+\infty) \times [0,2L]\times [-4N, 4N)^{d-2}$. Note that the above box is the translated version of the box  $\D_L$ appearing in Proposition \ref{carletto} by the vector $(L+3+m)e_1+Le_2$, and that     one has to translate all geometrical objects by the same vector when applying Definition 
  \ref{def_LR_bis}.
  }
  
    Moreover, due to the \rrr{bidimensionality of $\L$},  there is an integer $\ell$ \rrr{(independent from $N,M,L$)}  such that  every path  $\tilde \g$   can share some vertex with at most $\ell$ paths $\tilde \g'$ \rrr{built in a similar way}. 
 Since $M\geq  \rrr{c_3} L$ \rrr{for some $c_3>0$}, by the above observations we have proved that the event $\{ N_M \geq \rrr{c_1} M\}$ implies the event $F$ that  \rrr{there are  at least $ c_4 L$ vertex--disjoint LR crossings for $\rrr{\bbG_+}$ with the above property $\cP$}.  Hence, by our bound 
\rrr{$\bbQ(N_M\geq c _1 M) \geq 1- e^{-c_2 M}$}
   and since $M \leq \rrr{c_5} L$, we get  that  $\bbQ(F) \geq 1- e^{- \rrr{c_2} L}$. 
Since edges in $\rrr{\bbG_+}$ have length \rrr{strictly smaller than $1$}, the event $F$ does not depend on the vertexes of $\rrr{\bbG_+}$ in $\cup_{s=0}^{M-1} B(4N \bar x_1^s, m)$, neither on the edges exiting from the above region. \rrr{Indeed, by Definition \ref{def_LR_bis}, to check property $\cP$ it would be enough to know  the graph $\bbG_+$ (vertexes and edges)  inside the region 
  $\{(z_1,\dots, z_d ) \in \ezd \,:\,z_1>a-1=m\}$}. 
 In particular, \rrr{the event $F$ and  the event $D$ defined by \eqref{dedalo}}  are independent, thus implying that
$\bbP(F)= \bbP( F|\rrr{D}) = \bbQ(F) \geq  1- e^{- \rrr{c_2} L}$, and in particular \eqref{problema2d} is verified.

\bigskip
\bigskip

{\bf Acknowledgements}. The authors thank Davide Gabrielli and Vittoria Silvestri for useful discussions. They also thank an anonymous referee for valuable  comments improving the presentation  of the results. 

\end{document}